\documentclass[11pt,reqno]{amsart}
\usepackage{tikz}
\usepackage{subfig}
\usepackage{epstopdf}
\usepackage{epsfig}
\usepackage{hyperref}

\newcommand{\qi}[1]{{  #1}}

\usepackage{math}
\graphicspath{{./figures/}}
\usepackage{tikz} 
\usetikzlibrary{shapes,arrows}
\tikzstyle{decision} = [diamond, draw, fill=blue!20, 
    text width=4.5em, text badly centered, node distance=3cm, inner sep=0pt]
\tikzstyle{block} = [rectangle, draw, fill=blue!20, 
    text width=5em, text centered, rounded corners, minimum height=4em]
\tikzstyle{line} = [draw, -latex']
\tikzstyle{cloud} = [draw, ellipse,fill=red!20, node distance=3cm,
    minimum height=2em]
\usetikzlibrary{positioning}
\tikzset{main node/.style={circle,fill=blue!20,draw,minimum size=1cm,inner sep=0pt},  }

\newcommand{\la}{\langle}
\newcommand{\ra}{\rangle}

\newcommand{\ts}{\mathsf{T}}

\newcommand{\ba}{\begin{array}}
\newcommand{\ea}{\end{array}}
\newcommand{\be}{\begin{equation}}
\newcommand{\ee}{\end{equation}}
\newcommand{\bea}{\begin{eqnarray}}
\newcommand{\eea}{\end{eqnarray}}
\newcommand{\beaa}{\begin{eqnarray*}}
\newcommand{\eeaa}{\end{eqnarray*}}

\def\div{\mathbf{div}}

\def\hR{\mathbb{R}}

\def\q{\quad}
\def\qq{\qquad}

\def\pa{\partial}

\def\qed{ \hfill \vrule width.25cm height.25cm depth0cm\smallskip}

\newcommand{\basa}{\begin{assumption}}
\newcommand{\easa}{\end{assumption}}

\newcommand{\bas}{\begin{assum}}
\newcommand{\eas}{\end{assum}}

\def\1{{\bf 1}}

\def\:{\!:\!}

at 9pt

\begin{document}
%
%

\title[Hypoelliptic Entropy dissipation]{Hypoelliptic Entropy dissipation for stochastic differential equations}
\author{Qi Feng$^*$}
\author{Wuchen Li$^\dagger$}
\thanks{$^*$ Department of Mathematics, University of Michigan, Ann Arbor, MI 48109. Email: qif@umich.edu.}
\thanks{$^\dagger$Department of Mathematics, University of South Carolina,  Columbia, SC 29208. Email:  wuchen@mailbox.sc.edu.}
\thanks{Qi Feng is partially supported by the National Science Foundation under grant DMS-2306769. Wuchen Li is supported by AFOSR MURI FP 9550-18-1-502, AFOSR YIP award No. FA9550-23-1-0087, and NSF RTG: 2038080.}
\thanks{}
\begin{abstract}
We study the convergence analysis for {general} degenerate and non-reversible stochastic differential equations (SDEs). We apply the Lyapunov method to {analyze} the Fokker-Planck equation, in which the Lyapunov functional is chosen as a weighted relative Fisher information functional. We derive a structure condition and formulate the Lyapunov constant explicitly. We prove the exponential convergence result for the probability density function towards its invariant distribution in the $L_1$ {distance}. Two examples are presented: underdamped Langevin dynamics with variable diffusion matrices and three oscillator chain models with nearest-neighbor couplings.  
\end{abstract}
\maketitle
\section{Introduction}
Consider an It\^o stochastic differential equation (SDE)  
 \begin{equation}\label{a}
     dX_t=b(X_t)dt+\sqrt{2}a(X_t)dB_t,
 \end{equation}
where $X_t\in\mathbb{R}^{n+m}$ is a stochastic process with dimensions $n\in\mathbb{Z}_+$, {$m\in\mathbb{Z}_{\ge0}$}, $b\in \qi{C^{\infty}(\mathbb{R}^{n+m}; \mathbb{R}^{n+m})  }$ is a smooth drift vector field, $a\in \qi{C^{\infty}( \mathbb R^{n+m};\mathbb{R}^{(n+m)\times n}) }$ is a smooth diffusion matrix, and $B_t$ is a standard Brownian motion in $\mathbb{R}^{n}$.  

We study the long-time {dynamical behavior} of SDE \eqref{a}. How fast does the probability density function of SDE \eqref{a} converge to its invariant distribution? The convergence behavior of SDE \eqref{a} plays an essential role in both theory and applications. Theoretically, it is a core problem in probability \cite{BE, Gross}, non-equilibrium dynamics \cite{C1, HairerMattingly, IOS2019, Nelson, Villani2006_hypocoercivity}, differential geometry \cite{FB, BaudoinGarofalo09} and ergodic theory \cite{BE2008, LY, MSH, Luc}. In applications, the convergence rate is useful for studying long-time molecular dynamics behaviors \cite{CCBJ, Leimuhler, Sachs}. Nowadays, the convergence rate is important in estimating the speed of Markov-Chain-Monte-Carlo methods (MCMC). This is an emerging issue in artificial intelligence (AI), and Bayesian sampling algorithms \cite{LuWang, CCBJ, Duncan, Sachs}. 

In \qi{the} current framework, 
when the diffusion matrix $a$ is non-degenerate, and the drift vector field $b$ is a gradient \qi{vector field}, SDE \eqref{a} is known as a reversible diffusion process. Its convergence rate has been studied \qi{by using various methods}. A typical way is the entropy method; see \cite{AMTU}. In deriving the entropy method, a method is known as \qi{the} Bakry-{\'E}mery Gamma calculus. \qi{By applying this method}, one can derive the convergence rate of SDE \eqref{a} by \qi{using the} Bochner's formula. However, Bochner's type formula often fails when \qi{the diffusion} matrix $a$ is a degenerate matrix. In other words, if SDE \eqref{a} is a degenerate diffusion process with a possible irreversible drift vector field, the existing entropy method (e.g., \cite{Villani2006_hypocoercivity}) {may} not be applied directly.  

In this paper, we derive \qi{the} convergence analysis for SDE \eqref{a} by \qi{using} a modified entropy dissipation method. We design a weighted relative Fisher information functional as a Lyapunov functional. By studying the dissipation of \qi{the} weighted Fisher information functional, we introduce a structure condition to derive an ``information Bochner type formula''; see Assumption \ref{prop:main condition}. We formulate algebraic tensors, a.k.a. {the} modified Hessian matrices or {the} ``Ricci curvature type tensors''; see details in Theorem \ref{thm: information bochner}. {Here, the modified Hessian matrix represents the generalized Hessian operator of negative logarithms of invariant density.} The smallest eigenvalue of the modified Hessian matrix characterizes the convergence behavior of SDE \eqref{a} in terms of $L_1$ distances. In examples, we present explicit convergence rates for variable diffusion coefficient underdamped Langevin dynamics and three oscillator chain models. 

Many studies exist for convergence behaviors of degenerate SDEs in \cite{AM2019, BFM2020, CCBJ, EckmannHairer, HerauNier, IOS2019, Gadat, LucThomas,LuWang, MN, Talay, Villani2006_hypocoercivity}. Previous works have established the convergence rates under various metrics, e.g., $H^{1} $ or $L_2$ distances. In particular, \cite{AE} has applied a modified entropy dissipation method to constant {$a$} degenerate diffusion processes. Besides, the modified Gamma calculus has been invented in \cite{BaudoinGarofalo09}. It has been applied to study underdamped Langevin dynamics \cite{FB, BGH2019}. This method establishes the convergence in the $H^1$ distance, {and it requires} an iterative symmetry condition between \qi{$a$ and $z$, which will be discussed in the next section}. Compared to the previous results, we derive a structure condition in Assumption \ref{prop:main condition}, \qi{which does not require the iterative symmetry condition}. We {next} derive \qi{a new modified Hessian matrix} which can handle non-gradient drift vector fields and degenerate diffusion matrices \qi{at the same time}. Technically speaking, we apply the second-order calculus in probability density space embedded with an optimal transport type metric \cite{Lafferty, Li2018_geometrya, Villani2009_optimal}. 
{This} computation extends the second-order derivative of {Kullback-Leibler} divergence in sub-Riemannian density manifold \cite{FL, FL1}. 

We organize the paper as follows. In Section \ref{sec2}, we present the main result of this paper. We derive a modified Hessian matrix, whose smallest eigenvalue characterizes the convergence behavior of SDE \eqref{a}. We present two examples \qi{for} the main theorem:  underdamped Langevin dynamics with variable diffusion matrices in Section \ref{sec4}, and three oscillator chain models in Section \ref{sec: anharmonic}. The proofs of the main theorem are presented in Section \ref{sec3:1} with an ``information Bochner type formula''. The detailed computations of all examples are given in the Appendix.   

\section{Main result}\label{sec2}
In this section, we present the main result of this paper. 
\subsection{Settings} \label{sec2.1}
For a diffusion matrix $a(x)\in C^{\infty}(\mathbb R^{n+m}; \mathbb R^{(n+m)\times n} )$, 
we denote $n$ as the rank of $a(x)$. We denote $a(x)^{\ts}$ as the transpose of matrix $a(x)$, and $a(x)a(x)^{\ts}$ as the standard matrix multiplication. For $i=1,\cdots, n$,  we denote $a^{\ts}_i=(a(x)^{\ts})_i$ as the row vectors of $a(x)^{\ts}$, and $a_{\cdot i}=a(x)_{\cdot i}$ as the column vectors of $a(x)$, i.e. $a^{\ts}_{i\hat i}=a_{\hat ii}$, for $\hat i=1,\cdots, n+m$. Furthermore, for each row vector  $a^{\ts}_i\in \mathbb R^{n+m}$ with $i=1,\cdots,n$, we denote $\mathbf A_i(x):=\sum_{\hat i=1}^{n+m}a^{\ts}_{i\hat i}\frac{\partial}{\partial x_{\hat i}}$ as the corresponding {vector field} for each row vector {$a^{\ts}_i=\sum_{\hat i=1}^{n+m}a^{\ts}_{i\hat i}e_i$. Here, we denote $\{e_i\}_{i=1}^{n+m}$ and $\{\frac{\partial}{\partial x_i }\}_{i=1}^{n+m}$ as the canonical basis for the Euclidean space $\mathbb R^{n+m}$ and its tangent space $T\mathbb R^{n+m}$, respectively}. Similarly, we denote $\mathbf A_0(x):=\sum_{\hat i=1}^{n+m}b_{\hat i}(x)\frac{\partial}{\partial x_{\hat i}}$ as the vector field associated to the drift term $b$. For SDE \eqref{a}, we assume that $b\colon \mathbb{R}^{n+m}\rightarrow\mathbb{R}^{n+m}$ and $a\colon \mathbb{R}^{n+m}\rightarrow \mathbb{R}^{(n+m)\times n}$ are smooth Lipschitz functions. We also assume that the collections of vector fields $\{\mathbf A_0(x),\mathbf A_1(x),\cdots,\mathbf A_n(x)\}$ satisfy the following weak H\"ormander condition \cite{Bismut}:
\begin{equation*}\qi{ 
    \mathrm{Span}\Big\{\mathbf A_1(x),\cdots, \mathbf A_n(x),~ [\mathbf A_{i_1},\cdots,[\mathbf A_{i_{k-1}},\mathbf A_{i_k}]\cdots](x),~0\leq i_1,\cdots,i_k\leq n, k\geq 2 \Big\}=\hR^{n+m},}
\end{equation*}
where $[\cdot , \cdot ]$ represents \qi{the} Lie bracket between two vector fields. For $i,j=0,1,\cdots, n$, we have \qi{ $[\mathbf A_i(x),\mathbf A_j(x)]= \sum_{\hat i,\hat j=1}^{n+m}\Big[  a^{\ts}_{i\hat i}\frac{\partial}{\partial x_{\hat i}}\Big(  a^{\ts}_{j\hat j}\frac{\partial}{\partial x_{\hat j}} \Big)-a^{\ts}_{j\hat j}\frac{\partial}{\partial x_{\hat j}}\Big(a^{\ts}_{i\hat i}\frac{\partial}{\partial x_{\hat i}} \Big) \Big]$}.  The weak H\"ormander condition means that the Lie algebra generated by the vector fields $\{ \mathbf A_0,\mathbf A_1,\cdots,\mathbf A_n\}$ is of full rank at every point $x\in\hR^{n+m}$. This condition \cite{Bismut,Hormander} ensures the existence of a smooth probability density function of SDE \eqref{a}. Under the above assumptions, the Fokker-Planck equation of SDE \eqref{a} satisfies  
\begin{equation}\label{FPE}
\partial_tp(t,x)=-\nabla\cdot(p(t,x)b(t,x))+\sum_{i=1}^{n+m}\sum_{j=1}^{n+m}\frac{\partial^2}{\partial x_i\partial x_j} \Big(\Big(a(x)a(x)^{\ts}\Big)_{ij}p(t,x)\Big),
\end{equation}
where $p(t,x)$ is the probability density function of SDE \eqref{a} with {an initial condition satisfying} 
\begin{equation*}
p_0(x)=p(0,x), \quad \int p_0(x) dx=1,\quad p_0(x)\geq 0.    
\end{equation*}

Here $\int$ denotes the integration over the entire spatial domain $\mathbb{R}^{n+m}$. Under our assumptions, $p(t,x)$ is smooth w.r.t both time and spatial variables $t$, $x$, respectively, and stays positive for all positive time $t\in (0,\infty)$. We assume that the SDE \eqref{a} has a unique invariant 
density $\pi$, which is strictly positive everywhere. In fact, $\pi$ is the equilibrium of equation \eqref{FPE}: 
\begin{equation*}
    -\nabla_x\cdot(\pi(x)b(x))+\sum_{i=1}^{n+m}\sum_{j=1}^{n+m}\frac{\partial^2}{\partial x_i\partial x_j}\Big(\Big(a(x)a(x)^{\ts}\Big)_{ij} \pi(t,x)\Big)=0. 
\end{equation*}
We also assume that $\pi\in C^{2}(\mathbb{R}^{n+m}; \mathbb{R})$, which has an \qi{explicit} formula. This paper studies the convergence behavior of $p(t,x)$ towards the invariant density function $\pi(x)$. 

The convergence result is based on the following decomposition/reformulation of the Fokker-Planck equation \eqref{FPE}.
\begin{proposition}[Decomposition]
Denote a vector field $\gamma\colon \mathbb{R}^{n+m}\rightarrow\mathbb{R}^{n+m}$ satisfying 
 \begin{equation*}\qi{ 
\gamma(x):=  \Big( a(x)a(x)^{\ts}\Big)\nabla\log \pi(x)-b(x)+\Big(\sum_{j=1}^{n+m}\frac{\partial}{\partial x_j}\Big(a(x)a(x)^{\ts}\Big)_{ij}\Big)_{i=1}^{n+m}.}
 \end{equation*}
Then equation \eqref{FPE}  \qi{is equivalent to the following equation}:
\begin{equation}\label{FPE1}
\begin{aligned}\qi{ 
\pa_t p(t,x)=\nabla\cdot \Big(p(t,x)\Big(a(x)a(x)^{\ts}\Big)\nabla \log \frac{p(t,x)}{\pi(x)}\Big)+\nabla\cdot(p(t,x)\gamma(x)).}
\end{aligned}
\end{equation}
In addition, 
\begin{equation*}
 \nabla\cdot (\pi(x)\gamma(x))=0. 
\end{equation*}
\end{proposition}
\begin{proof}
The proof follows by a direct calculation. 

(i) We note that 
\begin{equation*}
\begin{aligned}
   &\sum_{i=1}^{n+m}\sum_{j=1}^{n+m}\frac{\partial^2}{\partial x_i\partial x_j} \Big(\Big(a(x)a(x)^{\ts}\Big)_{ij}p(t,x)\Big)
      =\sum_{i=1}^{n+m}\frac{\partial}{\partial x_i}\sum_{j=1}^{n+m}\frac{\partial}{\partial x_j}\Big(\Big(a(x)a(x)^{\ts}\Big)_{ij}p(t,x)\Big)\\
            =&\sum_{i=1}^{n+m}\frac{\partial}{\partial x_i}\sum_{j=1}^{n+m}\Big(\frac{\partial}{\partial x_j}\Big(a(x)a(x)^{\ts}\Big)_{ij}p(t,x)+\Big(a(x)a(x)^{\ts}\Big)_{ij}\frac{\partial}{\partial x_j}p(t,x)\Big)\\
 =&\sum_{i=1}^{n+m}\frac{\partial}{\partial x_i}\Big(p(t,x)\frac{\partial}{\partial x_j}\sum_{j=1}^{n+m}\Big(a(x)a(x)^{\ts}\Big)_{ij}\Big)+\sum_{i,j=1}^{n+m}\frac{\partial}{\partial x_i}\Big(\Big(a(x)a(x)^{\ts}\Big)_{ij}\frac{\partial}{\partial x_j}p(t,x)\Big)\\
 =&\sum_{i=1}^{n+m}\frac{\partial}{\partial x_i}\Big(p(t,x)\frac{\partial}{\partial x_j}\sum_{j=1}^{n+m}\Big(a(x)a(x)^{\ts}\Big)_{ij}\Big)+\sum_{i,j=1}^{n+m}\frac{\partial}{\partial x_i}\Big(\Big(a(x)a(x)^{\ts}\Big)_{ij}p(t,x)\frac{\partial}{\partial x_j}\log p(t,x)\Big),
\end{aligned}
\end{equation*}
where we use the fact $\frac{\partial}{\partial x_j}p(t,x)=p(t,x)\frac{\partial}{\partial x_j}\log p(t,x)$ in the last equality. \qi{For simplicity of notations, we skip the variables $(t,x)$ below.} 
 Using the definition of $\gamma$ and the above observation, we show that the R.H.S. of Fokker-Planck equation \eqref{FPE} is equivalent to the following representation, 
\begin{equation*}
   \begin{aligned}
&-\nabla\cdot(pb)+\sum_{i=1}^{n+m}\sum_{j=1}^{n+m}\frac{\partial^2}{\partial x_i\partial x_j} \Big((aa^{\ts})_{ij}p\Big)\\
=&-\nabla\cdot(p b)+\sum_{i,j=1}^{n+m}\frac{\partial}{\partial x_i}\Big(p\frac{\partial}{\partial x_j}(aa^{\ts})_{ij}\Big)+\nabla\cdot(p(aa^{\ts})\nabla\log p)\\
=&-\nabla\cdot(p b)+\sum_{i,j=1}^{n+m}\frac{\partial}{\partial x_i}\cdot\Big(p\frac{\partial}{\partial x_j}(aa^{\ts})_{ij}\Big)+\nabla\cdot(p(aa^{\ts})\nabla\log\pi)\\
&-\nabla\cdot(p(aa^{\ts})\nabla\log\pi)+\nabla\cdot(p(aa^{\ts})\nabla\log p)\\
=&\nabla\cdot\Big(p (-b+{\Big[\sum_{j=1}^{n+m}}\frac{\partial}{\partial x_j}\Big(aa^{\ts}\Big)_{ij}\Big]_{i=1}^{n+m}+aa^{\ts}\nabla\log\pi)\Big) +\nabla\cdot(paa^{\ts}\nabla\log\frac{p}{\pi})\\
=&\nabla\cdot(p \gamma) +\nabla\cdot\Big(p(aa^{\ts})\nabla\log\frac{p}{\pi}\Big),
\end{aligned}
\end{equation*}
where we use the definition of $\gamma$ and the fact that $\nabla\log\frac{p}{\pi}=\nabla\log p-\nabla\log \pi$. This finishes the first part of the proof. 

(ii) Since $\pi$ is the equilibrium of equation \eqref{FPE}, {then} 
\begin{equation*}
   \Big(\nabla\cdot(p\gamma)+\nabla\cdot\Big(paa^{\ts}\nabla\log\frac{p}{\pi}\Big)\Big)\Big|_{p=\pi}=\nabla\cdot(\pi \gamma)=0,  
\end{equation*}
where we use the fact \qi{that} $\log \frac{\pi}{\pi}=\log 1=0$. This finishes the proof.  \qed 
\end{proof}


We observe that the term $\nabla\cdot(p aa^{\ts}\nabla\log\frac{p}{\pi})$ corresponds to the ``reversible'' (gradient) and ``hypoelliptic'' (degeneracy) component of SDE \eqref{a}'s Kolmogorov forward operator. In contrast, the term $\nabla\cdot(p\gamma)$ represents the ``irreversible'' (non-gradient) component in the Fokker-Planck equation \eqref{FPE}. In later proofs, we develop a Lyapunov analysis to study the dynamical behaviors of the SDE \eqref{a}. Using the decomposition \eqref{FPE1}, we derive exponential convergence results for equation \eqref{FPE}, where the convergence rate addresses both degenerate and irreversible components of equation \eqref{FPE}. 

\subsection{Structure condition and Hessian matrix}
 For a \qi{diffusion matrix} 
\bea\label{defn: a} \qi{
a\in C^{\infty}(\mathbb R^{n+m}
;\mathbb R^{(n+m)\times n}),}
\eea 
associated with SDE \eqref{a} with rank $n$, we introduce a complementary matrix, {defined by}, 
\bea\label{defn: z} 
z\in C^{\infty}(\mathbb R^{n+m}; \mathbb R^{(n+m)\times m}),
\eea  
such that 
\bea \label{full rank condition}
\text{Rank}\Big(a(x)a(x)^{\ts}+z(x)z(x)^{\ts}\Big) =n+m,\quad \text{for all} ~~x\in \hR^{n+m}.
\eea 
Recall all notations introduced in Section \ref{sec2.1}, we denote $a^{\ts}$ and $z^{\ts}$ as the transpose of  matrices $a$ and $z$, and write $\{ a^{\ts}_i\}_{i=1}^n$ and $\{ z^{\ts}_j\}_{j=1}^m$ as the row vectors of $a^{\ts}$ and $z^{\ts}$. The condition \eqref{full rank condition} means that the span of the vector fields associated with the row vectors $\{ a^{\ts}_i\}_{i=1}^n$ and $\{ z^{\ts}_j\}_{j=1}^m$ generates the whole space $\hR^{n+m}$. Throughout this paper, we denote 
\beaa 
\mathsf{U}(x)\in\mathcal C^{\infty}(\hR^{n+m};
\hR^{n+m})
\eeaa 
 as a smooth vector field. The vector field $\mathsf U$ {can be chosen as} 
\beaa 
\mathsf U=(\mathsf U_1,\cdots, \mathsf U_{n+m})^{\ts} =\nabla f(x)= \Big(\frac{\partial f}{\partial x_1,}\cdots, \frac{\partial f}{\partial x_{n+m}} \Big)^{\ts},
\eeaa  
for a smooth function $f\in C^{\infty}(\hR^{n+m})$. We will keep the convention of viewing {gradient operator $\nabla$} and vector field $\mathsf U$ as column vectors in matrix multiplication below. 
{We define $\la \cdot, \cdot\ra $ as the inner product for two vectors in the Euclidean space for any dimension.} 
We keep the following notation throughout the paper, 
we denote $a^{\ts}\mathsf U=( a^{\ts}_1\mathsf U, \cdots, a^{\ts}_n\mathsf U)^{\ts}$, and   
\bea \label{notation: nabla a} a^{\ts}_ia^{\ts}_k\nabla^2=\sum_{\hat i,\hat k=1}^{n+m}a^{\ts}_{i\hat i}a^{\ts}_{k\hat k}\frac{\partial^2}{\partial x_{\hat i}\partial x_{\hat k}}, \quad \text{and}\quad a^{\ts}\nabla^2 a = \Big(\sum_{\hat i,\hat j =1}^{n+m} a^{\ts}_{i\hat i} \frac{\partial}{\partial x_{\hat i}} \frac{\partial}{\partial x_{\hat j}} a_{\hat j j}\Big)_{i,j=1}^n.
\eea 
Next, we denote $\circ$ as a {product operation between a matrix and a vector as below},
\bea
 a^{\ts}\nabla^2 a\circ (a^{\ts}\mathsf U)=\Big(a^{\ts}_{i\hat i}\frac{\partial^2 a_{\hat j j}}{\partial x_{\hat i}\partial x_{\hat j}} a^{\ts}_j\mathsf U \Big)_{i=1}^n.
\eea 
{For notation convenience, we denote $\nabla_k=\frac{\partial}{\partial x_k}$, and write $\nabla a:=\nabla^{\ts}a= \Big(\sum_{\hat i=1}^{n+m}\frac{\partial}{\partial x_{\hat i}} a_{\hat i i}\Big)_{i=1}^n$}. The same convention applies to matrix $z^{\ts}$ as above.
We present two assumptions for our main results. The first assumption is about the {relationship of the $n+m$ vector fields}
\bea \label{defn: vector fields A Z}
\Big\{\sum_{\hat i=1}^{n+m}a^{\ts}_{i\hat i}\frac{\partial}{\partial x_{\hat i}}\Big\}_{i=1}^n=: \{\mathbf A_i \}_{i=1}^n,\quad  \text{and}\quad  \Big\{\sum_{k'=1}^{n+m}z^{\ts}_{kk'}\frac{\partial}{\partial x_{k'}} \Big\}_{k=1}^m=:\{\mathbf Z_k\}_{k=1}^m,
\eea  
associated to the row vectors of $a^{\ts}$  and $z^{\ts}$.
\begin{assum}\label{prop:main condition}\qi{ For a given matrix $a$ in \eqref{defn: a}, there exists a matrix $z$ in \eqref{defn: z}, such that condition \eqref{full rank condition} holds true. Furthermore, for any $i\in \{1,\cdots,n\}$ and $k\in\{1,\cdots,m\}$, we assume
\bea 
\sum_{\hat i=1}^{n+m} \mathbf Z_k a^{\ts}_{i\hat i} \frac{\partial}{\partial x_{\hat i}}
\in \mathrm{Span}\{\mathbf A_1,\cdots,\mathbf A_n\}.
\eea 
In other words, we assume that the following relation holds true
\bea\label{main condition} 
\sum_{\hat i=1}^{n+m} \mathbf Z_k a^{\ts}_{i\hat i} \frac{\partial}{\partial x_{\hat i}}=  \sum_{\hat i=1}^{n+m}\Big(\sum_{k'=1}^{n+m}z^{\ts}_{kk'}\frac{\partial a^{\ts}_{i\hat i}}{\partial x_{k'}} \Big)\frac{\partial}{\partial x_{\hat i}}=\sum_{k'=1}^{n+m}\sum_{l=1}^nz^{\ts}_{kk'} \lambda^{k'i}_{l}\mathbf A_l,
\eea
for a sequence of functions $\lambda^{k'i}_{l}$ with $i, l\in\{1,\cdots,n\}$ and  $k'\in\{1, \cdots, n+m\}$.}
\end{assum}

\begin{remark}[Structure conditions]
Assumption \ref{prop:main condition} can be easily verified for a class of degenerate SDEs. Two examples are given below. 
\begin{itemize}
\item[(i)] \qi{ If $a$ is a constant matrix, then according to \eqref{notation: nabla a}, $\nabla a$ is a $n$-dimensional zero vector, it is obvious that Assumption \ref{prop:main condition} is satisfied with all $\lambda_{\cdot}^{\cdot\cdot} \equiv 0$;}
\item[(ii)] \qi{ If the matrix $a(x)\in C^{\infty}(\hR^{n+m};\hR^{(n+m)\times n})$ only depends on the first {$n$} variables, i.e., 
\beaa 
a(x)=a(x_1,\cdots,x_n),
\eeaa 
and the vector fields $\{\mathbf Z_k\}_{k=1}^m$ associated with matrix $z \in C^{\infty}(\hR^{n+m};\hR^{(n+m)\times m})$ as defined in \eqref{defn: vector fields A Z} {satisfy} the following condition, 
\beaa 
\{\mathbf Z_k\}_{k=1}^m \in \text{Span}\Big\{\frac{\partial}{\partial x_{n+1}} ,\cdots, \frac{\partial}{\partial x_{n+m}} \Big\},
\eeaa 
where {$\frac{\partial}{\partial x_{i}}$, $i=n+1,\cdots, n+m$ denotes the $i$-th canonical basis on the tangent space $T\hR^{n+m}$}, then Assumption \ref{prop:main condition} holds true. This follows from the direct computation of \eqref{main condition}, since $\sum_{k'=1}^{n+m}z^{\ts}_{kk'}\frac{\partial}{\partial x_{k'}} a^{\ts}_{i\hat i}=\sum_{k'=n+1}^{n+m}z^{\ts}_{kk'}\frac{\partial}{\partial x_{k'}} a^{\ts}_{i\hat i}(x_1,\cdots,x_n)=0$ for all $k,i$.} 
\end{itemize}
\end{remark}

We next define the following modified Hessian {matrix $\mathfrak R(x)$}, which is a generalization {of} the Hessian of the negative logarithm of the invariant distribution (i.e., $-\log \pi$). For simplicity of presentation, we always call the modified Hessian matrix as {\em the Hessian matrix}. Some examples of Hessian matrices are presented in Proposition \ref{prop: non deg non rev a}. {The Hessian matrix essentially provides a generalized sub-Riemannian Ricci curvature tensor induced by matrix $a$. We refer to \cite{FL} for more details and its connection to geometry.}

\begin{definition}[Hessian matrix]\label{def: curvature sum}\qi{
Let matrices $a$ and $z$ satisfy Assumption \ref{prop:main condition}. We define a bilinear form associated with SDE \eqref{a}, matrices $a,z$ and a constant $\beta\in\mathbb R$ as below, for a smooth vector field $\mathsf{U}\in C^{\infty}(\hR^{n+m};\hR^{n+m})$,}  
\bea\label{defn: curvature tensor}
\mathfrak{R}(\mathsf{U},\mathsf{U})
&=& (\mathfrak R_a+\mathfrak R_z+\mathfrak R_{\pi})(\mathsf{U},\mathsf{U})-\Lambda_1^{\ts}\Lambda_1-\Lambda_2^{\ts}\Lambda_2+\mathsf D^{\ts}\mathsf D+\mathsf E^{\ts}\mathsf E\nonumber \\
&&+(\beta\mathfrak R_{\mathcal I_a}+(1-\beta) \mathfrak R_{\gamma_a}+\mathfrak R_{\gamma_z})(\mathsf{U},\mathsf{U}).
\eea 
\qi{We define $\mathfrak{R}(x): \mathbb R^{n+m}\rightarrow \mathbb{R}^{(n+m)\times (n+m)}$ as the corresponding matrix function such that }
\bea 
\mathsf U^\ts\mathfrak{R}(x)\mathsf U=\mathfrak{R}(\mathsf U, \mathsf U),
\eea 
\qi{for all vector fields  $\mathsf U$}. The bilinear forms in \eqref{defn: curvature tensor}
 are defined below.  
\bea\label{tensor a}\mathfrak R_{a}(\mathsf{U},\mathsf{U})&=& \sum_{i,k=1}^n a^{\ts}_i\nabla a^{\ts}_i\nabla  a^{\ts}_k \mathsf{U} (a^{\ts}_k\mathsf{U})+\sum_{i,k=1}^n a^{\ts}_i a^{\ts}_i\nabla^2a^{\ts}_k \mathsf{U} (a^{\ts}_k\mathsf{U})\nonumber  \\
	&& -\sum_{i,k=1}^n a^{\ts}_k\nabla a^{\ts}_i\nabla a^{\ts}_i\mathsf{U}(a_k^{\ts}\mathsf{U})-\sum_{i,k=1}^n a^{\ts}_k a^{\ts}_i \nabla^2a^{\ts}_i\mathsf{U}(a_k^{\ts}\mathsf{U})\nonumber \\
&&+\sum_{i=1}^n\sum_{\hat k=1}^{n+m}\Big[ (aa^{\ts} \nabla\log \pi)_{\hat k} \nabla_{\hat k} a^{\ts}_i\mathsf{U}-a_i^{\ts}\nabla (aa^{\ts}\nabla\log \pi)_{\hat k} \mathsf U_{\hat k}\Big]a^{\ts}_i\mathsf{U} \nonumber \\
	&&+\la \nabla a, \sum_{k=1}^n\Big[ a^{\ts}\nabla a^{\ts}_k\mathsf{U}-a^{\ts}_k \nabla a^{\ts}\mathsf{U})\Big] a^{\ts}_k\mathsf{U}\ra -\la  a^{\ts} \nabla^2 a\circ  (a^{\ts}\mathsf{U}),a^{\ts}\mathsf{U}\ra,\nonumber\\
  \label{tenser z}\mathfrak R_{z}(\mathsf{U},\mathsf{U})&=& \sum_{i=1}^n\sum_{k=1}^m a^{\ts}_i\nabla a^{\ts}_i\nabla  z^{\ts}_k \mathsf{U} (z^{\ts}_k\mathsf{U})+\sum_{i,k=1}^n a^{\ts}_i a^{\ts}_i\nabla^2z^{\ts}_k \mathsf{U} (z^{\ts}_k\mathsf{U}) \nonumber \\
	&& -\sum_{i=1}^n\sum_{k=1}^m z^{\ts}_k\nabla a^{\ts}_i\nabla a^{\ts}_i\mathsf{U}(z_k^{\ts}\mathsf{U})-\sum_{i,k=1}^n z^{\ts}_k a^{\ts}_i \nabla^2a^{\ts}_i\mathsf{U}(z_k^{\ts}\mathsf{U})\nonumber \\
&&+\sum_{k=1}^m\sum_{\hat k=1}^{n+m}\Big[ (aa^{\ts} \nabla\log \pi)_{\hat k} \nabla_{\hat k} z^{\ts}_k\mathsf{U}-z_k^{\ts}\nabla (aa^{\ts}\nabla\log \pi)_{\hat k} \mathsf U_{\hat k}\Big]z^{\ts}_k\mathsf{U}\nonumber \\
	&&+\la \nabla a, \sum_{k=1}^m\Big[ a^{\ts}\nabla z^{\ts}_k\mathsf{U}-z^{\ts}_k \nabla a^{\ts}\mathsf{U})\Big] z^{\ts}_k\mathsf{U}\ra -\la  z^{\ts} \nabla^2 a\circ  (a^{\ts}\mathsf{U}),z^{\ts}\mathsf{U}\ra,\nonumber\\
\label{tensor R Psi}\mathfrak{R}_{\pi}(\mathsf{U},\mathsf{U})
	&=&2\sum_{k=1}^m \sum_{i=1}^n\left[\nabla z^{\ts}_{k} z^{\ts}_{k} \nabla a^{\ts}_{i}\mathsf{U} a^{\ts}_{i}\mathsf{U}+z^{\ts}_{k}\nabla z^{\ts}_{k} \nabla a^{\ts}_{i}\mathsf{U} a^{\ts}_{i}\mathsf{U}+z^{\ts}_{k} z^{\ts}_{k} \nabla^2 a^{\ts}_{i} \mathsf{U}   a^{\ts}_{i}\mathsf{U}   \right]\nonumber \\
&&+2\sum_{k=1}^m \sum_{i=1}^n\Big[(z^{\ts}_{k}\nabla a^{\ts}_i\mathsf{U} )^2+ (z^{\ts}\nabla\log\pi)_k \left[ z^{\ts}_{ k}\nabla a^{\ts}_i\mathsf{U} a^{\ts}_i\mathsf{U} \right]\Big] \nonumber\\
&&-2\sum_{j=1}^m\sum_{l=1}^n\left[ \nabla a^{\ts}_{l} a^{\ts}_{l}\nabla  z^{\ts}_{j}\mathsf{U} z^{\ts}_{j}\mathsf{U}+ a^{\ts}_{l}\nabla a^{\ts}_{l}\nabla z^{\ts}_{j}\mathsf{U} z^{\ts}_{j}\mathsf{U}+a^{\ts}_{l}a^{\ts}_{l} \nabla^2z^{\ts}_{j}\mathsf{U} z^{\ts}_{j} \mathsf{U} \right] \nonumber\\
&&-2\sum_{j=1}^m\sum_{l=1}^n \left[ (a^{\ts}_{l}\nabla z^{\ts}_j \mathsf{U})^2  +(a^{\ts}\nabla\log\pi)_l \Big[ a^{\ts}_{l}\nabla z^{\ts}_{j}\mathsf{U} z^{\ts}_{j}\mathsf{U}\Big]  \right], \nonumber
\eea
\bea 
\mathfrak R_{\mathcal I_a}(\mathsf{U},\mathsf{U})
&=&  \la \mathsf{U},\gamma \ra\Big[ \la \nabla a, a^{\ts}\mathsf{U}\ra +\sum_{i=1}^na^{\ts}_i\nabla a^{\ts}_i\mathsf{U}+\la aa^{\ts}\nabla\log\pi,\mathsf{U}\ra\Big]\nonumber-\sum_{\hat i=1}^n\sum_{k=1}^{n+m} \gamma_{k}\nabla_{x_{k}} a^{\ts}_{\hat i}\mathsf{U}( a^{\ts}_{\hat i}\mathsf{U}),\nonumber\\
\mathfrak R_{\gamma_a}(\mathsf{U},\mathsf{U})&=& \frac{1}{2}\sum_{\hat k=1}^{n+m}\gamma_{\hat k} \la \mathsf U,\nabla_{\hat k}(aa^{\ts})\mathsf U\ra -  \la \nabla \gamma \mathsf{U},aa^{\ts}\mathsf{U}\ra_{\mathbb R^{n+m}},\nonumber\\ 
\mathfrak R_{\gamma_z}(\mathsf{U},\mathsf{U})&=& \frac{1}{2}\sum_{\hat k=1}^{n+m}\gamma_{\hat k} \la \mathsf U,\nabla_{\hat k}(zz^{\ts})\mathsf U\ra -  \la \nabla \gamma \mathsf{U},zz^{\ts}\mathsf{U}\ra_{\mathbb R^{n+m}}.\nonumber
\eea
 We define vector \qi{functions $\mathsf D:\hR^{n+m}\rightarrow \mathbb R^{n^2\times 1}$, and $\mathsf E:\hR^{n+m}\rightarrow \mathbb R^{({nm})\times 1}$} as below,
\bea\label{vector D E}
\mathsf D_{ik}= \sum_{\hat i, \hat k=1}^{n+m}a^{\ts}_{i\hat i}\pa_{x_{\hat i}} a^{\ts}_{k\hat k}\mathsf U_{\hat k}, \quad \mathsf E_{ik}=\sum_{\hat i,\hat k=1}^{n+m}a^{\ts}_{i\hat i}\pa_{x_{\hat i}} z^{\ts}_{k\hat k} \mathsf U_{\hat k}.
\eea 
For $\beta\in\mathbb R$, the vector \qi{functions $\Lambda_1: \hR^{n+m}\rightarrow \mathbb R^{n^2\times 1}$ and $\Lambda_2:\hR^{n+m}\rightarrow \mathbb R^{({nm})\times 1}$ are defined as, for $i,l\in\{1,\cdots,n\}$}, 
\beaa 
(\Lambda_1)_{il}&=&\sum_{k=1}^n[\sum_{i'=1}^{n+m} a^{\ts}_{ii'}\lambda^{i'k}_l-\sum_{k'=1}^{n+m}a^{\ts}_{kk'} \lambda^{k'i}_l ]a^{\ts}_k\mathsf U+ \sum_{k=1}^m\Big(\sum_{i'=1}^{n+m}  a^{\ts}_{ii'}\omega^{i'k}_l -\sum_{k'=1}^{n+m} z^{\ts}_{kk'}\lambda^{k'i}_l\Big)z^{\ts}_k\mathsf U \\
&&- \sum_{k=1}^m \sum_{i'=1}^{n+m} a^{\ts}_{ii'} \omega^{i'k}_l  z^{\ts}_k\mathsf U -\frac{\beta}{2} \alpha_l (a^{\ts}_{ i}\mathsf U)+\frac{\beta}{2}   \la \mathsf U,\gamma \ra \mathbf{1}_{\{i=l\}}+\mathsf D_{il},
\eeaa 
and for $i\in \{1,\cdots,n\}$, $l\in\{1,\cdots,m\}$, 
\beaa (\Lambda_2)_{il}
&=&\sum_{k=1}^n[ \sum_{i'=1}^{n+m} a^{\ts}_{ii'} \lambda ^{i'k}_{l+n}-\sum_{k'=1}^{n+m}a^{\ts}_{kk'}\lambda^{k'i}_{l+n}]a^{\ts}_k\mathsf U+\sum_{k=1}^m\Big(\sum_{i'=1}^{n+m}  a^{\ts}_{ii'}\omega^{i'k}_{l+n} -\sum_{k'=1}^{n+m} z^{\ts}_{kk'} \lambda_{l+n}^{k'i}\Big)z^{\ts}_k\mathsf U \\
&&+\sum_{k=1}^n\sum_{k'=1}^{n+m}  z^{\ts}_{lk'} \lambda^{k'k}_i a^{\ts}_{k} \mathsf U+ z^{\ts}_{l}\nabla a^{\ts}_i \mathsf U-\sum_{k=1}^m\sum_{i'=1}^{n+m}a^{\ts}_{ii'} \omega^{i'k}_{l+n}  z^{\ts}_k\mathsf U -a^{\ts}_{i}\nabla z^{\ts}_{l} \mathsf U-\frac{\beta}{2} \alpha_{l+n}(a^{\ts}_{ i}\mathsf U)+\mathsf E_{il}.
\eeaa 
\end{definition}

We last introduce the following Hessian matrix lower bound condition. 
\begin{assum}[Hessian matrix condition]\label{asum2}
For the Hessian matrix $\mathfrak R$ defined in Definition \ref{def: curvature sum}, assume that there exists a constant $\kappa> 0$, such that  
\begin{equation}\label{C}
\mathfrak{R}(x) \succeq  \kappa(a(x)a(x)^{\ts}+z(x)z(x)^{\ts}), \quad\textrm{for all $x\in\mathbb{R}^{n+m}$.}
\end{equation}
{Here ``$\succeq$" represents the order of positive semidefinite matrices.}
\end{assum}

\begin{remark}
    Under Assumption \ref{main condition}, for each indices ${i'}, k, \hat k$, there exist smooth functions $\lambda^{i'k}_{l}$, $\omega^{i'k}_l$ and $\alpha_l$ for $l=1,\cdots,n+m$, such that
\beaa 
\nabla_{i'}a^{\ts}_{k\hat k}&=&\sum_{l=1}^n\lambda^{i'k}_la^{\ts}_{l\hat k}+\sum_{l=1}^{m}\lambda^{i'k}_{l+n} z^{\ts}_{l\hat k},\quad 
 \nabla_{i'}z^{\ts}_{k\hat k}=\sum_{l=1}^n\omega^{i'k}_la^{\ts}_{l\hat k}+\sum_{l=1}^{m}\omega^{i'k}_{l+n} z^{\ts}_{l\hat k}.
 \eeaa
  For a vector field  $\gamma\in\mathbb R^{n+m}$,  we define $\nabla\gamma\in \mathbb R^{(n+m)\times (n+m)}$ with $(\nabla\gamma)_{ij}=\nabla_i\gamma_j$, and denote $\gamma_{\hat k}=\sum_{l=1}^n\alpha_la^{\ts}_{l\hat k}+\sum_{l=1}^m\alpha_{l+n}z^{\ts}_{l\hat k}$.
\end{remark}
{
\begin{remark}
We can verify Assumption \ref{asum2} in several examples, including Langevin dynamics {with the $x$-dependent diffusion matrix} (Section \ref{sec4}). It can also be applied to study three oscillator chains models \cite{EckmannHairer, HairerMattingly} (Section \ref{sec: anharmonic}), where the invariant Gibbs measure has an explicit formula. {Assumption \ref{main condition} and Assumption \ref{asum2}} have been checked for drift-diffusion processes on Lie group-induced sub-Riemannian manifolds. Examples include the Heisenberg group, the displacement group, and the Martinet flat sub-Riemannian structure. See examples in \qi{\cite[Examples 4.1, 4.2, 4.3]{FL2}.} {In addition, we do not assume the iterative symmetry condition between $a$ and $z$, i.e.  $\la a^{\ts}\nabla f, a^{\ts}\nabla |z^{\ts}\nabla f|^2 \ra=\la z^{\ts}\nabla f, z^{\ts}\nabla |a^{\ts}\nabla f|^2 \ra $. We introduce $\mathfrak R_{\pi}$ to handle the situation when the iterative symmetry condition does not hold true.}
\end{remark}

\begin{remark}[Connections with Bakry--{\'E}mery conditions]\label{rmk1}
The Hessian matrix condition in Assumption \ref{asum2} extends the classical Bakry-{\'E}mery condition \cite{BE}. 
\begin{itemize}
\item[(i)] If $\gamma=0$ and $m=0$ {(i.e. $z:=0$)}, then the Hessian matrix condition recovers the Bakry-{\' E}mery condition. 
\item[(ii)] If $\gamma=0$ and $m\neq 0$, then the Hessian matrix condition has been derived in \cite{FL}. 
\item[(iii)] If $\beta=0$ and $m=0$, then the Hessian matrix leads to the Arnold-Carlen-Ju tensor \cite{AC, ACJ};
\item[(iv)] If $a$, $z$ are {constant matrices} and $\beta=0$, then the Hessian matrix leads to the one in Arnold-Erb tensor \cite{AE};
\item[(v)] If $\beta=1$, $m=0$ and $a=\mathbb{I}$, then the Hessian matrix formulates the one proposed in \cite{FL1}. 
\end{itemize}
\end{remark}
}

\subsection{Entropy dissipation}
We are now ready to present the convergence result. Denote 
\begin{equation*}
\mathcal{I}_{a,z}(p\|\pi)=\int \Big\la \nabla \log\frac{p(x)}{\pi(x)}, (a(x)a(x)^{\ts}+z(x)z(x)^{\ts})\nabla\log\frac{p(x)}{\pi(x)}\Big\ra p(x)dx,
\end{equation*}
which measures the difference between functions $p$ and $\pi$. We call functional $\mathcal{I}_{a,z}$ the {\em weighted relative Fisher information functional}. 

In the next theorem, we establish the convergence of SDEs from the dissipation estimation of functional $\mathcal{I}_{a,z}$. 
\begin{theorem}[Weighted Fisher information dissipation]\label{thm1}
Assume that Assumption \ref{prop:main condition} and Assumption \ref{asum2} hold. Assume {the initial distribution $p_0$ satisfies $\mathcal{I}_{a,z}(p_0\|\pi)<+\infty$.} Then
\begin{equation*}
\mathcal{I}_{a,z}(p \|\pi)\leq  e^{-2\kappa t}\mathcal{I}_{a,z}(p_0\|\pi),
\end{equation*}
where $p=p(t,x)$ is the solution of Fokker-Planck equation \eqref{FPE1}. 
\end{theorem}
Using Theorem \ref{thm1}, we demonstrate several convergence results for other functionals, such as 
{Kullback-Leibler} (KL) divergence and $L_1$ distance.  
\begin{corollary}[Functionals decay]\label{col3}
\qi{Assume that Assumption \ref{asum2} holds true with a constant $\kappa>0$.} Then \qi{we have the following decay results}. 
\begin{itemize}
\item[(i)] {KL divergence decay:}
\begin{equation*}
\mathrm{D}_{\mathrm{KL}}(p \|\pi)\leq  \frac{1}{2\kappa}e^{-2\kappa t}\mathcal{I}_{a,z}(p_0\|\pi),
\end{equation*}
where
\begin{equation*}
    \mathrm{D}_{\mathrm{KL}}(p \|\pi):=\int p(t,x)\log\frac{p(t,x)}{\pi(x)}dx,
\end{equation*}
is the KL divergence between functions $p $ and $\pi$.
\item[(ii)]{$L_1$ distance decay:}
\begin{equation*}
\int |p(t,x)-\pi(x)| dx\leq \sqrt{\frac{1}{\kappa}\mathcal{I}_{a,z}(p_0\|\pi) }e^{-\kappa t}.
\end{equation*}
\end{itemize}
\end{corollary}

{
\begin{remark}
Although the proposed convergence study is presented in the spatial domain $\mathbb{R}^{n+m}$, one can also prove the same main result on a torus. In other words, suppose the Fokker-Planck equation \eqref{FPE} is defined on a torus (periodic boundary conditions). All convergence proofs still hold. In the proof, we apply the integration by parts formulas. There are no boundary terms for $x\in \mathbb{R}^{n+m}$ or $x\in\mathbb{T}^{n+m}$, where $\mathbb{T}^{n+m}$ denotes a $n+m$ dimensional torus. Thus, all derivations in the proofs of Theorem \ref{thm1} and Corollary \ref{col3} are the same for $x\in \mathbb{R}^{n+m}$ or $x\in\mathbb{T}^{n+m}$. 
\end{remark}

}

\begin{remark}[Perturbed gradient flows in sub-Riemannian density manifold]
We remark that equation \eqref{FPE1} can be viewed as a perturbed gradient flow in probability density space embedded with the sub-Riemannian Wasserstein metric \cite{FL, FL1}. See related studies for gradient flows in probability density space in \cite{AGS}. The energy functional is chosen as the KL divergence, while a perturbation is given by the vector field $\gamma$. Using this viewpoint, our decomposition allows studying the dynamical behavior of SDE \eqref{a} in three components. Shortly, we show that the Gamma two operator studies the reversible component in SDE \eqref{a}; the generalized Gamma $z$ operator overcomes the degenerate diffusion matrix function $a$; and the irreversible Gamma operator handles the non-gradient drift vector field $b$. {See precise definitions of these Gamma operators at Definition \ref{defn:tilde gamma 2 znew} in  Section \ref{sec3:1}. }  
\end{remark}

\begin{remark}[Optimal convergence rate]
We note that the convergence \qi{rate} $\kappa$ depends on the choice of parameter $\beta$ and the selection of matrix $z(x)$. { In section \ref{sec4}, we demonstrate some examples of $\beta$ and $z$, which guarantee the convergence rate for the underdamped Langevin dynamics. We leave the optimal choice of these parameters for computing the sharp convergence rate in future work.} 
\end{remark}

\subsection{Example of Hessian matrices}
This subsection provides an example of the Hessian matrix. 
Consider a stochastic process:
\begin{equation}\label{example SDE}
dX_t=\Big(-a(X_t)a(X_t)^{\ts}\nabla V(X_t)+\Big(\sum_{j=1}^n \frac{\partial}{\partial X_j}(a(X_t)a(X_t)^{\ts})_{ij}\Big)_{1\le i\le n}-\gamma(X_t)\Big)dt+\sqrt{2} a(X_t) dB_t, 
\end{equation}
where $X_t\in \mathbb{R}^n$, $B_t$ is a standard $n$ dimensional Brownian motion, $V\in C^2(\mathbb{R}^n; \mathbb{R})$ is a potential function satisfying $\int e^{-V(y)}dy<\infty$, $\gamma\in C^{1}(\mathbb{R}^n;\mathbb{R}^n)$ is a vector field satisfying 
\begin{equation*}
  \nabla\cdot(e^{-V(x)}\gamma(x))=0,   
\end{equation*}
and the diffusion matrix $a$ is a smooth diagonal matrix, which satisfies $a_{ii}(x)=a_{ii}(x_i)>0$, for all $x_i\in \mathbb{R}$, $i=1,\cdots,n$, with   
\bea \label{non degenerate a}
a(x)=\begin{pmatrix}a_{11}(x_1)&0&\cdots&0\\
0&a_{22}(x_2)&\cdots&0\\
0 &0&\cdots&0\\
0&\cdots&0&a_{nn}(x_n)
\end{pmatrix}.
\eea 
The solution of equation \eqref{example SDE} is with a non-reversible/non-gradient drift direction ($\gamma$ can be nonzero) and a non-degenerate ($a>0$) process. \qi{Its} invariant density satisfies 
\begin{equation*}
   \pi(x)=\frac{1}{Z}e^{-V(x)}, \quad Z=\int e^{-V(y)}dy.  
\end{equation*}
In this case, we derive the Hessian matrix of SDE \eqref{example SDE}. The proof is provided in {the} Appendix. 

\begin{proposition}\label{prop: non deg non rev a}
    The \qi{Hessian} matrix $\mathfrak R$ for SDE \eqref{example SDE} has the following form, for a constant $\beta\in\hR$,
    \beaa 
    \mathfrak R(x)= (\mathfrak R_a +\beta \mathfrak R_{\mathcal I_a}+(1-\beta) \mathfrak R_{\gamma_a})(x),
    \eeaa 
    where 
    \bea \label{ricci curvature}
\begin{cases}
\mathfrak R_{a,ii}&= a_{ii}^3\partial_{x_i}a_{ii}\partial_{x_i}V +a_{ii}^4\partial^2_{x_ix_i}V-a_{ii}^3\partial_{x_ix_i}^2a_{ii}, \hfill  i=1,\cdots,n;\\
\mathfrak R_{a,ij}&=a_{ii}^2a_{jj}^2\partial^2_{x_ix_j}V, \hfill i,j=1,\cdots,n,  i\neq j.\\
\mathfrak R_{\mathcal I_a,ii}&= \gamma_i[a^{}_{ii}\pa_{x_i}a^{}_{ii}-(a^{}_{ii})^2\pa_{x_i}V], \hfill  i=1,\cdots,n;\\ 
\mathfrak R_{\mathcal I_a,ij}&=\frac{1}{2}[ \gamma_j(2a^{}_{ii}\pa_{x_i}a^{}_{ii}-(a^{}_{ii})^2\pa_{x_i}V)+\gamma_i(2a^{}_{jj}\pa_{x_j}a^{}_{jj}-(a^{}_{jj})^2\pa_{x_j}V)],\hfill  i,j=1,\cdots,n, i\neq j;\\
\mathfrak R_{\gamma_a,ii}&=\gamma_ia^{}_{ii}\pa_{x_i}a_{ii}-\pa_{x_i}\gamma_i (a^{}_{ii} )^2,  \hfill i=1,\cdots,n;\\
\mathfrak R_{\gamma_a,ij}&=-\frac{1}{2}[\pa_{x_i}\gamma_j (a^{}_{jj})^2+ \pa_{x_j}\gamma_i (a^{}_{ii})^2 ], \hfill i,j=1,\cdots,n, i\neq j.  
\end{cases}
\eea 
\end{proposition}
Proposition \ref{prop: non deg non rev a} can be used to compare different tensors in \qi{the} literature; see Remark \ref{rmk1}. For example, we let $n=1$ and $m=0$. In this case, the Hessian matrix is a scalar, which satisfies 
\begin{equation*}
   \mathfrak R= a^3a'V'' +a^4V''-a^3a''+\beta\gamma(aa'-a^2V')+(1-\beta)(\gamma aa'-a^2\gamma'), 
\end{equation*}
Here $'$, $''$ represent the first and second-order derivatives w.r.t. $x$. 

\qi{Based on the above representation of $\mathfrak R$, we note that} $\mathfrak R$ is a modified Hessian matrix \qi{of $V$ through the parameter $\beta$ and the vector field $\gamma$}. Several examples are \qi{given} below:
\begin{itemize}
\item If $a=1$ and $\gamma=0$, then $\mathfrak{R}=V''(x)$;
\item If $a=1$ and $\gamma\neq 0$, $\beta=0$, then $\mathfrak{R}=V''(x)-\gamma'(x)$;
\item If $a=1$ and $\gamma\neq 0$, $\beta=1$, then $\mathfrak{R}=V''(x)-\gamma(x) V'(x)$;
\item If $a=1$ and $\gamma\neq 0$, $\beta\in\mathbb{R}$, then $\mathfrak{R}=V''(x)-\beta\gamma(x)V'(x)-(1-\beta)\gamma'(x)$. 
\end{itemize}
\qi{The positive lower bound of the above modified Hessian matrix} determines \qi{the} convergence behavior of one dimensional SDE \eqref{example SDE}.
\section{Example I: underdamped Langevin dynamics}\label{sec4}
In this section, we apply Theorem \ref{thm1} to prove exponential convergence results for variable coefficient irreversible(non-gradient drift) degenerate SDEs. 

Consider an underdamped Langevin dynamics with variable diffusion coefficients:
\begin{equation}\label{SDE}
\left\{\begin{aligned}
       dx_t=&v_t dt \\
   dv_t=&(- r(x_t) v_t-\nabla_xU(x_t))dt+\sqrt{2 r(x_t)}dB_t,
\end{aligned}\right.
\end{equation}
where $(x_t, v_t)\in\mathbb{R}^2$ is a two dimensional stochastic process, $U\in C^2(\mathbb{R}^1)$ is 
a Lipschitz potential function with $\int e^{-U(x)}dx<+\infty$, $B_t$ is a standard Brownian motion in $\mathbb{R}$, and $r\in\mathbb{R}_+$ is a positive smooth Lipschitz function. Equation \eqref{SDE} often arises in molecular dynamics and Bayesian sampling algorithms; see motivations in \cite{Sachs}. We first check the following facts about SDE \eqref{SDE}. 
\begin{proposition}
An invariant distribution of SDE \eqref{SDE} satisfies 
\bea \label{invariant}
\pi(x,v) = \frac{1}{Z}e^{-H(x,v)},
\eea 
where $H(x,v)=\frac{\|v\|^2}{2}+U(x)$ and $Z=\int_{\mathbb{R}^2} e^{-H(x,v)} dxdv<+\infty$ is a normalization constant. 
\end{proposition}
\begin{proof}
Denote the drift vector field and the diffusion \qi{matrix} as
\bea\label{variable temp underdamped}
b(x,v)=\begin{pmatrix}
	v\\
	-r(x) v-\nabla_xU(x)
\end{pmatrix}, \quad a=\begin{pmatrix}
	0\\
	\sqrt{r(x)}
\end{pmatrix}.
\eea
We need to show that $-\nabla \cdot(\pi(x,v) b(x,v))+\nabla^2_{vv}(r(x)\pi(x,v))=0$. 
The above equality can be formulated into
\begin{equation*}
-\pi \nabla \cdot\begin{pmatrix}v\\-\nabla_xU\end{pmatrix}-\begin{pmatrix}\nabla_x\pi \\ \nabla_v\pi\end{pmatrix}\cdot \begin{pmatrix}v \\ -\nabla_x U\end{pmatrix} + r\nabla_v\cdot(\pi v)+ r \nabla^2_{vv}\pi=0.
\end{equation*}
We observe that $ r\nabla_v\cdot(\pi v)+r\nabla^2_{vv}\pi=r\nabla_v\cdot(\pi\nabla_v\log\frac{\pi}{e^{-\frac{v^2}{2}}})=0$. 
And we know that  
\begin{equation*}
\nabla \cdot\begin{pmatrix}v\\-\nabla_xU\end{pmatrix}=0,\quad\nabla_x\pi\cdot v+\nabla_v\pi\cdot (-\nabla_x U)=0. 
\end{equation*}
This finishes the proof. \qed
\end{proof}

We next prove the convergence result for SDE \eqref{SDE}. 

\noindent{\textbf{Notations:} Denote $X(t)=(x(t),v(t))^{\ts}$. 
Write a vector field $\gamma\in\mathbb{R}^2$ as
\beaa
\gamma(x,v)
=&\begin{pmatrix}0\\
-r v
\end{pmatrix}-\begin{pmatrix}v\\
-r v-\nabla_xU
\end{pmatrix}
=\begin{pmatrix}-v\\
\nabla_x U
\end{pmatrix}.
\eeaa
Consider a constant vector $z=\begin{pmatrix} z_1 \\ z_2\end{pmatrix}$. 
\begin{proposition}\label{thm: constant z}
For any constant $\beta\in\mathbb R$, define a matrix function $\mathfrak{R}(x)\in\mathbb{R}^{2\times 2}$ as
\beaa 
\mathfrak{R}(x)&=&  \mathfrak R_a+\mathfrak R_z+\mathfrak R_{\pi}-\mathfrak M_{\Lambda}+\beta\mathfrak R_{\mathcal I_a} +(1-\beta) \mathfrak R_{\gamma_a}+
\mathfrak R_{\gamma_z},
\eeaa
where
\beaa
\mathfrak R_a&=&\begin{pmatrix}
0&0\\
0& -\frac{\partial^2  \log\pi }{\partial v^2} |a^{\ts}_{12}|^4
\end{pmatrix},\quad \mathfrak{R}_{\pi}
	=\begin{pmatrix}
0&0\\
0& C_{\pi}
\end{pmatrix}, \\
\mathfrak R_z&=&\frac{1}{2}\Big[  \begin{pmatrix}
0\\
-z^{\ts}_1\nabla((a^{\ts}_{12})^2 \frac{\partial  \log\pi }{\partial v})
\end{pmatrix} z^{\ts}_1+z_1\begin{pmatrix}
0& -z^{\ts}_1\nabla((a^{\ts}_{12})^2 \frac{\partial  \log\pi }{\partial v})
\end{pmatrix} \Big],\\ 
\mathfrak R_{\mathcal I_a}&=&\frac{1}{2}\Big[ \gamma (aa^{\ts}\nabla\log\pi)^{\ts}+ (aa^{\ts}\nabla\log\pi)\gamma^{\ts}  \Big] - \begin{pmatrix}
0&0\\
0& \gamma_1\frac{\partial }{\partial x}a^{\ts}_{12}a^{\ts}_{12}
\end{pmatrix},\\ 
\mathfrak R_{\gamma_a}&=&{\frac{1}{2}}\gamma_1\nabla_1(aa^{\ts})-\frac{1}{2}\Big[ (\nabla \gamma)^{\ts}aa^{\ts}+aa^{\ts}\nabla\gamma \Big],\\ 
\mathfrak R_{\gamma_z}&=&-\frac{1}{2}\Big[(\nabla \gamma)^{\ts}zz^{\ts}+zz^{\ts}\nabla\gamma \Big],
\quad \mathfrak M_{\Lambda}=\frac{1}{(a^{\ts}_{12})^2} \mathsf K^{\ts}(aa^{\ts}+zz^{\ts})^{-1}\mathsf K,
\eeaa
with
\beaa\label{C pi}
C_\pi&=&2\Big[z^{\ts}_{1} z^{\ts}_{1} \nabla^2 a^{\ts}_{12}  a^{\ts}_{12}+(z^{\ts}_{1}\nabla a^{\ts}_{12} )^2+ (z^{\ts}_1\nabla\log\pi) [ z^{\ts}_{ 1}\nabla a^{\ts}_{12} a^{\ts}_{12}]\Big],\\
\mathsf K&=&\begin{pmatrix} \mathsf 0& 2z_1^2\pa_x[a^{\ts}_{12}]a^{\ts}_{12}-\frac{1}{2}\beta\gamma_1(a^{\ts}_{12})^2\\
-z_1^2\pa_x[a^{\ts}_{12}]a^{\ts}_{12}+\frac{1}{2}\beta \gamma_1(a^{\ts}_{12})^2&z_1z_2\pa_x[a^{\ts}_{12}]a^{\ts}_{12}\end{pmatrix}.
\eeaa 
If there exists a constant $\kappa>0$, such that 
\begin{equation*}
\mathfrak{R}(x)\succeq \kappa (aa^{\ts}+zz^{\ts}), 
\end{equation*}
then the entropy dissipation results in Theorem \ref{thm1} and Corollary \ref{col3} hold. In other words, 
\begin{equation*}
   { \int |p(t,x,v)-\pi(x,v)| dxdv=O(e^{-\kappa t}). }
\end{equation*}
\end{proposition}
We present several examples to illustrate the convergence rates of underdamped Langevin dynamics \eqref{SDE}. 

\begin{example}[Constant diffusion]
Consider $r(x)=1$. From Theorem \ref{thm: constant z}, 
\beaa
\mathfrak{R}&=&\begin{pmatrix}
0&0\\
0& -\frac{\partial^2  \log\pi }{\partial v^2}
\end{pmatrix}+\frac{1}{2}\Big[  \begin{pmatrix}
0\\
-z^{\ts}_1\nabla( \frac{\partial  \log\pi }{\partial v})
\end{pmatrix} z^{\ts}_1+z \begin{pmatrix}
0&
-z^{\ts}_1\nabla( \frac{\partial  \log\pi }{\partial v})
\end{pmatrix} \Big]-\frac{1}{2}[(\nabla \gamma)^{\ts}zz^{\ts}+zz^{\ts}\nabla\gamma]
\\
&&-\frac{1-\beta}{2}[(\nabla \gamma)^{\ts}aa^{\ts}+aa^{\ts}\nabla\gamma]+\frac{\beta}{2}[\gamma (aa^{\ts}\nabla\log\pi)^{\ts}+(aa^{\ts}\nabla\log\pi)\gamma^{\ts}]  
\\
&&+\frac{\beta^2}{(a^{\ts}_{12})^2} \begin{pmatrix}
0&-\frac{\gamma_1}{2}\\
\frac{\gamma_1}{2}&0
\end{pmatrix}^{\ts}(aa^{\ts}+zz^{\ts})^{-1}\begin{pmatrix}
0&-\frac{\gamma_1}{2}\\
\frac{\gamma_1}{2}&0
\end{pmatrix}.
	\eeaa
We remark that if $\beta=0$, then $\mathfrak{R}$ leads to the tensor in \cite[Formula (7.7)]{AE}. Our formulation adds an additional constant $\beta\in\mathbb{R}$. For example, we consider $U(x)=\frac{x^2}{2}$. Let $\kappa$ be the smallest eigenvalue of $(aa^{\ts}+zz^{\ts})^{-1}\mathfrak{R}$ for $z=\begin{pmatrix} 1\\ 0.1\end{pmatrix}\in\mathbb{R}^2$. We next plot the smallest eigenvalue $\kappa$ on a spatial domain $[-1,1]^2$. As in Figure \ref{figure1}, we observe that a suitable constant $\beta\in\mathbb{R}$ can improve the convergence rate in a local region. In details, if $\beta=0$, we observe that $\kappa(x)=0.0975$ for all $x$. If $\beta=0.1$, we know that $\kappa(x)=0.1$ when $x=[0,0]$. 

 \begin{figure}[H]
    \includegraphics[scale=0.35]{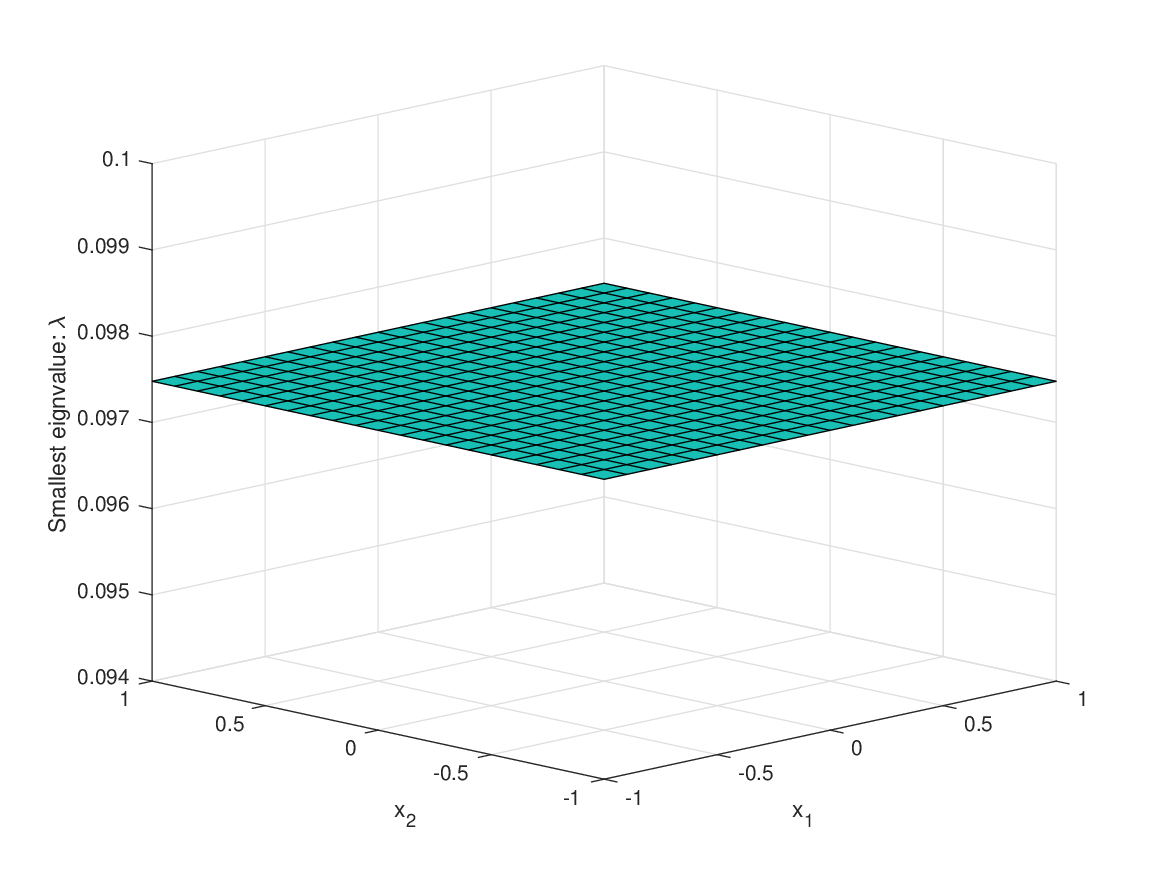}\includegraphics[scale=0.35]{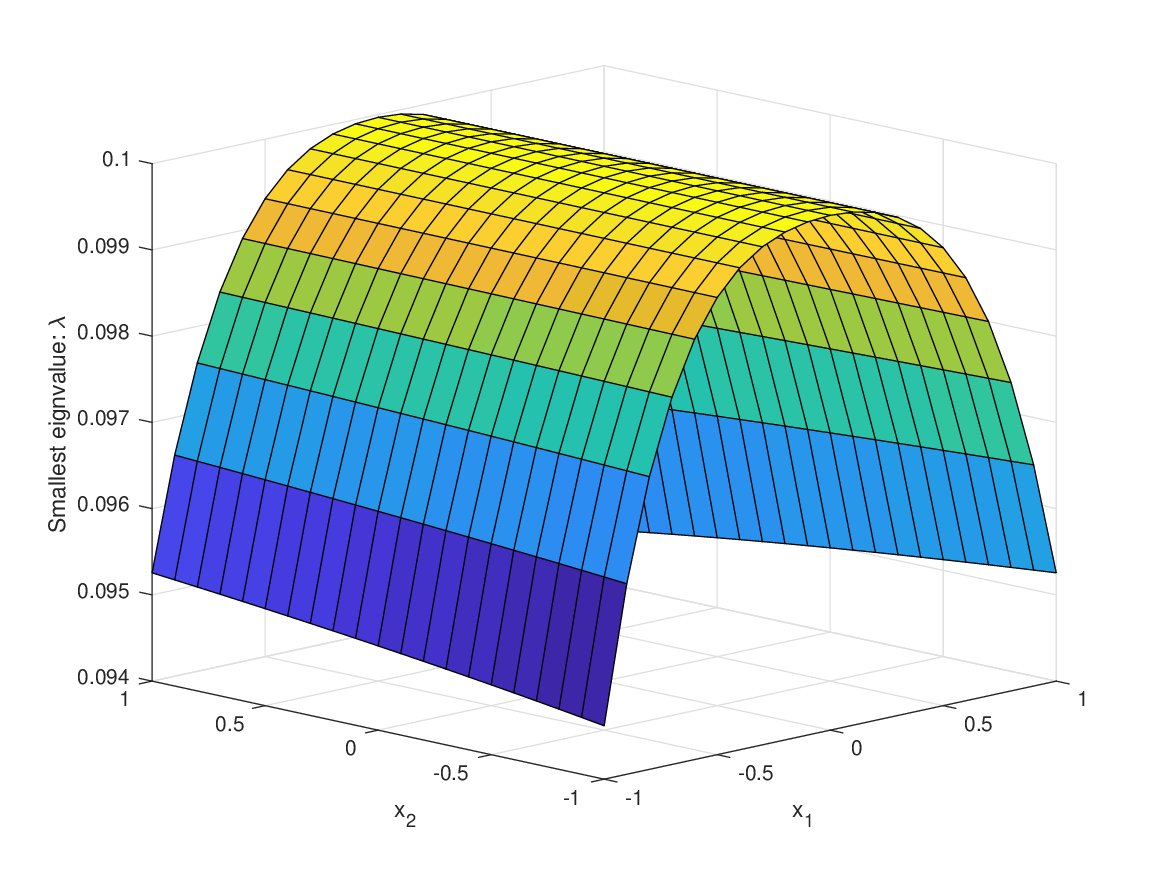}
    \caption{Illustration of convergence rates $\kappa$ with different choices of $\beta$. Left: $\beta=0$; Right $\beta=0.1$. }
    \label{figure1}
\end{figure}
\end{example}

\begin{example}[Variable diffusions]
Consider $U(x)=\frac{x^{2.5}-x}{2.5*1.5}$ with a variable temperature function $r(x)=(\nabla_{xx} U(x))^{-1}$. For simplicity of discussions, we study SDE \eqref{SDE} on a spatial domain $[0.5, 1]^2$ with periodic boundary conditions. In this case, we can establish convergence rates of SDE \eqref{SDE}. Similarly, we plot the smallest eigenvalue $\kappa=\kappa_{\min}((aa^{\ts}+zz^{\ts})^{-1}\mathfrak{R})$ on a spatial domain $[0.5,1]^2$. Here we choose $z=\begin{pmatrix} 1\\0.1\end{pmatrix}$. 
 \begin{figure}[H]
    \includegraphics[scale=0.35]{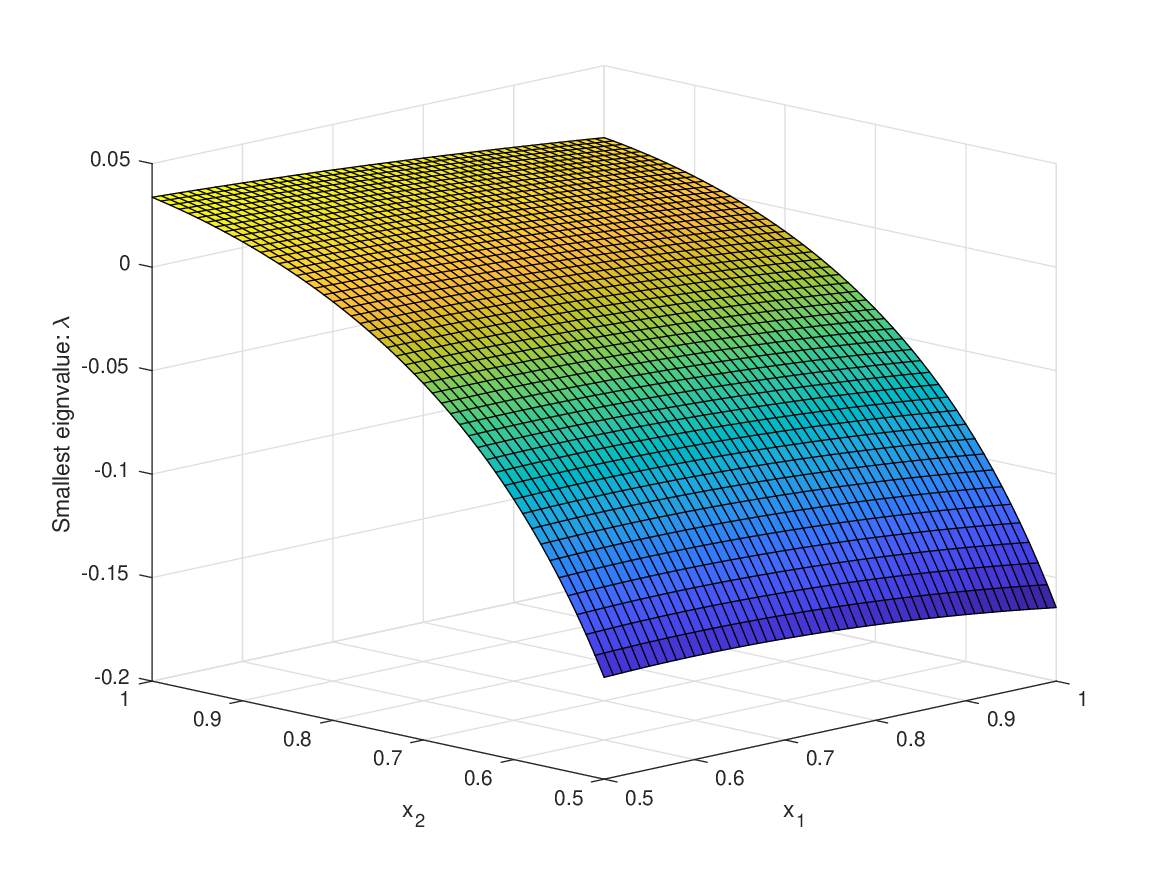}\includegraphics[scale=0.35]{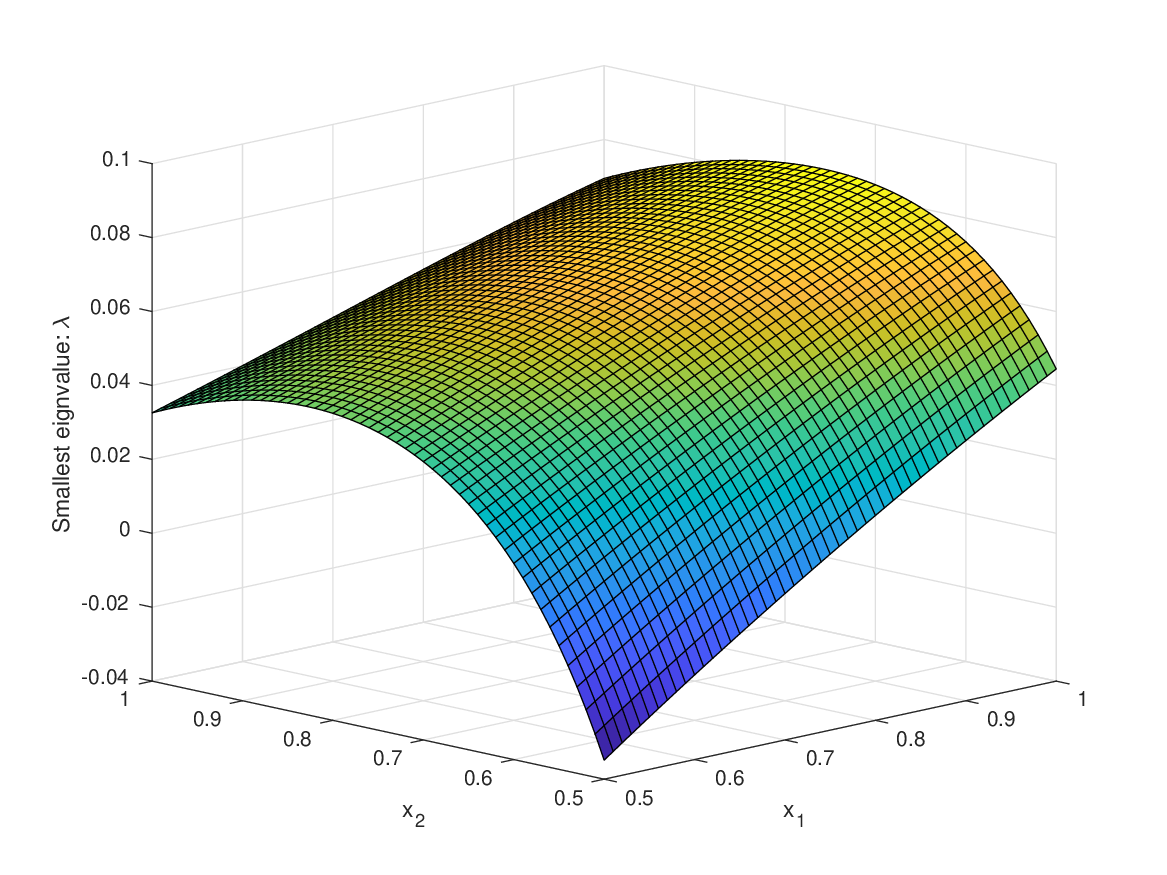}
    \caption{Illustration of convergence rates $\kappa$ with different choices of $\beta$. Left: $\beta=0$; Right: $\beta=0.6$.}
    \label{figure1}
\end{figure}
\end{example}
\begin{proposition}[Sufficient conditions] In Example $1$, let $\beta=0$, $d=1$, $r(x)=r\in\hR_+$, then
\beaa 
\mathfrak R&=& 
	\begin{pmatrix}
	z_1z_2& \frac{1}{2}[rz_1z_2+z_2^2-\partial^2_{xx}U z_1^2+r ]\\
	\frac{1}{2}[rz_1z_2+z_2^2-\partial^2_{xx}U z_1^2+r ]& r^2+rz_2^2-\partial^2_{xx} U z_1z_2
\end{pmatrix}_{2\times 2}. 
\eeaa 
Assume that $\underline{\kappa} \le \pa_{xx}^2U\le \overline{\kappa}$, and there exist constants $z_2\in(0,\frac{r+\sqrt{r^2+4r}}{2})$, such that $\underline{\kappa}$, $\overline{\kappa}$ satisfy the following conditions:
\bea \label{1d condition simple}
-\overline{\kappa}^2+[2(r(1+z_2)-z_2^2)] \underline{\kappa}-[(1-z_2)r+z_2^2]^2>0,\q r^2+rz_2^2-\overline{\kappa} z_2>0.
\eea 
Then there exists $\kappa>0$, such that 
\beaa
\mathfrak R\succeq \kappa (aa^{\ts}+zz^{\ts}).
\eeaa
\end{proposition}
\begin{proof}
It is sufficient to prove $\mathsf{det}(\mathfrak R)>0$ for $z_1z_2>0$ and  $r^2+rz_2^2-\partial^2_{xx} U z_1z_2>0$, which is equivalent to 
\beaa 
z_1z_2(r^2+rz_2^2-\partial^2_{xx} U z_1z_2)-\frac{1}{4}(rz_1z_2+z_2^2-\partial^2_{xx}U z_1^2+r)^2>0.
\eeaa 
It is equivalent to the following inequality:
\bea\label{1d condition general} 
-z_1^4(\pa_{xx}^2U)^2+[2(r(1+z_1z_2)-z_2^2)z_1^2] \pa_{xx}^2U-[(1-z_1z_2)r+z_2^2]^2>0.
\eea 
According to the assumption of $\pa_{xx}^2U$, it is sufficient to prove the following conditions: 
\bea\label{1d condition}
\begin{cases}
& z_1z_2>0,\q r^2+rz_2^2-\overline{\kappa} z_1z_2>0,\q   (r(1+z_1z_2)-z_2^2)>0;\\
&-z_1^4\overline{\kappa}^2+[2(r(1+z_1z_2)-z_2^2)z_1^2] \underline{\kappa}-[(1-z_1z_2)r+z_2^2]^2>0.
\end{cases}
\eea

Let $z_1=1$, then \eqref{1d condition simple} is equivalent to \eqref{1d condition}. We complete the proof.
\end{proof}
\begin{remark}
In general, \eqref{1d condition general} is more general than \eqref{1d condition}. For example, if we take $z_1=1$ and $z_2=\frac{r}{2}$, we get 
\beaa 
-(\partial^2_{xx} U)^2+2(r+\frac{r^2}{4})\partial^2_{xx} U-(r-\frac{r^2}{4})^2>0, \q r^2+\frac{r^3}{4}-\partial^2_{xx}U\frac{r}{2} >0.
\eeaa 
The above inequality implies that $r+\frac{r^2}{4}-\sqrt{r^3} <\partial^2_{xx} U< r+\frac{r^2}{4}+\sqrt{r^3}$, since we also have $r+\frac{r^2}{4}+\sqrt{r^3}<2r+\frac{r^2}{2}$.
\end{remark}

\section{Example II: $3$-oscillators with nearest-neighbor
couplings}\label{sec: anharmonic}
In this section, we apply Theorem \ref{thm1} to prove the exponential convergence in $L_1$ distance for three oscillator chains with nearest-neighbor couplings \cite{HairerMattingly}. 
Consider
\begin{equation}\label{anharmonic chain}
\left\{\begin{aligned}
dq_j=&p_j dt, \quad j=0,1,2, \\
dp_0=& -\xi_0p_0 dt -V'_1(q_0)dt-V_2'(q_0-q_1)dt+\sqrt{2\xi_0T} dB_0, \\
dp_1=& -V'_1(q_1)dt-V_2'(q_1-q_0)dt-V_2'(q_1-q_2)dt, \\
dp_2=& -\xi_2 p_2dt -V_1'(q_2)dt - V_2'(q_2-q_1)dt+\sqrt{2\xi_2T}dB_2.
\end{aligned}\right.
\end{equation}
We abuse the notation slightly. Denote \qi{ $q\in\mathbb{R}^3$} as the state variable and $p\in\mathbb{T}^3$ as the moment variable, where $\mathbb{T}^3$ is a $3$ dimensional torus. The dynamics \eqref{anharmonic chain} is associated with the Hamiltonian function 
\beaa
H(p,q)=\sum_{i=0}^2(\frac{\qi{|p_i|^2}}{2}+V_1(q_i))+\sum_{i=1}^2V_2(q_i-q_{i-1}),
\eeaa 
where function $V_2$ is a smooth interaction potential between neighboring oscillators, and function $V_1$ is a smooth pinning potential. In this paper, we consider $\xi_0=\xi_2$. We assume that there exists a unique invariant measure $\pi$, where
\beaa
 \pi(p,q)=\frac{1}{Z} e^{-\frac{H(p,q)}{T}},
\eeaa  
with a normalization constant $Z=\int_{\mathbb{R}^3\times\mathbb{T}^3} e^{-\frac{H(p,q)}{T}} dqdp<\infty$. 
\begin{remark}
If $\xi_0\neq \xi_2$, there is no {explicit} formulation for the invariant density $\pi$. 
See related studies \cite{EckmannHairer, EckmannPillet, HairerMattingly}. {This example focuses on the case when $\pi$ is explicitly known.}
\end{remark}

\noindent{\textbf{Notations:}}
Denote $X(t)=(q_0(t),q_1(t),q_2(t),p_0(t),p_1(t),p_2(t))^{\ts}$ and $x=(q_0, q_1, q_2, p_0, p_1,p_2)^{\ts}$. Equation \eqref{anharmonic chain} can be written as 
\beaa
dX_t &=& b(X_t)dt+\sqrt{2}adB_t,
\eeaa
where we write $B_t=(B_0(t),B_2(t))^{\ts}$,
\beaa
b(X)&=&\begin{pmatrix}
	p_0& p_1& p_2& -\xi p_0  -\pa_{q_0}H& -\pa_{q_1}H& -\xi p_2 -\pa_{q_2}H\end{pmatrix}^{\ts},
	\eeaa
and
\beaa 
	 a=\begin{pmatrix}
	0&0&0&\sqrt{\xi T}&0&0\\
	0&0&0&0&0&\sqrt{\xi T}
\end{pmatrix}^{\ts}.
	  \eeaa 
Denote $\mathsf I_3=\mathsf{Diag}(1,1,1)_{3\times 3}$ as an identity matrix. 
We select the matrix $z$ as
\beaa
z=\begin{pmatrix}
	z_1\mathsf I_3&0_{3\times 1}\\
	z_2\mathsf I_3&(\overrightarrow{z_3})_{3\times 1}
\end{pmatrix}_{6\times 4},
\eeaa
where we denote $\overrightarrow{z_3}=\begin{pmatrix}
	z_{31}& z_{32}(p_0,p_2)& z_{33}
\end{pmatrix}^{\ts}$ for constants $z_{31},z_{33}\in\hR$, and function $z_{32}(p_0,p_2)$, $z_{32}:\hR^2\rightarrow \hR$. We observe that 
\beaa
aa^{\ts}=\begin{pmatrix}
	0&0\\
	0&(\xi T)\mathsf I_3^O
\end{pmatrix}_{6\times 6},\quad zz^{\ts}=\begin{pmatrix}
	z_1^2\mathsf I_3& z_1z_2 \mathsf I_3\\
	z_1z_2\mathsf I_3&z_2^2\mathsf I_3+\overrightarrow{z_3}\overrightarrow{z_3}^{\ts}
\end{pmatrix}_{6\times 6},
\eeaa
where we write $\mathsf I_3^O=\mathsf{Diag}(1,0,1)_{3\times 3}$, and 
	\beaa
\gamma(x)
&=&  a(x)a(x)^{\ts}\nabla\log \pi(x)-b(x)+\big(\sum_{j=1}^{6}\frac{\partial}{\partial x_j}(a(x)a(x)^{\ts})_{ij}\big)_{1\leq i\leq 6}\\
&=& \begin{pmatrix}
	-p_0&-p_1&-p_2&\pa_{q_0}H&\pa_{q_1}H&\pa_{q_2}H
\end{pmatrix}^{\ts}.
 \eeaa 
It is easy to verify that $\nabla\cdot(\pi \gamma)=0$. We have the following matrix function $\mathfrak R$ for SDE \eqref{anharmonic chain}.
\begin{proposition}\label{prop: tensor oscilator}
 Define a matrix function $\mathfrak{R}\in\mathbb{R}^{6\times 6}$ as
\beaa
\mathfrak R= \begin{pmatrix}
	\mathsf R_1 & \mathsf R_2\\
	\mathsf R_2 & \mathsf R_3
\end{pmatrix}_{6\times 6},
\eeaa
with 
\beaa 
\mathsf R_1 &=& z_1z_2\mathsf I_3,\quad  \mathsf R_2=  \frac{1}{2}(z_1z_2\xi+\xi T)\mathsf I_3^O+\frac{1}{2}(z_2^2\mathsf I_3+\overrightarrow{z_3}\overrightarrow{z_3}^{\ts}-z_1^2L),\\
\mathsf R_3&=& ((\xi T)^2+z_2^2\xi )\mathsf I_3^O+{\frac{1}{2}[\partial_{q_0}H\partial_{p_0}+\partial_{q_2}H\partial_{p_2}](\overrightarrow{z_3}\overrightarrow{z_3}^{\ts} )}-z_1z_2L+\frac{1}{2}(\overrightarrow{z_3^{\mathsf S}}\overrightarrow{z_3}^{\ts}+\overrightarrow{z_3}(\overrightarrow{z_3^{\mathsf S}})^{\ts})+\mathsf I_{\pi},
\eeaa 
where we write 
\bea 
\mathsf I_3^O&=&\mathsf{Diag}(1,0,1)_{3\times 3},\q \overrightarrow{z_3^\mathsf S}=(z_{31}\xi\quad \mathsf S_1 \quad z_{33}\xi)^{\ts},\nonumber \\
 \mathsf I_{\pi}&=&\begin{pmatrix}
	0&0&0\\
	0&-2 z_{32}\mathsf S_1- 2\xi T [|\pa_{p_0}z^{\ts}_{45}|^2+|\pa_{p_2}z^{\ts}_{45}|^2 ]&0\\
	0&0&0
\end{pmatrix}_{3\times 3}, \q \qi{L= (\nabla^2_{q_iq_j}H)_{0\le i,j\le 2}},
\nonumber\\
\mathsf S_1&=&\xi T\pa_{p_0p_0}z_{32} +\xi T \pa_{p_2p_2}z_{32}+(\xi T \pa_{p_0}\log\pi)\pa_{p_0}z_{32}+(\xi T \pa_{p_2}\log\pi)\pa_{p_2}z_{32}. \nonumber
\eea 
If there exists a constant $\kappa>0$, such that 
\begin{equation*}
\mathfrak{R}\succeq \kappa (aa^{\ts}+zz^{\ts}), 
\end{equation*}
then entropy dissipation results in Theorem \ref{thm1} and Corollary \ref{col3} hold. 
	\end{proposition} 
To simplify the computation, we choose the following parameters:
\bea\label{para oscilator}
\xi=T=1,\q  z_{31}=z_{33}=0,\q z_{32}(p_0,p_2)=N-\varepsilon p_0^2/2,\q z_1=1,
\eea 
with a fixed constant $N>0$ and a fixed scaling constant $\varepsilon>0$.  { We further assume that $\partial_{q_0}V_1(q_0)+\partial_{q_0}V_2(q_1-q_0)<\infty$.}
\begin{proposition}[Sufficient conditions]\label{prop chain} Assume that \qi{$(\nabla^2_{q_iq_j}H)_{0\le i,j\le 2}$} satisfies the following conditions, for $\delta_1,\delta_2>0$, $z_2\in (0,\min\{\frac{1+\sqrt{5}}{2},N\})$, and a small constant $\varepsilon>0$,
\bea \label{condtion oscillator new}
\begin{cases}
	& \underline\kappa \mathsf I_3\preceq \qi{(\nabla^2_{q_iq_j}H)_{0\le i,j\le 2}}  \preceq\overline\kappa \mathsf I_3;\\
&  2\underline\kappa -\overline\kappa^2\ge 1-\delta_1,\quad -(z_2^2+N^2)^2+2(N^2-z_2^2)\underline\kappa-\overline{\kappa}^2>0;\\
	&2z_2+2z_2^3-z_2^4-3z_2^2+2(z_2-z_2^2)\underline\kappa>\delta_1.
\end{cases}
\eea 
We denote $0\le \underline\kappa\le\overline\kappa$ as the lower and upper bounds for the eigenvalues of the matrix \qi{$\nabla^2_{qq}H$}. Then there exists a constant $\kappa>0$, such that 
\beaa
\mathfrak R\succeq \kappa (aa^{\ts}+zz^{\ts}).
\eeaa
\end{proposition}
\begin{remark}
	The condition of \qi{$\nabla^2_{qq}H$} in the above proposition implies the condition for \qi{$\nabla^2_{qq}V_1$} and \qi{$\nabla^2_{qq}V_2$}.
\end{remark}
\begin{proof}
	It is sufficient to prove that $\mathfrak R$ defined in Proposition \ref{prop: tensor oscilator} is positive definite. Applying the Schur complement for symmetric matrix function \cite{vandenberghe2004convex}[Appendix A.5], the following conditions are equivalent:
\begin{itemize}
\item[(1)]
$\mathfrak R\succeq 0$ ($\mathfrak R$ is positive semidefinite).
\item[(2)]$\mathsf R_1 \succeq 0$, $(\mathsf I_3-\mathsf R_1\mathsf R_1^{-1})\mathsf R_2 =0$, $\mathsf R_3-\mathsf R_2 \mathsf R_1^{-1}\mathsf R_2\succeq 0$.
\end{itemize}
Plugging in the parameters from \eqref{para oscilator}, and letting $z_2>0$, we need to show  $\mathsf R_3-\mathsf R_2 \mathsf R_1^{-1}\mathsf R_2\succeq 0$. In other words, 
\beaa 
&&\mathsf R_3-\mathsf R_2 \mathsf R_1^{-1}\mathsf R_2\\
&=&\begin{pmatrix}
	1+z_2^2 & 0 & 0\\
	0 & \varepsilon(1-p_0^2)(N-\varepsilon p_0^2/2)-2\varepsilon^2p_0^2-(N-\varepsilon p_0^2/2)\varepsilon p_0\partial_{q_0}H & 0 \\
	0& 0& 1+z_2^2
\end{pmatrix}-z_2L\\
&&-\frac{1}{z_2}\Big(\begin{pmatrix}
	\frac{1}{2}z_2^2+\frac{1}{2}z_2+\frac{1}{2}&0 & 0\\
	0 & \frac{1}{2}z_2^2+\frac{1}{2} (N-\varepsilon p_0^2/2 )^2& 0 \\
	0& 0 &\frac{1}{2}z_2^2+ \frac{1}{2}z_2+\frac{1}{2}
\end{pmatrix}-\frac{1}{2}L \Big)^2\\
&:=&\mathsf{Diag}(C_1,C_2,C_3)-z_2\mathsf O \cdot \mathsf{Diag}(\kappa_1,\kappa_2,\kappa_3)_{3\times 3}\cdot\mathsf O^{\ts}\\
&&-\frac{1}{4z_2}\Big(\mathsf{Diag}(C_4,C_5,C_6)- \mathsf O \cdot \mathsf{Diag}(\kappa_1,\kappa_2,\kappa_3)_{3\times 3}\cdot\mathsf O^{\ts} \Big)^2,
\eeaa 
where 
\beaa 
C_1&=&C_3=1+z_2^2,\quad  C_2=\varepsilon(1-p_0^2)(N-\varepsilon p_0^2/2)-2\varepsilon^2p_0^2 -(N-\varepsilon p_0^2/2)\varepsilon p_0\partial_{q_0}H,\q \\
C_4&=&C_6=1+z_2^2+z_2,\q C_5=z_2^2+(N-\varepsilon p_0^2/2 )^2,
\eeaa 
and we denote the eigenvalue decomposition of matrix $L$ as
\beaa 
L=\mathsf O \cdot \mathsf{Diag}(\kappa_1,\kappa_2,\kappa_3)_{3\times 3}\cdot\mathsf O^{\ts}.
\eeaa  
We write $\mathsf O\in\hR^{3\times 3}$ as the orthogonal matrix with $\mathsf O^{-1}=\mathsf O^{\ts}$ and $\kappa_1\le  \kappa_2\le  \kappa_3$ as the eigenvalues of matrix $L$. Thus
\beaa 
&&\mathsf R_3-\mathsf R_2 \mathsf R_1^{-1}\mathsf R_2\\
&=&\mathsf O\cdot\mathsf{Diag}\Big(\Big(C_i-z_2\kappa_i-\frac{1}{4z_2}( C_{i+3}^2-2C_{i+3}\kappa_i+\kappa^2_i) \Big)_{1\le i\le 3} \Big)_{3\times 3}\cdot\mathsf O^{\ts}\\
&=&\mathsf O\cdot\mathsf{Diag}\Big(\Big(C_i-\frac{C_{i+3}^2}{4z_2}+(\frac{C_{i+3}}{2z_2}-z_2)\kappa_i-\frac{\kappa^2_i }{4z_2}\Big)_{1\le i\le 3} \Big)_{3\times 3}\cdot\mathsf O^{\ts}
\eeaa 
It is sufficient to find parameters $z_2$, $N$, and $\varepsilon$, such that the following inequalities hold:
\beaa 
\Big(C_i-\frac{C_{i+3}^2}{4z_2}+(\frac{C_{i+3}}{2z_2}-z_2)\kappa_i-\frac{\kappa^2_i }{4z_2}\Big)>0, \q 1\le i\le 3.
\eeaa 
Since $z_2>0$, it is equivalent to prove
\bea\label{condition oscilator}
\begin{cases}
z_2(1+z_2^2)-\frac{(1+z_2+z_2^2)^2}{4}+(\frac{1+z_2+z_2^2}{2}-z_2^2)\kappa_i-\frac{\kappa^2_i }{4}>0, \\
z_2\varepsilon(1-p_0^2)(N-\varepsilon p_0^2/2)-2\varepsilon^2p_0^2{z_2}-(N-\varepsilon p_0^2/2)\varepsilon p_0\partial_{q_0}Hz_2\\
\qq-\frac{(z_2^2+(N-\varepsilon p_0^2/2 )^2)^2}{4}
+(\frac{z_2^2+(N-\varepsilon p_0^2/2 )^2}{2}-z_2^2)\kappa_i-\frac{\kappa^2_i }{4}>0.
\end{cases}
\eea 
The first inequality in \eqref{condition oscilator} holds under the following conditions. Suppose $1+z_2-z_2^2>0$,  
\beaa 
2\underline\kappa -\overline\kappa^2\ge 1-\delta_1,\qq 2z_2+2z_2^3-z_2^4-3z_2^2+2(z_2-z_2^2)\underline\kappa>\delta_1.
\eeaa 
{For $p_0\in\mathbb{T}^1$}, and $\partial_{q_0}V_1(q_0)+\partial_{q_0}V_2(q_1-q_0)<\infty$,  we select $\varepsilon$ to be very small, such that 
\beaa 
-\frac{(z_2^2+N^2)^2}{4}+(N^2-z_2^2)\frac{\underline\kappa}{2}-\frac{\overline{\kappa}^2}{4}>0.
\eeaa 
The proof is finished.
\qed 
\end{proof}

\begin{remark}
    In the current paper, we restrict $p\in\mathbb T^3$, which differs from the setting in \cite{HairerMattingly}. We mainly use the torus property to construct  the matrix $z$, i.e. $z_{32}(p_0,p_2)$ in \eqref{para oscilator}, such that we can establish the desired bounds in Proposition \ref{prop chain}. We leave the general setting with $p\in\hR^3$ for future studies.
\end{remark}

\section{Proof}\label{sec3:1}
In this section, we present the main proofs of this paper. 

\subsection{Outline of main results}
Before showing all the detailed proofs \qi{in} Theorem \ref{thm1}, we outline the main ideas in the following three steps. We postpone all derivations \qi{to} subsection \ref{subsection: proof}. 
\begin{itemize}
\item[Step $1$:] We first compute the dissipation of the weighted Fisher information functional. 
\begin{equation*}
\frac{d}{dt}\mathcal{I}_{a,z}(p \|\pi)=-2\int \Gamma_{2,a,z,\gamma}(f,f) pdx,
\end{equation*}
where $f=\log\frac{p}{\pi}$ and 
\begin{equation*}
    \Gamma_{2,a,z,\gamma}(f,f)=\widetilde \Gamma_{2}(f,f)+\widetilde \Gamma_{2}^{z,\pi}(f,f)+\Gamma_{\mathcal{I}_{a,z}}(f,f).
\end{equation*}
The definition of the above operators will be introduced shortly in Definition \ref{defn:tilde gamma 2 znew}.
\item[Step $2$:] We next decompose the weak form of information Gamma calculus. This is to derive the information Bochner's formula (Theorem \ref{thm: information bochner}):
	\beaa
\int \Gamma_{2,a,z,\gamma}(f,f) pdx=\int \Big[  \|\mathfrak{Hess}_{\beta} f\|_{\mathrm{F}}^2+\mathfrak{R}(\nabla f, \nabla f) \Big]pdx.
\eeaa 
\item[Step $3$:] From the information Bochner's formula, we establish the convergence result. In other words, if $\mathfrak{R}\succeq\kappa (aa^{\ts}+zz^{\ts})$, then
\begin{equation*}
    \frac{d}{dt}\mathcal{I}_{a,z}(p \|\pi)\leq -2\kappa \mathcal{I}_{a,z}(p \|\pi). 
\end{equation*}
From Gronwall's inequality, we can prove that the weighted Fisher information decays in Theorem \ref{thm1}. \qi{This decay result also establishes other functionals decay in Corollary \ref{col3}.}  
\end{itemize}
We \qi{remark} that Step $1$ is based on a global in-space computation, which comes from the second-order calculus of entropy functional in the sub-Riemannian density manifold. We carefully handle the non-communicative operators based on $\Gamma_1$, $\Gamma_1^{z}$, and vector field $\gamma$. Step $2$ is a local in-space calculation, which can be viewed as a generalization of Bochner's formula. Step $3$ follows directly from the modified Lyapunov method.  

\subsection{Information Gamma calculus}\label{sec31}
We next introduce \qi{all} tensors for the main proof. These tensors are derived from the information Gamma calculus. Denote the following operators of SDE \eqref{a}. 
For any $f\in C^{\infty}(\mathbb{R}^{n+m})$, the generator of SDE \eqref{a} satisfies
\begin{equation*}
\begin{aligned}
Lf=\widetilde Lf -\la \gamma,\nabla f\ra,
\end{aligned}
\end{equation*}
where
\bea
\widetilde L f&=& \nabla\cdot(aa^{\ts}\nabla f)+\la aa^{\ts}\nabla \log\pi,\nabla f\ra,
\eea
is the reversible component of the Kolmogorov backward operator. For a given matrix function $a\in\mathbb R^{(n+m)\times n }$, we construct a matrix function $z\in\mathbb R^{(n+m)\times m }$ to handle the degenerate component of SDE \eqref{a}. We also denote a $z$-direction generator as
\begin{equation*}
\widetilde L_z f=\nabla\cdot(zz^{\ts}\nabla f)+\la zz^{\ts}\nabla \log\pi,\nabla f\ra.    
\end{equation*}
Define Gamma one bilinear forms $\Gamma_{1},\Gamma_{1}^z\colon C^{\infty}(\mathbb{R}^{n+m})\times C^{\infty}(\mathbb{R}^{n+m})\rightarrow C^{\infty}(\mathbb{R}^{n+m})$ as 
\bea
\Gamma_{1}(f,f)=\la a^{\ts}\nabla f, a^{\ts}\nabla f\ra_{\hR^n},\quad \Gamma_{1}^z(f,f)=\la z^{\ts}\nabla f, z^{\ts}\nabla f\ra.
\eea

We now define the {\em information Gamma operators} for SDE \eqref{a}. In fact, it is a derivative calculation based on the reversible Kolmogorov backward operator $\widetilde L$, the non-reversible vector field $\gamma$, and the invariant distribution $\pi$.
\begin{definition}[Information Gamma operators]\label{defn:tilde gamma 2 znew}
Define the following three bi-linear forms: 
\begin{equation*}
 \widetilde\Gamma_{2}, \widetilde\Gamma_{2}^{z,\pi}, \Gamma_{\mathcal{I}_{a,z}}\colon C^{\infty}(\mathbb{R}^{n+m})\times C^{\infty}(\mathbb{R}^{n+m})\rightarrow C^{\infty}(\mathbb{R}^{n+m}).
\end{equation*}
\begin{itemize}
\item[(i)] Gamma two operator: 
\begin{equation*}
\widetilde\Gamma_{2}(f,f)=\frac{1}{2}\widetilde L\Gamma_{1}(f,f)-\Gamma_{1}(\widetilde Lf, f).
\end{equation*}
\item[(ii)] Generalized Gamma $z$ operator:
\beaa
\widetilde\Gamma_2^{z,\pi}(f,f)&=&\quad\frac{1}{2}\widetilde L\Gamma_1^{z}(f,f)-\Gamma_{1}^z(\widetilde Lf,f)\\
&& \label{new term}+\div^{{\pi}}_z\Big(\Gamma_{1,\nabla(aa^{\ts})}(f,f )\Big)-\div^{{\pi} }_a\Big(\Gamma_{1,\nabla(zz^{\ts})}(f,f )\Big).
  \eeaa 
Here $\div^{{\pi} }_a$, $\div^{{\pi} }_z$ are divergence operators defined by
\begin{equation*}
\div^{{\pi}}_a(F)=\frac{1}{{\pi} }\nabla\cdot({\pi} aa^{\ts} F), \quad\div^{{\pi} }_z(F)=\frac{1}{{\pi} }\nabla\cdot({\pi} zz^{\ts}F),
\end{equation*}
for any smooth vector field $F\in \mathbb{R}^{n+m}$, and $\Gamma_{1, \nabla (aa^{\ts})}$, $\Gamma_{1, \nabla (zz^{\ts})}$ are vector Gamma one bilinear forms defined by 
\beaa
\Gamma_{1,\nabla(aa^{\ts)}}(f,f)&=&\la \nabla f,\nabla(aa^{\ts})\nabla f\ra=(\la \nabla f,\frac{\partial}{\partial x_{\hat k}}(aa^{\ts})\nabla f\ra)_{\hat k=1}^{n+m},\\
\Gamma_{1,\nabla(zz^{\ts)}}(f,f)&=&\la \nabla f,\nabla(aa^{\ts})\nabla f\ra=(\la \nabla f,\frac{\partial}{\partial x_{\hat k}}(zz^{\ts})\nabla f\ra)_{\hat k=1}^{n+m}.
\eeaa
\item[(iii)] Irreversible Gamma operator: 
\beaa
\Gamma_{\mathcal{I}_{a,z}}(f,f)&=& (\widetilde Lf+\widetilde L_zf) \la \nabla f,\gamma\ra -\frac{1}{2}\la \nabla \big(\Gamma_1(f,f)+\Gamma_1^z(f,f)\big),\gamma\ra .
\eeaa
\end{itemize}
\end{definition}

Under Assumption \ref{prop:main condition}, we establish the following {\em information Bochner's formula}. 
\begin{theorem}[Information Bochner's formula]\label{thm: information bochner}
If Assumption \ref{prop:main condition} is satisfied, then the following decomposition holds. For any $f=\log\frac{p}{\pi}\in C^\infty(\mathbb{R}^{n+m},\mathbb{R})$ and any $\beta\in\mathbb R$,
	\beaa
\int \Big[ \widetilde \Gamma_{2}(f,f)+\widetilde \Gamma_2^{z,\pi}(f,f)+\Gamma_{\mathcal I_{a,z}}(f,f)\Big] pdx=\int \Big[  \|\mathfrak{Hess}_{\beta}f\|_{\mathrm{F}}^2
+\mathfrak{R}(\nabla f,\nabla f)\Big]pdx.
\eeaa 
We denote 
\beaa
\|\mathfrak{Hess}_{\beta}f\|_{\mathrm{F}}^2=[\mathsf Q\mathsf X+\Lambda_1]^{\ts} [\mathsf Q\mathsf X+\Lambda_1]+[\mathsf P\mathsf X+\Lambda_2]^{\ts} [\mathsf P\mathsf X+\Lambda_2],
\eeaa
where $\mathfrak R$, $\Lambda_1$, $\Lambda_2$ are defined in Definition \ref{def: curvature sum}. And we define matrices $\mathsf Q$  and $\mathsf P$ by
\bea\label{matrix Q P}
\mathsf Q= a^{\ts}\otimes a^{\ts}
	\in \mathbb{R}^{n^2\times(n+m)^2},\quad \mathsf P=a^{\ts}\otimes z^{\ts} \in \mathbb{R}^{(nm)\times(n+m)^2},
\eea
with $\mathsf Q_{ik\hat i\hat k}=a^{\ts}_{i\hat i}a^{\ts}_{k\hat k}$ and $\mathsf P_{ik\hat i\hat k}=a^{\ts}_{i\hat i}z^{\ts}_{k\hat k}$. More precisely, for each row (resp. column) of $\mathsf Q$, the row (resp. column) indices of $\mathsf Q_{ik\hat i\hat k}$ follow the double summation $\sum_{i=1}^n\sum_{k=1}^n$  (resp. $\sum_{\hat i=1}^{n+1}\sum_{\hat k=1}^{n+m}$). 
For any smooth function $f:\hR^{n+m}\rightarrow \hR$, we define $\mathsf X\in \mathbb{R}^{ (n+m)^2\times 1}$ by the vectorization of the Hessian matrix for function $f$ with $$\mathsf X_{\hat i\hat k}=\frac{\pa^2 f}{\pa x_{\hat i} \pa x_{\hat k}},\quad \textrm{for $\hat i, \hat k=1,\cdots, n+m$.}$$ 
\end{theorem}

\begin{remark}[Optimal Bochner's formula]
We remark that the proposed information Bochner's formula can be asymmetric, which could be formulated into a more symmetric way. In other words, denote a quadratic matrix function by 
\begin{equation*}
\mathcal{F}(Y)=    [\mathsf Q\mathsf Y+\Lambda_1]^{\ts} [\mathsf Q\mathsf Y+\Lambda_1]+[\mathsf P\mathsf Y+\Lambda_2]^{\ts} [\mathsf P\mathsf Y+\Lambda_2],
\end{equation*}
where $Y\in\mathbb{R}^{(n+m)^2\times 1}$ is a vectorization of a symmetric matrix. We next consider a symmetric matrix optimization problem by
\begin{equation*}
Y^*=\arg\min_Y\mathcal{F}(Y). 
\end{equation*}
Then a Bochner's formula forms 
\begin{equation*}
\begin{split}
&  \|\mathfrak{Hess}_{\beta}f\|^2_{\mathrm{F}}
+\mathfrak{R}(\nabla f,\nabla f)\\
=&\mathcal{F}(X)-\mathcal{F}(Y^*)+\mathcal{F}(Y^*)+\mathfrak{R}(\nabla f,\nabla f).
\end{split}
\end{equation*}
We may call the term $\mathcal{F}(X)-\mathcal{F}(Y^*)$ the ``squared Hessian formula'' while name the term $\mathcal{F}(Y^*)+\mathfrak{R}(\nabla f,\nabla f) +\mathfrak{R}_{\mathcal I}(\nabla f,\nabla f)$ the ``Ricci curvature type tensor''. This optimal choice of Bochner's formula is left for future works. {See Remark \ref{hessian relation} in the Appendix for an example of the optimal Bochner's formula.}
\end{remark}

\subsection{Proof of Step 1.}\label{subsection: proof}
\begin{proposition}\label{prop1}
The following equality holds. 
\begin{equation*}
\frac{d}{dt}\mathcal{I}_{a,z}(p \|\pi)=-2\int \Gamma_{2,a,z,\gamma}(\log\frac{p }{\pi},\log\frac{p }{\pi}) p dx,
\end{equation*}
where $p $ is the solution of Fokker-Planck equation \eqref{FPE}.
\end{proposition}
\begin{proof}
The proof combines calculations in \cite{FL, FL1}. We first present the outline of the proof here. Denote 
\begin{equation*}
\mathcal{I}_{a,z}(p \|\pi)=\mathcal{I}_a(p \|\pi)+\mathcal{I}_z(p \|\pi),    
\end{equation*}
where 
\begin{equation*}
    \mathcal{I}_a(p \|\pi)=\int (\nabla \log\frac{p }{\pi}, aa^{\ts}\nabla \log\frac{p }{\pi}) p  dx,\quad 
        \mathcal{I}_z(p \|\pi)=\int (\nabla \log\frac{p }{\pi}, zz^{\ts}\nabla \log\frac{p }{\pi}) p  dx. 
\end{equation*}
We observe the following facts. 

\noindent{\textbf{Claim $0$:}}
\begin{equation*}
    \frac{d}{dt}\mathcal{I}_a(p \|\pi)=-2\int \Big(\widetilde \Gamma_{2}+\Gamma_{\mathcal{I}_a}\Big)(\log\frac{p }{\pi},\log\frac{p }{\pi}) p  dx, 
\end{equation*}
and 
\begin{equation*}
    \frac{d}{dt}\mathcal{I}_z(p \|\pi)=-2\int \Big(\widetilde \Gamma_{2}^{z,\pi}+\Gamma_{\mathcal{I}_z}\Big)(\log\frac{p }{\pi},\log\frac{p }{\pi}) p  dx. 
\end{equation*}
From the Claim $0$, we obtain 
\begin{equation*}
 \frac{d}{dt}\mathcal{I}_{a,z}(p \|\pi)=\frac{d}{dt}\mathcal{I}_a(p \|\pi)+\frac{d}{dt}\mathcal{I}_z(p \|\pi)=-2\int \Gamma_{2,a,z,\gamma}(\log\frac{p }{\pi},\log\frac{p }{\pi})p dx,
\end{equation*}
which finishes the proof. Next, we prove Claim $0$ in details. Denote that $p $ solves the Fokker-Planck equation \eqref{FPE} and let $\tilde L^*$ be the forward Kolmogorov operator. 

Firstly, the dissipation of $\mathcal{I}_a(p \|\pi)$ follows 
\begin{equation*}
\begin{split}
&\frac{d}{dt}\mathcal{I}_a(p \|\pi)\\=& \frac{d}{dt}\int \Gamma_1(\log\frac{p }{\pi}, \log\frac{p }{\pi})p  dx    \\ 
=&\int \Gamma_1(\log\frac{p }{\pi}, \log\frac{p }{\pi})\partial_t p +2\Gamma_1(\partial_t\log\frac{p }{\pi}, \log\frac{p }{\pi})p  dx \\
=&\int \Gamma_1(\log\frac{p }{\pi}, \log\frac{p }{\pi})\partial_t p +2\Gamma_1(\frac{\partial_t p }{p }, \log\frac{p }{\pi})p  dx \\ 
=&\int \Gamma_1(\log\frac{p }{\pi}, \log\frac{p }{\pi})\partial_t p -2\frac{\nabla\cdot(p aa^{\ts}\nabla\log\frac{p }{\pi})}{p }\partial_t p  dx \\
=&\int \Gamma_1(\log\frac{p }{\pi}, \log\frac{p }{\pi})\partial_t p -2\Big(\la\nabla\log p , aa^{\ts}\nabla\log\frac{p }{\pi}\ra+\nabla\cdot(aa^{\ts}\nabla\log\frac{p }{\pi})\Big)\partial_t p  dx \\
=&\int \Gamma_1(\log\frac{p }{\pi}, \log\frac{p }{\pi})\partial_t p -2\Big(\la\nabla\log\frac{p }{\pi}, aa^{\ts}\nabla\log\frac{p }{\pi}\ra+\widetilde L\log\frac{p }{\pi}\Big)\partial_t p  dx \\
=&-2\int \Big\{\frac{1}{2}\Gamma_1(\log\frac{p }{\pi}, \log\frac{p }{\pi})\partial_t p +\widetilde L\log\frac{p }{\pi}\partial_t p  \Big\}dx \\
=&-2\int \Big\{\quad\frac{1}{2}\Gamma_1(\log\frac{p }{\pi}, \log\frac{p }{\pi})\nabla\cdot(p \gamma)+\widetilde L\log\frac{p }{\pi}\nabla\cdot(p \gamma)\\
&\hspace{1.5cm}+\frac{1}{2}\Gamma_1(\log\frac{p }{\pi}, \log\frac{p }{\pi})\widetilde L^*p +\widetilde L\log\frac{p }{\pi}\widetilde L^*p \Big\} dx \\
=&-2\int \Big\{-\frac{1}{2}\la\nabla\Gamma_1(\log\frac{p }{\pi}, \log\frac{p }{\pi}), \gamma\ra+\widetilde L\log\frac{p }{\pi}\la\nabla\log\frac{p }{\pi}, \gamma\ra\\
&\hspace{1.5cm}+\frac{1}{2}\widetilde L\Gamma_1(\log\frac{p }{\pi}, \log\frac{p }{\pi})-\Gamma_1( \widetilde L\log\frac{p }{\pi}, \log\frac{p }{\pi})\Big\} p dx \\
=&-2\int \Big\{\Gamma_{\mathcal{I}_a}(\log\frac{p }{\pi}, \log\frac{p }{\pi})+\widetilde \Gamma_2(\log\frac{p }{\pi}, \log\frac{p }{\pi})\Big\} p dx.
\end{split}
\end{equation*}
In above derivations, we apply the following fact:
\begin{equation*}
\nabla\cdot(p \gamma)=p \Big(\la\nabla\log p , \gamma\ra+\nabla\cdot\gamma\Big)=p \la\nabla\log\frac{p }{\pi}, \gamma\ra,    
\end{equation*}
where we use the identity $\nabla\cdot(\pi\gamma)=\pi\Big(\la\nabla\log\pi, \gamma\ra+\nabla\cdot\gamma\Big)=0$. 

Secondly, we compute the dissipation of $\mathcal{I}_z(p \|\pi)$. Denote 
\begin{equation*}
\widetilde L_z =(\nabla\log\pi, zz^{\ts}\nabla)+\nabla\cdot(zz^{\ts}\nabla).     
\end{equation*}
Hence
\begin{equation*}
\begin{split}
&\frac{d}{dt}\mathcal{I}_z(p \|\pi)\\=& \frac{d}{dt}\int \Gamma_1^z(\log\frac{p }{\pi}, \log\frac{p }{\pi})p  dx    \\ 
=&\int \Gamma_1^z(\log\frac{p }{\pi}, \log\frac{p }{\pi})\partial_t p +2\Gamma_1^z(\partial_t\log\frac{p }{\pi}, \log\frac{p }{\pi})p  dx \\
=&\int \Gamma_1^z(\log\frac{p }{\pi}, \log\frac{p }{\pi})\partial_t p +2\Gamma_1^z(\frac{\partial_t p }{p }, \log\frac{p }{\pi})p  dx \\
=&\int \Gamma_1^z(\log\frac{p }{\pi}, \log\frac{p }{\pi})\partial_t p -2\frac{\nabla\cdot(p zz^{\ts}\nabla\log\frac{p }{\pi})}{p }\partial_t p  dx \\
=&\int \Gamma_1^z(\log\frac{p }{\pi}, \log\frac{p }{\pi})\partial_t p -2\Big(\la\nabla\log p , zz^{\ts}\nabla\log\frac{p }{\pi}\ra+\nabla\cdot(zz^{\ts}\nabla\log\frac{p }{\pi})\Big)\partial_t p  dx\\
=&\int \Gamma_1^z(\log\frac{p }{\pi}, \log\frac{p }{\pi})\partial_t p -2\Big(\la\nabla\log\frac{p }{\pi}, zz^{\ts}\nabla\log\frac{p }{\pi}\ra+\widetilde L_z\log\frac{p }{\pi}\Big)\partial_t p  dx 
 \end{split}
 \end{equation*}
 \begin{equation*}
    \begin{split}
=&-2\int \Big\{\frac{1}{2}\Gamma_1^z(\log\frac{p }{\pi}, \log\frac{p }{\pi})\partial_t p +\widetilde 
L_z\log\frac{p }{\pi}\partial_t p  \Big\}dx \\
=&-2\int \Big\{\quad\frac{1}{2}\Gamma_1^z(\log\frac{p }{\pi}, \log\frac{p }{\pi})\nabla\cdot(p \gamma)+\widetilde L_z\log\frac{p }{\pi}\nabla\cdot(p \gamma)\\
&\hspace{1.5cm}+\frac{1}{2}\Gamma_1^z(\log\frac{p }{\pi}, \log\frac{p }{\pi})\widetilde L^*p +\widetilde L_z\log\frac{p }{\pi}\widetilde L^*p \Big\} dx \\
=&-2\int \Big\{\Gamma_{\mathcal{I}_z}(\log\frac{p }{\pi}, \log\frac{p }{\pi})+\frac{1}{2}\widetilde L_z\Gamma_1(\log\frac{p }{\pi}, \log\frac{p }{\pi})-\Gamma_1( \widetilde L_z\log\frac{p }{\pi}, \log\frac{p }{\pi})\Big\} p dx. 
\end{split}
\end{equation*}
We note that the above formula can not be used to derive Bochner's formula. Since the related third derivative tensors can not be eliminated. To handle this issue, we present the following equality 
\begin{equation}\label{identity}
\int\Big\{\frac{1}{2}\widetilde L_z\Gamma_1(\log\frac{p }{\pi}, \log\frac{p }{\pi})-\Gamma_1( \widetilde L_z\log\frac{p }{\pi}, \log\frac{p }{\pi})\Big\} p dx=\int  \widetilde\Gamma_2^{z,\pi}(\log\frac{p }{\pi},\log\frac{p }{\pi}) p  dx.
\end{equation}
The identity \eqref{identity} has been derived in \cite[Proposition $5.11$]{FL}. For completeness of this paper, we also present its derivation here.  Notice that
\begin{equation*}
    \begin{split}
&\frac{1}{2}\widetilde L_z\Gamma_1(\log\frac{p }{\pi}, \log\frac{p }{\pi})-\Gamma_1( \widetilde L_z\log\frac{p }{\pi}, \log\frac{p }{\pi})\\
=&\frac{1}{2}\widetilde L_z\Gamma_1(\log\frac{p }{\pi}, \log\frac{p }{\pi})-\Gamma_1^z( \widetilde L\log\frac{p }{\pi}, \log\frac{p }{\pi})+\Gamma_1^z( \widetilde L\log\frac{p }{\pi}, \log\frac{p }{\pi})-\Gamma_1( \widetilde L_z\log\frac{p }{\pi}, \log\frac{p }{\pi}).
\end{split}    
\end{equation*}

We next claim that 
\begin{equation*}
\begin{split}
    &\int\Big\{ \Gamma_1^z( \widetilde L\log\frac{p }{\pi}, \log\frac{p }{\pi})-\Gamma_1( \widetilde L_z\log\frac{p }{\pi}, \log\frac{p }{\pi})\Big\} p  dx\\
    =&\int \Big\{\div^{{\pi}}_z\Big(\Gamma_{1,\nabla(aa^{\ts})}(\log\frac{p }{\pi},\log\frac{p }{\pi} )\Big)-\div^{{\pi} }_a\Big(\Gamma_{1,\nabla(zz^{\ts})}(\log\frac{p }{\pi},\log\frac{p }{\pi} )\Big)\Big\} p  dx.
\end{split}
\end{equation*}
We derive it as follows. Firstly, 
\begin{equation*}
\begin{split}
&    \int \Gamma_1^z( \widetilde L\log\frac{p }{\pi}, \log\frac{p }{\pi}) p  dx\\
=& \int\Gamma_1^z\Big(\la\nabla\log \pi, aa^{\ts}\nabla \log\frac{p }{\pi}\ra+\nabla\cdot(aa^{\ts}\nabla\log\frac{p }{\pi}), \log\frac{p }{\pi}\Big)p  dx\\
=&\int\Gamma_1^z\Big(\Gamma_1(\log\pi, \log\frac{p }{\pi}), \log\frac{p }{\pi}\Big)p +\Gamma_1^z\Big(\nabla\cdot(aa^{\ts}\nabla\log\frac{p }{\pi}), \log\frac{p }{\pi}\Big)p  dx\\
=&\int\Gamma_1^z\Big(\Gamma_1(\log\pi, \log\frac{p }{\pi}), \log\frac{p }{\pi}\Big)p dx-\int \nabla\cdot(aa^{\ts}\nabla\log\frac{p }{\pi})\nabla\cdot(p zz^{\ts}\nabla\log\frac{p }{\pi})dx.
\end{split}
\end{equation*}
We note that the second term in the above formula can be written below. 
\begin{equation*}
\begin{split}
&-\int\nabla\cdot(aa^{\ts}\nabla\log\frac{p }{\pi})\nabla\cdot(p zz^{\ts}\nabla\log\frac{p }{\pi}) dx\\
=&-\int\nabla\cdot(\frac{\pi}{p }aa^{\ts}\nabla\frac{p }{\pi})\nabla\cdot(p zz^{\ts}\nabla\log\frac{p }{\pi}) dx\\
=&-\int \Big(\pi\la\nabla\frac{1}{p }, aa^{\ts}\nabla\frac{p }{\pi}\ra+\frac{1}{p }\nabla\cdot(\pi aa^{\ts}\nabla\frac{p }{\pi})\Big)\nabla\cdot(p zz^{\ts}\nabla\log\frac{p }{\pi}) dx\\
=&-\int \pi\la\nabla\frac{1}{p }, aa^{\ts}\nabla\frac{p }{\pi}\ra\nabla\cdot(p zz^{\ts}\nabla\log\frac{p }{\pi}) dx-\int \frac{1}{p }\nabla\cdot(\pi aa^{\ts}\nabla\frac{p }{\pi})\nabla\cdot(\pi zz^{\ts}\nabla\frac{p }{\pi}) dx\\
=&\int \frac{\pi}{p ^2}\la\nabla p , aa^{\ts}\nabla\frac{p }{\pi}\ra\nabla\cdot(p zz^{\ts}\nabla\log\frac{p }{\pi}) -\frac{1}{p }\nabla\cdot(\pi aa^{\ts}\nabla\frac{p }{\pi})\nabla\cdot(\pi zz^{\ts}\nabla\frac{p }{\pi}) dx\\
=&\int \la\nabla \log p , aa^{\ts}\nabla\log\frac{p }{\pi}\ra\nabla\cdot(p zz^{\ts}\nabla\log\frac{p }{\pi}) -\frac{1}{p }\nabla\cdot(\pi aa^{\ts}\nabla\frac{p }{\pi})\nabla\cdot(\pi zz^{\ts}\nabla\frac{p }{\pi}) dx\\
=&\int \Big\{\la\nabla \log \frac{p }{\pi}, aa^{\ts}\nabla\log\frac{p }{\pi}\ra\nabla\cdot(p zz^{\ts}\nabla\log\frac{p }{\pi})+\la\nabla\log\pi, aa^{\ts}\nabla\log\frac{p }{\pi}\ra\nabla\cdot(p zz^{\ts}\nabla\log\frac{p }{\pi})\\
&\qquad -\frac{1}{p }\nabla\cdot(\pi aa^{\ts}\nabla\frac{p }{\pi})\nabla\cdot(\pi zz^{\ts}\nabla\frac{p }{\pi})\Big\} dx\\
=&\int \Big\{-\Gamma_1^{z}(\Gamma_1(\log \frac{p }{\pi},\log\frac{p }{\pi}), \log\frac{p }{\pi})-\Gamma_1^z(\Gamma_1(\log\pi,\log\frac{p }{\pi}), \log\frac{p }{\pi}) p \\ &\qquad-\frac{1}{p }\nabla\cdot(\pi aa^{\ts}\nabla\frac{p }{\pi})\nabla\cdot(\pi zz^{\ts}\nabla\frac{p }{\pi})\Big\} dx\\
=&\int \Big\{-\Gamma_1^{z}(\Gamma_1(\log \frac{p }{\pi},\log\frac{p }{\pi}), \log\frac{p }{\pi})p -\Gamma_1^z(\Gamma_1(\log\pi,\log\frac{p }{\pi}), \log\frac{p }{\pi}) p \\ &\qquad-\frac{1}{p }\nabla\cdot(\pi aa^{\ts}\nabla\frac{p }{\pi})\nabla\cdot(\pi zz^{\ts}\nabla\frac{p }{\pi})\Big\} dx. 
\end{split}
\end{equation*}
We observe that 
\begin{equation*}
\begin{split}
    &-\int\Gamma_1^{z}(\Gamma_1(\log \frac{p }{\pi},\log\frac{p }{\pi}), \log\frac{p }{\pi}) p  dx\\
=&-2\int \nabla^2\log\frac{p }{\pi} (aa^{\ts}\nabla\log\frac{p }{\pi}, zz^{\ts}\nabla\log\frac{p }{\pi}) p  dx\\
&-\int \Big(\Gamma_{1,\nabla(aa^\ts)}(\nabla\log\frac{p }{\pi}, \nabla\log\frac{p }{\pi}), zz^{\ts}\nabla\log\frac{p }{\pi}\Big)p  dx\\
=&-2\int \nabla^2\log\frac{p }{\pi} (aa^{\ts}\nabla\log\frac{p }{\pi}, zz^{\ts}\nabla\log\frac{p }{\pi}) p  dx\\
&-\int \Big(\Gamma_{1,\nabla(aa^\ts)}(\nabla\log\frac{p }{\pi}, \nabla\log\frac{p }{\pi}), zz^{\ts}\nabla\frac{p }{\pi}\Big)\pi dx\\
=&-2\int \nabla^2\log\frac{p }{\pi} (aa^{\ts}\nabla\log\frac{p }{\pi}, zz^{\ts}\nabla\log\frac{p }{\pi}) p  dx\\
&+\int \frac{1}{\pi}\nabla\cdot \Big(\pi zz^{\ts}\Gamma_{1,\nabla(aa^\ts)}(\nabla\log\frac{p }{\pi}, \nabla\log\frac{p }{\pi})\Big) dx.
\end{split}
\end{equation*}
Combining the above three formulas, we have 
\begin{equation*}
\begin{split}
    &\int \Gamma_1^z( \widetilde L\log\frac{p }{\pi}, \log\frac{p }{\pi}) p  dx\\=&\int \frac{1}{\pi}\nabla\cdot \Big(\pi zz^{\ts}\Gamma_{1,\nabla(aa^\ts)}(\nabla\log\frac{p }{\pi}, \nabla\log\frac{p }{\pi})\Big) dx\\
    &-2\int \nabla^2\log\frac{p }{\pi} (aa^{\ts}\nabla\log\frac{p }{\pi}, zz^{\ts}\nabla\log\frac{p }{\pi}) p  dx-\int \frac{1}{p }\nabla\cdot(\pi aa^{\ts}\nabla\frac{p }{\pi})\nabla\cdot(\pi zz^{\ts}\nabla\frac{p }{\pi}) dx. 
\end{split}
\end{equation*}
Secondly, by switching the order of $a$ and $z$, we have 
\begin{equation*}
\begin{split}
    &\int \Gamma_1( \widetilde L_z\log\frac{p }{\pi}, \log\frac{p }{\pi}) p  dx\\=&\int \frac{1}{\pi}\nabla\cdot \Big(\pi aa^{\ts}\Gamma_{1,\nabla(zz^\ts)}(\nabla\log\frac{p }{\pi}, \nabla\log\frac{p }{\pi})\Big) dx\\
    &-2\int \nabla^2\log\frac{p }{\pi} (aa^{\ts}\nabla\log\frac{p }{\pi}, zz^{\ts}\nabla\log\frac{p }{\pi}) p  dx-\int \frac{1}{p }\nabla\cdot(\pi aa^{\ts}\nabla\frac{p }{\pi})\nabla\cdot(\pi zz^{\ts}\nabla\frac{p }{\pi}) dx. 
\end{split}
\end{equation*}
By subtracting the above two items, we derive the information Gamma $z$ operator. \qed
\end{proof}

\subsection{Proof of Step 2.} We first introduce the following definition.
\begin{definition}\label{definition: F G V}
For any vector field $\mathsf{U}\in\mathcal C^{\infty}(\hR^{n+m})$, 
we define vectors $\mathsf C, \mathsf F,\mathsf G, \mathsf V^a \in \mathbb R^{(n+m)^2\times 1}$, $\mathsf D \in \mathbb R^{n^2\times 1}$ and $\mathsf E\in \mathbb R^{({nm})\times 1}$ as below. For $\hat i,\hat k=1,\cdots,n+m$, 
\bea\label{vecter CDFEGV}
\mathsf C_{\hat i\hat k} &=& \sum_{i,k=1}^n\Big( a^{\ts}_{i\hat i}a^{\ts}_{i}\nabla a^{\ts}_{k\hat k}  -a^{\ts}_{i\hat k}a^{\ts}_{k}\nabla a^{\ts}_{i\hat i} \Big) a^{\ts}_k \mathsf{U},\quad \mathsf D_{ik}= a^{\ts}_{i}\nabla a^{\ts}_k\mathsf U,  \nonumber\\ 
\mathsf F_{\hat i\hat k}&=& \sum_{i=1}^n\sum_{k=1}^m\Big(  a^{\ts}_{i\hat i}a^{\ts}_{i}\nabla z^{\ts}_{k\hat k}- z^{\ts}_{k}\nabla a^{\ts}_{i\hat i} a^{\ts}_{i\hat k}\Big)z^{\ts}_k\mathsf U, \quad E_{ik}=a^{\ts}_{i}\nabla z^{\ts}_k \mathsf U,\nonumber \\
\mathsf G_{\hat i\hat k}&=& \sum_{i=1}^n\sum_{k=1}^m \left[\left( z^{\ts}_{k \hat k} z^{\ts}_{k}\nabla a^{\ts}_{i\hat i}a^{\ts}_{i} \mathsf U+a^{\ts}_{i\hat i}z^{\ts}_{k \hat k} z^{\ts}_{k}\nabla a^{\ts}_i \mathsf U\right)-\left(a^{\ts}_{i\hat i}a^{\ts}_{i}\nabla z^{\ts}_{k\hat k} z^{\ts}_k\mathsf U +z^{\ts}_{k\hat k}a^{\ts}_{i\hat i}a^{\ts}_{i}\nabla z^{\ts}_{k} \mathsf U\right)\right],\nonumber\\ 
\mathsf V_{\hat i\hat k}^a
&=& -\frac{1}{2}\sum_{i=1}^n  \gamma_{\hat k} a^{\ts}_{ i\hat i}(a^{\ts}_{ i}\mathsf U)+\frac{1}{2}  \sum_{k=1}^n a^{\ts}_{k\hat k} a^{\ts}_{k\hat i}  \la \mathsf U,\gamma \ra.\nonumber
\eea 
\end{definition}
We further divide all needed calculations into the following steps. Let $\mathsf U=\nabla f$ in Definition \ref{definition: F G V}. We have the following facts.

\begin{proposition} \label{prop: gammma I} {Denote $f=\log\frac{p}{\pi}$. Then}
\beaa
\int \Gamma_{\mathcal{I}_{a}}(f,f)p dx= \int\mathfrak R_{\gamma_a}(\nabla f,\nabla f)pdx,\quad 
\int \Gamma_{\mathcal{I}_{z}}(f,f)p dx=\int\mathfrak R_{\gamma_z}(\nabla f,\nabla f) pdx.
\eeaa
\end{proposition}

 \begin{proof}
 We show the proof of the case $\int [\Gamma_{\mathcal I_a}(f,f)+\Gamma_{\mathcal I_z}(f,f)]pdx$
 below. For $f=\log\frac{p}{\pi}$, we have  
\beaa
&&\int \Gamma_{\mathcal I_{a,z}}(f,f)p dx\\
&=&\int \Big[(\widetilde Lf+\widetilde L_zf)\la \nabla f,\gamma\ra-\frac{1}{2}\la \nabla(\Gamma_1(f,f)+\Gamma_1^z(f,f)),\gamma\ra  \Big]pdx  \\
&=&\int \Big[\nabla\cdot ((aa^{\ts}+zz^{\ts})\nabla f)\la \nabla f,\gamma\ra +\la \nabla f,\gamma\ra  \la \nabla f, (aa^{\ts}+zz^{\ts})\nabla \log \pi \ra \Big] p dx\\
&&+ \int \frac{1}{2}\nabla\cdot( \gamma p) \la \nabla f,(aa^{\ts}+zz^{\ts})\nabla f\ra dx\\
&=&\int \Big[\nabla\cdot ((aa^{\ts}+zz^{\ts})\nabla f)\la \nabla f,\gamma\ra +\la \nabla f,\gamma\ra  \la \nabla f, (aa^{\ts}+zz^{\ts})\nabla \log \pi \ra \Big] p dx\\
&&+\int \frac{1}{2}\la  \nabla f, \gamma \ra  \la \nabla f,(aa^{\ts}+zz^{\ts})\nabla f\ra pdx,
\eeaa
where we use the fact $\nabla\cdot(p\gamma)=\frac{\nabla\cdot(p\gamma)}{p}p=(\la \nabla\log p,\gamma\ra +\nabla\cdot\gamma)p=\la \nabla f,\gamma\ra p$. Notice that
\beaa
&&\int \nabla\cdot ((aa^{\ts}+zz^{\ts})\nabla f)\la \nabla f,\gamma\ra  p dx\\
&=&- \int\Big[ \la  (aa^{\ts}+zz^{\ts})\nabla f,\nabla \log p \ra \la \nabla f,\gamma\ra  +\la (aa^{\ts}+zz^{\ts})\nabla f,\nabla^2f \gamma\ra \Big] p dx\\
&&-\int \la  (aa^{\ts}+zz^{\ts})\nabla f,\nabla \gamma  \nabla f\ra p dx.
\eeaa
Combining the above two terms, we have
\beaa
&&\int \Gamma_{\mathcal I_{a,z}}(f,f)p dx\\
&=& \int -\frac{1}{2}(\nabla f,\gamma) (\nabla f, (aa^{\ts}+zz^{\ts})\nabla f)p\\
&&- \int\Big[ \la  (aa^{\ts}+zz^{\ts})\nabla f,\nabla \gamma \nabla f\ra  +\la (aa^{\ts}+zz^{\ts})\nabla f,\nabla^2 f \gamma\ra \Big] p dx\\
&=& \int -\frac{1}{2}\la \nabla p,\gamma\ra  \la \nabla f, (aa^{\ts}+zz^{\ts})\nabla f\ra +\int\frac{1}{2}\la \nabla\log\pi, \gamma\ra   \la \nabla f, (aa^{\ts}+zz^{\ts})\nabla f\ra  pdx\\
&&- \int\Big[\la (aa^{\ts}+zz^{\ts})\nabla f, \nabla \gamma \nabla f\ra  +\la (aa^{\ts}+zz^{\ts})\nabla f,\nabla^2 f\gamma\ra \Big] p dx.
\eeaa
In other words, 
\beaa 
&&\int \Gamma_{\mathcal I_{a,z}}(f,f)p dx\\
&=& \int \frac{1}{2} \nabla \cdot (\gamma \la \nabla f, (aa^{\ts}+zz^{\ts})\nabla f)\ra p +\int\frac{1}{2}\la \nabla\log\pi, \gamma\ra   \la \nabla f, (aa^{\ts}+zz^{\ts})\nabla f\ra  pdx\\
&&- \int\Big[\la (aa^{\ts}+zz^{\ts})\nabla f, \nabla \gamma \nabla f\ra  +\la (aa^{\ts}+zz^{\ts})\nabla f,\nabla^2 f\gamma\ra \Big] p dx\\
&=&\frac{1}{2}\int \la \gamma, \la\nabla f,\nabla(aa^{\ts}+zz^{\ts})\nabla f \ra \ra p dx-\int  \la \nabla \gamma\nabla f,(aa^{\ts}+zz^{\ts})\nabla f\ra pdx,
\eeaa
where we use the fact $\frac{\nabla\cdot (\pi \gamma)}{\pi}=\la \nabla\log\pi,\gamma\ra+\nabla\cdot\gamma=0$. The proof is completed.
\qed 
 \end{proof}

\begin{proposition}\label{prop: info gamma tensor} 
For any $f\in C^{\infty}(\mathbb{R}^{n+m})$, we have
\beaa
\Gamma_{\mathcal I_a}(f,f)&=& \mathfrak R_{\mathcal I_a}(\nabla f,\nabla f)+2\mathsf V_a^{\ts}\mathsf X.
\eeaa
\end{proposition}

\begin{proof}
From our definition, 
\beaa
\Gamma_{\mathcal I_a}(f,f)= \la \nabla f,\gamma(x)\ra \widetilde L f-\frac{1}{2}\la \nabla \Gamma_1(f,f),\gamma(x)\ra.
\eeaa
We look at the first term by 
\beaa
\la \nabla f,\gamma \ra \widetilde Lf&=&\la \nabla f,\gamma \ra\Delta_af+\la \nabla f,\gamma \ra\la aa^{\ts}\nabla\log\pi,\nabla f\ra\\
&=&\la \nabla f,\gamma \ra\Big[ \la (\nabla a), (a^{\ts}\nabla f)\ra +\sum_{i=1}^na^{\ts}_i\nabla a^{\ts}_i\nabla f+\sum_{i=1}^na^{\ts}_ia^{\ts}_i\nabla^2f\Big]\\
&&+\la \nabla f,\gamma \ra\la aa^{\ts}\nabla\log\pi,\nabla f\ra,
\eeaa
where we use the fact that
\beaa
\Delta_af&=&\nabla \cdot (aa^{\ts} \nabla f)\\
 &=& \la (\nabla a), (a^{\ts}\nabla f)\ra +\la a^{\ts}\nabla , (a^{\ts}\nabla f)\ra \\
 &=& \la (\nabla a),(a^{\ts}\nabla f)\ra +\sum_{i=1}^na^{\ts}_i\nabla a^{\ts}_i\nabla f+\sum_{i=1}^na^{\ts}_ia^{\ts}_i\nabla^2f.
\eeaa
We further reformulate the cross terms as below:
\beaa
\la \nabla f,\gamma \ra\sum_{i=1}^na^{\ts}_ia^{\ts}_i\nabla^2f&=&\la \nabla f,\gamma \ra\sum_{i=1}^na^{\ts}_{i\hat i}a^{\ts}_{i\hat k}\frac{\partial^2 f}{\partial x_{\hat i}\partial x_{\hat k}}.
\eeaa
As for the second term, we have 
\beaa
\frac{1}{2}\gamma \nabla \Gamma_{1}(f,f)&=&\frac{1}{2}\sum_{k=1}^{n+m}\gamma_{k}\frac{\pa}{\pa x_{k}}(\la a^T\nabla f,a^T\nabla f\ra_{\hR^n})\\
&=&\sum_{k, \hat i=1}^{n+m}\sum_{i=1}^n (\gamma_{k}\frac{\pa a^T_{i\hat i}}{\pa x_{k}}\frac{\pa f}{\pa x_{\hat i}}+\gamma_{k}a^T_{ i\hat i}\frac{\pa^2f}{\pa x_{k}\pa x_{\hat i}})(a^T_i\nabla f),\\
&=&\sum_{i=1}^n\sum_{k=1}^{n+m} \gamma_{k}\nabla_{x_{k}} a^{\ts}_{ i}\nabla f( a^{\ts}_{ i}\nabla f)+\sum_{i=1}^n\sum_{k=1}^{n+m}  \gamma_{k} a^{\ts}_{ i}\nabla^2f(a^{\ts}_{ i}\nabla f).
\eeaa
Combining the above terms and applying Definition \eqref{definition: F G V}, we derive $\Gamma_{\mathcal I_a}(f,f)$.
\qed
\end{proof}

\begin{proposition}\label{prop: gamma 2}
For any $f\in C^{\infty}(\mathbb{R}^{n+m})$, we have
\beaa  \widetilde  \Gamma_2(f,f) &=&(\mathsf Q \mathsf X+\mathsf D)^{\ts}\mathsf (\mathsf Q\mathsf X+\mathsf D)+ 2\mathsf C^{\ts}\mathsf X+ \mathfrak R_{a}(\nabla f,\nabla f),
\eeaa
\beaa
  \widetilde  \Gamma_2^z(f,f) &=& (\mathsf P\mathsf X+\mathsf E)^{\ts}(\mathsf P\mathsf X+\mathsf E)+2\mathsf F^{\ts}\mathsf X+\mathfrak R_{z}(\nabla f,\nabla f),
	\eeaa 
	\beaa
	\div^{\pi}_z(\Gamma_{\nabla(aa^{\ts})}f,f )-\div^{\pi}_a(\Gamma_{\nabla(zz^{\ts})}f,f )&=&\mathfrak{R}_{\pi}(\nabla f,\nabla f)+2\mathsf G^{\ts}\mathsf X.
	\eeaa
\end{proposition}
The detailed derivations of the above propositions are given in the following lemmas. 


\qi{For the convenience of notation}, we recall the operator $\widetilde L$ as below:
\bea
\widetilde L f&=& \Delta_a f+\la aa^{\ts}\nabla \log\pi,\nabla f\ra,
\eea
with 
\bea
\Delta_af=\nabla\cdot(aa^{\ts}\nabla f),
\eea
and 
\bea
\Gamma_{2,a}=\frac{1}{2}\Delta_a\Gamma_1(f,f)-\Gamma_1(\Delta_af,f).
\eea
\subsubsection{Derivation of $\widetilde \Gamma_{2}(f,f)$}

\begin{lemma}\label{lemma step 0}
\beaa  \widetilde  \Gamma_2(f,f) &=& \Gamma_{2,a}(f,f)+
\sum_{i=1}^n\sum_{\hat k=1}^{n+m}\Big[ (aa^{\ts} \nabla\log \pi)_{\hat k} \nabla_{\hat k} a^{\ts}_i\nabla f-a_i^{\ts}\nabla (aa^{\ts}\nabla\log \pi)_{\hat k} \nabla_{\hat k}f\Big]a^{\ts}_i\nabla f.
\eeaa
\end{lemma}

\begin{proof}[Proof of Lemma \ref{lemma step 0}] From our definition, one can get  
	\beaa  \widetilde  \Gamma_2(f,f) &=& \Gamma_{2,a}(f,f)+\frac{1}{2}\la aa^{\ts}\nabla\log \pi, \nabla \Gamma_1(f,f)\ra-\Gamma_1(\la aa^{\ts}\nabla\log\pi, \nabla f\ra,f),
	\eeaa
where we take $b=aa^{\ts}\nabla\log \pi$. We \qi{obtain}
\beaa
\frac{1}{2}b\nabla \Gamma_{1}(f,f)&=&\frac{1}{2}\sum_{\hat k=1}^{n+m}b_{\hat k}\frac{\pa}{\pa x_{\hat k}}(\la a^T\nabla f,a^T\nabla f\ra_{\hR^n})\\
&=&\sum_{\hat k,\hat i=1}^{n+m}\sum_{i=1}^n(b_{\hat k}\frac{\pa a^T_{i\hat i}}{\pa x_{\hat k}}\frac{\pa f}{\pa x_{\hat i}}+b_{\hat k}a^T_{i\hat i}\frac{\pa^2f}{\pa x_{\hat k}\pa x_{\hat i}})(a^T\nabla f)_i,
\eeaa
and
\beaa
-\Gamma_{1}(b\nabla f,f)&=&-\la a^T\nabla(b\nabla f),a^T\nabla f\ra _{\hR^n}\\
&=&-\sum_{i=1}^n\sum_{\hat k,\hat i=1}^{n+m}(a^{\ts}_{i\hat i}\frac{\pa b_{\hat k}}{\pa x_{\hat i}}\frac{\pa f}{\pa x_{\hat k}} +a^{\ts}_{i\hat i}b_{\hat k}\frac{\pa^2f}{\pa x_{\hat i}\pa x_{\hat k}}) (a^T\nabla f)_i.
\eeaa
Summing over all the terms , we have 
	\beaa
\widetilde	\Gamma_2(f,f)&=&\Gamma_{2,a}(f,f)+\sum_{\hat k,\hat i=1}^{n+m}\sum_{i=1}^n(b_{\hat k}\frac{\pa a^T_{i\hat i}}{\pa x_{\hat k}}\frac{\pa f}{\pa x_{\hat i}}-a^{\ts}_{i\hat i}\frac{\pa b_{\hat k}}{\pa x_{\hat i}}\frac{\pa f}{\pa x_{\hat k}} ) (a^T\nabla f)_i\\
&=& \Gamma_{2,a}(f,f)+\sum_{i=1}^n\sum_{\hat k=1}^{n+m}\Big[ (aa^{\ts} \nabla\log \pi)_{\hat k} \nabla_{\hat k} a^{\ts}_i\nabla f-a_i^{\ts}\nabla (aa^{\ts}\nabla\log \pi)_{\hat k} \nabla_{\hat k}f\Big]a^{\ts}_i\nabla f.
	\eeaa
\qed
\end{proof}{
\begin{lemma}\label{lemma step 1}
	\bea\label{gamma 2 relation}
	&&\frac{1}{2}\Delta_a \Gamma_{1}(f,f)-\Gamma_{1}(\Delta_a f,f)\\
	&=&\frac{1}{2}\la a^{\ts}\nabla, (a^{\ts}\nabla |a^{\ts}\nabla f|^2 )\ra -\la a^{\ts} \nabla \la (a^{\ts}\nabla) ,(a^{\ts}\nabla f)\ra ,a^{\ts}\nabla f\ra_{\hR^n}\nonumber \\
	&&+\la \nabla a,\sum_{k=1}^n\Big[ a^{\ts}\nabla a^{\ts}_k\nabla f-a^{\ts}_k \nabla a^{\ts}\nabla f)\Big] a^{\ts}_k\nabla f\ra -\la  a^{\ts} \nabla^2 a\circ  (a^{\ts}\nabla f),a^{\ts}\nabla f\ra.\nonumber
	\eea
	\end{lemma}}
\begin{proof}[Proof of Lemma \ref{lemma step 1} ]
According to our definition, we have 
 \beaa\label{gamma 2 --1}
 \Delta_a\Gamma_{1} (f,f)&=&\nabla\cdot (aa^{\ts}|a^{\ts}\nabla f|^2)\\
 &=&\la \nabla a , a^{\ts}\nabla |a^{\ts}\nabla f|^2\ra +\la a^{\ts}\nabla, a^{\ts}\nabla |a^{\ts}\nabla f|^2 \ra .
 \eeaa
Similarly, we have 
\beaa
\nabla\cdot (aa^{\ts} \nabla f)&=&\la \nabla a,  a^{\ts}\nabla f\ra +\la a^{\ts}\nabla, a^{\ts}\nabla f\ra ,
\eeaa
which gives us 
\beaa
\Gamma_{1}(\Delta_af,f)&=&\la a^{\ts} \nabla (\nabla\cdot (aa^{\ts}\nabla f) ),a^{\ts}\nabla f\ra, \\
&=& \la a^{\ts} \nabla \left(\la \nabla a, a^{\ts}\nabla f\ra \right),a^{\ts}\nabla f\ra +\la a^{\ts} \nabla (\la a^{\ts}\nabla, a^{\ts}\nabla f\ra ),a^{\ts}\nabla f\ra\nonumber\\
&=&\la  a^{\ts} \nabla^2 a\circ  (a^{\ts}\nabla f),a^{\ts}\nabla f\ra+\la \nabla a\circ (a^{\ts} \nabla (a^{\ts}\nabla f)),a^{\ts}\nabla f\ra\nonumber\\
&&+\la a^{\ts} \nabla (\la a^{\ts}\nabla, a^{\ts}\nabla f\ra ),a^{\ts}\nabla f\ra.
\eeaa
Combining the above terms, we get  
\beaa
&&\frac{1}{2}\Delta_a \Gamma_{1} (f,f)-\Gamma_1 (\Delta_a f,f)\\
&=&\frac{1}{2}\la a^{\ts}\nabla,  (a^{\ts}\nabla |a^{\ts}\nabla f|^2 )\ra -\la a^{\ts} \nabla (\la a^{\ts}\nabla, a^{\ts}\nabla f\ra ),a^{\ts}\nabla f\ra+\underbrace{\frac{1}{2} \la \nabla a , a^{\ts}\nabla |a^{\ts}\nabla f|^2\ra }_{ \textbf{I}}\\
&& -\la  \left(a^{\ts} \nabla^2 a\circ  (a^{\ts}\nabla f)\right),a^{\ts}\nabla f\ra-\la \nabla a\circ (a^{\ts} \nabla (a^{\ts}\nabla f)),a^{\ts}\nabla f\ra\nonumber \cdots \textbf{II} 
\eeaa
Recalling the fact $a^{\ts}_{i\hat i}=a_{\hat i i}$, we have 
\beaa
\textbf{I}&=&\frac{1}{2} \la \nabla a, a^{\ts}\nabla |a^{\ts}\nabla f|^2\ra  \\
&=& \la \nabla a, (\sum_{k=1}^n a^{\ts}\nabla(a^{\ts}_k\nabla f)a^{\ts}_k\nabla f)\ra \\
&=& \la \nabla a, (\sum_{k=1}^n a^{\ts}\nabla a^{\ts}_k\nabla fa^{\ts}_k\nabla f)\ra+ \la \nabla a, \sum_{k=1}^n a^{\ts}a^{\ts}_k\nabla^2 fa^{\ts}_k\nabla f\ra,
\eeaa
and 
\beaa
\textbf{II}&=&-\la  a^{\ts} \nabla^2 a\circ  (a^{\ts}\nabla f),a^{\ts}\nabla f\ra- \sum_{k=1}^n\la \nabla a, (a^{\ts}_k \nabla (a^{\ts}\nabla f))\ra a^{\ts}_k\nabla f \\
&=& -\la  a^{\ts} \nabla^2 a\circ  (a^{\ts}\nabla f),a^{\ts}\nabla f\ra- \sum_{k=1}^n\la \nabla a, (a^{\ts}_k  a^{\ts}\nabla^2 f)\ra a^{\ts}_k\nabla f\\
&&-\sum_{k=1}^n\la \nabla a, (a^{\ts}_k \nabla a^{\ts}\nabla f) a^{\ts}_k\nabla f\ra .
\eeaa
Subtracting the above two terms, we \qi{obtain}
\beaa
\textbf{I}-\textbf{
II}&=& \la \nabla a, \sum_{k=1}^n\Big[ a^{\ts}\nabla a^{\ts}_k\nabla f-a^{\ts}_k \nabla a^{\ts}\nabla f)\Big] a^{\ts}_k\nabla f\ra -\la  a^{\ts} \nabla^2 a\circ  (a^{\ts}\nabla f),a^{\ts}\nabla f\ra.
\eeaa
We now derive the following step
\beaa
&&\frac{1}{2}\Delta_a \Gamma_{1}(f,f)-\Gamma_{1}(\Delta_a f,f)\\
&=&\frac{1}{2}\la a^{\ts}\nabla, (a^{\ts}\nabla |a^{\ts}\nabla f|^2 )\ra -\la a^{\ts} \nabla (\la a^{\ts}\nabla,a^{\ts}\nabla f\ra ),a^{\ts}\nabla f\ra \\
&&+\la \nabla a, \sum_{k=1}^n\Big[ a^{\ts}\nabla a^{\ts}_k\nabla f-a^{\ts}_k \nabla a^{\ts}\nabla f)\Big] a^{\ts}_k\nabla f\ra -\la  a^{\ts} \nabla^2 a\circ  (a^{\ts}\nabla f)\,a^{\ts}\nabla f\ra.
\eeaa
The proof is completed. \qed 
 \end{proof}
 
\begin{lemma}\label{lemma step 2}
	\beaa
	&& \frac{1}{2}\la a^{\ts}\nabla, a^{\ts}\nabla |a^{\ts}\nabla f|^2 \ra -\la a^{\ts} \nabla (\la a^{\ts}\nabla,a^{\ts}\nabla f\ra ),a^{\ts}\nabla f\ra \\
	&=&[\mathsf Q\mathsf X+\mathsf D]^{\ts}[\mathsf Q\mathsf X+\mathsf D]+2\mathsf C^{\ts}\mathsf X\\ 
	&&+ \sum_{i,k=1}^n a^{\ts}_i\nabla a^{\ts}_i\nabla  a^{\ts}_k \nabla f (a^{\ts}_k\nabla f)+\sum_{i,k=1}^n a^{\ts}_i a^{\ts}_i\nabla^2a^{\ts}_k \nabla f (a^{\ts}_k\nabla f)  \\
	&& -\sum_{i,k=1}^n a^{\ts}_k\nabla a^{\ts}_i\nabla a^{\ts}_i\nabla f(a_k^{\ts}\nabla f)-\sum_{i,k=1}^n a^{\ts}_k a^{\ts}_i \nabla^2a^{\ts}_i\nabla f(a_k^{\ts}\nabla f),
\eeaa 
where the vectors $\mathsf C$, $\mathsf D$ and the matrix $\mathsf Q$ are defined in Definition \ref{definition: F G V} and formula \eqref{matrix Q P}.
\end{lemma}
\begin{proof}[Proof of Lemma \ref{lemma step 2}]
	\beaa
	&&\frac{1}{2}\la a^{\ts}\nabla,  a^{\ts}\nabla |a^{\ts}\nabla f|^2 \ra -\la a^{\ts} \nabla (\la a^{\ts}\nabla, a^{\ts}\nabla f\ra ),a^{\ts}\nabla f\ra\\
	&=&\frac{1}{2}\sum_{i,k=1}^n(a^{\ts}_i\nabla)(a_i^{\ts}\nabla)|a^{\ts}_k\nabla f|^2-\sum_{i,k=1}^n (a^{\ts}_k\nabla)[(a^{\ts}_i\nabla)(a^{\ts}_i\nabla f)](a_k^{\ts}\nabla f)\\
	&=&\sum_{i,k=1}^na^{\ts}_i\nabla\big[ (a^{\ts}_i\nabla) (a^{\ts}_k\nabla f)(a^{\ts}_k\nabla f)]-\sum_{i,k=1}^n (a^{\ts}_k\nabla)[(a^{\ts}_i\nabla)(a^{\ts}_i\nabla f)](a_k^{\ts}\nabla f)\\
	&=&\sum_{i,k=1}^n |(a^{\ts}_i\nabla) (a^{\ts}_k\nabla f)|^2+\sum_{i,k=1}^n \la (a^{\ts}_i\nabla) (a^{\ts}_i\nabla ) (a^{\ts}_k \nabla f) (a^{\ts}_k\nabla f) -\sum_{i,k=1}^n (a^{\ts}_k\nabla)[(a^{\ts}_i\nabla)(a^{\ts}_i\nabla f)](a_k^{\ts}\nabla f)\\
	&=& \textbf{T}_1+\textbf{T}_2-\textbf{T}_3.
	\eeaa
We then expand $\textbf{T}_1$, $\textbf{T}_2$ and $\textbf{T}_3$.
	\beaa
	\textbf{T}_1&=&\sum_{i,k=1}^n |(a^{\ts}_i\nabla) (a^{\ts}_k\nabla f)|^2
	=\sum_{i,k=1}^n|a^{\ts}_ia^{\ts}_k\nabla^2f+a^{\ts}_i\nabla a^{\ts}_k\nabla f |^2 \\
	&=& \quad\sum_{i,k=1}^n \Big[|a^{\ts}_ia^{\ts}_k\nabla^2f|^2+2(a^{\ts}_ia^{\ts}_k\nabla^2f)(a^{\ts}_i\nabla a^{\ts}_k\nabla f )+|a^{\ts}_i\nabla a^{\ts}_k\nabla f |^2 \Big] \\
	&=& [\mathsf Q\mathsf X+\mathsf D]^{\ts}[\mathsf Q\mathsf X+\mathsf D],\\ 
	\textbf{T}_2&=&\sum_{i,k=1}^n \la (a^{\ts}_i\nabla) (a^{\ts}_i\nabla ) (a^{\ts}_k \nabla f) (a^{\ts}_k\nabla f)\\
	&=&\quad\sum_{i,k=1}^n a^{\ts}_ia^{\ts}_ia^{\ts}_k \nabla^3 f (a^{\ts}_k\nabla f)+ \sum_{i,k=1}^n a^{\ts}_i\nabla a^{\ts}_i\nabla  a^{\ts}_k \nabla f (a^{\ts}_k\nabla f)+\sum_{i,k=1}^n a^{\ts}_i a^{\ts}_i\nabla^2a^{\ts}_k \nabla f (a^{\ts}_k\nabla f)  \\
	&&+2\sum_{i,k=1}^n  a^{\ts}_i a^{\ts}_i\nabla a^{\ts}_k \nabla^2 f (a^{\ts}_k\nabla f)+\sum_{i,k=1}^n  a^{\ts}_i\nabla a^{\ts}_i a^{\ts}_k \nabla^2 f (a^{\ts}_k\nabla f),
	\eeaa
and 
	\beaa
\textbf{T}_3&=& \sum_{i,k=1}^n (a^{\ts}_k\nabla)[(a^{\ts}_i\nabla)(a^{\ts}_i\nabla f)](a_k^{\ts}\nabla f)\\
&=&\quad\sum_{i,k=1}^n a^{\ts}_ka^{\ts}_ia^{\ts}_i \nabla^3 f (a^{\ts}_k\nabla f)+ \sum_{i,k=1}^n a^{\ts}_k\nabla a^{\ts}_i\nabla a^{\ts}_i\nabla f(a_k^{\ts}\nabla f)+\sum_{i,k=1}^n a^{\ts}_k a^{\ts}_i \nabla^2a^{\ts}_i\nabla f(a_k^{\ts}\nabla f)\\ 
&&+\sum_{i,k=1}^n a^{\ts}_ka^{\ts}_i\nabla a^{\ts}_i \nabla^2 f(a_k^{\ts}\nabla f)+2\sum_{i,k=1}^n a^{\ts}_k\nabla a^{\ts}_i a^{\ts}_i\nabla^2 f(a_k^{\ts}\nabla f).
	\eeaa
Combing the above three terms, we have the following equality:
\beaa
&&\textbf{T}_1+\textbf{T}_2-\textbf{T}_3\\
&=&
\quad\sum_{i,k=1}^n a^{\ts}_i\nabla a^{\ts}_i\nabla  a^{\ts}_k \nabla f (a^{\ts}_k\nabla f)+\sum_{i,k=1}^n a^{\ts}_i a^{\ts}_i\nabla^2a^{\ts}_k \nabla f (a^{\ts}_k\nabla f) \\
	&&-\sum_{i,k=1}^n a^{\ts}_k\nabla a^{\ts}_i\nabla a^{\ts}_i\nabla f(a_k^{\ts}\nabla f)-\sum_{i,k=1}^n a^{\ts}_k a^{\ts}_i \nabla^2a^{\ts}_i\nabla f(a_k^{\ts}\nabla f)\\
	&&+[\mathsf Q\mathsf X+\mathsf D]^{\ts}[\mathsf Q\mathsf X+\mathsf D]+2\sum_{i,k=1}^n  a^{\ts}_i a^{\ts}_i\nabla a^{\ts}_k \nabla^2 f (a^{\ts}_k\nabla f)-2\sum_{i,k=1}^n a^{\ts}_k\nabla a^{\ts}_i a^{\ts}_i\nabla^2 f(a_k^{\ts}\nabla f).
	\eeaa
Furthermore, we investigate the last two terms in the above formula:
\beaa
&&2\sum_{i,k=1}^n  a^{\ts}_i a^{\ts}_i\nabla a^{\ts}_k \nabla^2 f (a^{\ts}_k\nabla f)-2\sum_{i,k=1}^n a^{\ts}_k\nabla a^{\ts}_i a^{\ts}_i\nabla^2 f(a_k^{\ts}\nabla f)\\
&=&2\sum_{\hat i,\hat k=1}^{n+m}\left[\sum_{i,k=1}^n\sum_{i'=1}^{n+m}\left(\la a^{\ts}_{i\hat i}a^{\ts}_{i i'} ( \frac{\partial a^{\ts}_{k\hat k}}{\partial x_{i'}}) ,(a^{\ts}\nabla)_kf\ra - \la a^{\ts}_{k i'}a^{\ts}_{i\hat k}\frac{\partial a^{\ts}_{i \hat i}}{\partial x_{i'}}   ,(a^{\ts}\nabla)_kf\ra\right)\right](\frac{\partial }{\partial x_{\hat i}}\frac{\partial f}{\partial x_{\hat k}})\\
&=&\sum_{
\hat i,\hat k=1}^n2\mathsf C_{\hat i\hat k}\frac{\partial^2 f}{\partial x_{\hat i}\partial x_{\hat k}}=2\mathsf C^{\ts}\mathsf X,
\eeaa
which completes the proof.
\qed 
\end{proof}

\subsubsection{Derivation of $\widetilde \Gamma_{2,z}(f,f)$ calculus}
As in the proof of $\widetilde \Gamma_{2}(f,f)$, we give the proofs for the key lemmas below. 
\begin{lemma}\label{lemma z step 0}
\beaa  \widetilde  \Gamma_2^z(f,f) &=& \frac{1}{2}\Delta_a\Gamma_{1}^z(f,f)-\Gamma_{1}^z(\Delta_a f,f)\\
&&+\sum_{i=1}^n\sum_{\hat k=1}^{n+m}\Big[ (aa^{\ts} \nabla\log \pi)_{\hat k} \nabla_{\hat k} z^{\ts}_i\nabla f-z_i^{\ts}\nabla (aa^{\ts}\nabla\log \pi)_{\hat k} \nabla_{\hat k}f\Big]z^{\ts}_i\nabla f.
\eeaa
\end{lemma}

\begin{proof}[Proof of Lemma \ref{lemma z step 0}] From our definition, one can get  
	\beaa  \widetilde  \Gamma_{2,z}(f,f) &=& \frac{1}{2}\Delta_a\Gamma_{1}^z-\Gamma_{1}^z(\Delta_a f,f)+\frac{1}{2}\la aa^{\ts}\nabla\log \pi, \nabla \Gamma_{1}^z(f,f)\ra-\Gamma_{1}^z(\la aa^{\ts}\nabla\log\pi, \nabla f\ra,f),
	\eeaa
where we take $b=aa^{\ts}\nabla\log \pi$. We obtain
\beaa
\frac{1}{2}b\nabla \Gamma_{1}^z(f,f)&=&\frac{1}{2}\sum_{\hat k=1}^{n+m}b_{\hat k}\frac{\pa}{\pa x_{\hat k}}(\la z^T\nabla f,z^T\nabla f\ra)\\
&=&\sum_{\hat k,\hat i=1}^{n+m}\sum_{i=1}^m(b_{\hat k}\frac{\pa z^T_{i\hat i}}{\pa x_{\hat k}}\frac{\pa f}{\pa x_{\hat i}}+b_{\hat k}z^T_{i\hat i}\frac{\pa^2f}{\pa x_{\hat k}\pa x_{\hat i}})(z^T\nabla f)_i,
\eeaa
and
\beaa
-\Gamma_{1}^z(b\nabla f,f)&=&-\la z^T\nabla(b\nabla f),z^T\nabla f\ra \\
&=&-\sum_{i=1}^m\sum_{\hat k,\hat i=1}^{n+m}(z^{\ts}_{i\hat i}\frac{\pa b_{\hat k}}{\pa x_{\hat i}}\frac{\pa f}{\pa x_{\hat k}} +z^{\ts}_{i\hat i}b_{\hat k}\frac{\pa^2f}{\pa x_{\hat i}\pa x_{\hat k}}) (a^T\nabla f)_i.
\eeaa
Summing over all the terms, we have 
	\beaa
\widetilde	\Gamma_2(f,f)
&=& \frac{1}{2}\Delta_a\Gamma_{1}^z-\Gamma_{1}^z(\Delta_a f,f)\\
&&+\sum_{i=1}^m\sum_{\hat k=1}^{n+m}\Big[ (aa^{\ts} \nabla\log \pi)_{\hat k} \nabla_{\hat k} z^{\ts}_i\nabla f-z_i^{\ts}\nabla (aa^{\ts}\nabla\log \pi)_{\hat k} \nabla_{\hat k}f\Big]z^{\ts}_i\nabla f.
	\eeaa
	\qed
\end{proof}
{
\begin{lemma}\label{lemma z step 1}
	\bea\label{gamma 2 relation}
	&&\frac{1}{2}\Delta_a\Gamma_{1}^z(f,f)-\Gamma_{1}^z(\Delta_a f,f)\\
	&=&\frac{1}{2}\la a^{\ts}\nabla, (a^{\ts}\nabla |z^{\ts}\nabla f|^2 \ra -\la z^{\ts} \nabla (\la a^{\ts}\nabla, a^{\ts}\nabla f\ra ),z^{\ts}\nabla f\ra\nonumber \\
	&&+\la \nabla a, \sum_{k=1}^m\Big[ a^{\ts}\nabla z^{\ts}_k\nabla f-z^{\ts}_k \nabla a^{\ts}\nabla f)\Big] z^{\ts}_k\nabla f\ra -\la  z^{\ts} \nabla^2 a\circ  (a^{\ts}\nabla f),z^{\ts}\nabla f\ra.\nonumber
	\eea
	\end{lemma}}

\begin{proof}[Proof of Lemma \ref{lemma z step 1} ] 
According to our definition, we have 
 \beaa\label{gamma 2 --1}
 \Delta_a\Gamma_{1}^z (f,f)&=&\nabla\cdot (aa^{\ts}|z^{\ts}\nabla f|^2)\\
 &=& \la \nabla a,a^{\ts}\nabla |z^{\ts}\nabla f|^2\ra +\la a^{\ts}\nabla,  a^{\ts}\nabla |z^{\ts}\nabla f|^2 \ra .
 \eeaa
Recall that, we have 
\beaa
\nabla\cdot (aa^{\ts} \nabla f)&=&\la \nabla a,  a^{\ts}\nabla f\ra +\la a^{\ts}\nabla,  a^{\ts}\nabla f\ra .
\eeaa
According to the definition of $\Gamma_{1}^z$, we have  
\beaa
\Gamma_{1}^z(\Delta_af,f)&=&\la z^{\ts} \nabla (\nabla\cdot (aa^{\ts}\nabla f) ),z^{\ts}\nabla f\ra, \\
&=& \la z^{\ts} \nabla \left( \la \nabla a, (a^{\ts}\nabla f)\ra \right),z^{\ts}\nabla f\ra +\la z^{\ts} \nabla (\la a^{\ts}\nabla ,a^{\ts}\nabla f\ra ),z^{\ts}\nabla f\ra\nonumber\\
&=& \la  z^{\ts} \nabla^2 a\circ  (a^{\ts}\nabla f),z^{\ts}\nabla f\ra+ \la \nabla a\circ (z^{\ts} \nabla (a^{\ts}\nabla f)),z^{\ts}\nabla f\ra\nonumber\\
&&+\la z^{\ts} \nabla (\la a^{\ts}\nabla,a^{\ts}\nabla f\ra ),z^{\ts}\nabla f\ra.
\eeaa
Combining the above terms, we get  
\bea\label{I-II z 1}
&&\frac{1}{2}\Delta_a \Gamma_{1}^z (f,f)-\Gamma_{1}^z (\Delta_a f,f)\nonumber \\
&=&\frac{1}{2}\la a^{\ts}\nabla,a^{\ts}\nabla |z^{\ts}\nabla f|^2 \ra -\la z^{\ts} \nabla (\la a^{\ts}\nabla, a^{\ts}\nabla f\ra ),z^{\ts}\nabla f\ra+\underbrace{\frac{1}{2} \la \nabla a , a^{\ts}\nabla |z^{\ts}\nabla f|^2\ra }_{ \textbf{I}_z}\nonumber \\
&&-\la  z^{\ts} \nabla^2 a\circ  (a^{\ts}\nabla f),z^{\ts}\nabla f\ra-\la \nabla a\circ (z^{\ts} \nabla (a^{\ts}\nabla f)),z^{\ts}\nabla f\ra \cdots \textbf{II}_z.
\eea
Recalling the fact $a^{\ts}_{i\hat i}=a_{\hat i i}$, we have 
\beaa
\textbf{I}_z&=&\frac{1}{2} \la \nabla a , a^{\ts}\nabla |z^{\ts}\nabla f|^2 \ra \\
&=& \la \nabla a, \sum_{k=1}^m a^{\ts}\nabla(z^{\ts}_k\nabla f)z^{\ts}_k\nabla f\ra \\
&=&  \la \nabla a, \sum_{k=1}^m a^{\ts}\nabla z^{\ts}_k\nabla fz^{\ts}_k\nabla f\ra +  \la \nabla a, \sum_{k=1}^m a^{\ts}z^{\ts}_k\nabla^2 fz^{\ts}_k\nabla f\ra ,
\eeaa
and 
\beaa
\textbf{II}_z&=&-\la  \left(z^{\ts} \nabla^2 a\circ  (a^{\ts}\nabla f)\right),z^{\ts}\nabla f\ra- \sum_{k=1}^m\la \nabla a,z^{\ts}_k \nabla (a^{\ts}\nabla f)\ra z^{\ts}_k\nabla f \\
&=& -\la  \la z^{\ts} \nabla^2 a\circ  (a^{\ts}\nabla f)\ra ,z^{\ts}\nabla f\ra- \sum_{k=1}^m\la \nabla a, (z^{\ts}_k  a^{\ts}\nabla^2 f)\ra z^{\ts}_k\nabla f\\
&&-\sum_{k=1}^m\la \nabla a, (z^{\ts}_k \nabla a^{\ts}\nabla f)z^{\ts}_k\nabla f\ra.
\eeaa
Subtracting the above two terms, we \qi{obtain}
\bea\label{I-II z 2}
\textbf{I}_z-\textbf{
II}_z&=& \la \nabla a,\sum_{k=1}^m\Big[ a^{\ts}\nabla z^{\ts}_k\nabla f-z^{\ts}_k \nabla a^{\ts}\nabla f)\Big] z^{\ts}_k\nabla f\ra -\la  z^{\ts} \nabla^2 a\circ  (a^{\ts}\nabla f),z^{\ts}\nabla f\ra.\nonumber\\
\eea
Plugging \eqref{I-II z 2} into \eqref{I-II z 1}, we get
\beaa
&&\frac{1}{2}\Delta_a \Gamma_{1}^z(f,f)-\Gamma_{1}^z(\Delta_a f,f)\\
&=&\frac{1}{2}\la a^{\ts}\nabla, a^{\ts}\nabla |z^{\ts}\nabla f|^2 \ra -\la z^{\ts} \nabla ( \la a^{\ts}\nabla ,a^{\ts}\nabla f\ra ),z^{\ts}\nabla f\ra \\
&&+\la \nabla a,\sum_{k=1}^m\Big[ a^{\ts}\nabla z^{\ts}_k\nabla f-z^{\ts}_k \nabla a^{\ts}\nabla f)\Big] a^{\ts}_k\nabla f\ra -\la  \left(z^{\ts} \nabla^2 a\circ  (a^{\ts}\nabla f)\right),z^{\ts}\nabla f\ra.
\eeaa
The proof is completed. \qed 
 \end{proof}

We further investigate the extra term explicitly in Lemma \ref{lemma z step 1}. 
\begin{lemma}\label{lemma z step 2}
	\beaa
	&& \frac{1}{2}\la a^{\ts}\nabla, a^{\ts}\nabla |z^{\ts}\nabla f|^2 \ra -\la z^{\ts} \nabla (\la a^{\ts}\nabla,a^{\ts}\nabla f\ra ),z^{\ts}\nabla f\ra \\
	&=&[\mathsf P\mathsf X+\mathsf E]^{\ts}[\mathsf P\mathsf X+\mathsf E]+2\mathsf F^{\ts}\mathsf X\\
	&&+ \sum_{i=1}^n\sum_{k=1}^m a^{\ts}_i\nabla a^{\ts}_i\nabla  z^{\ts}_k \nabla f (z^{\ts}_k\nabla f)+\sum_{i=1}^n\sum_{k=1}^m a^{\ts}_i a^{\ts}_i\nabla^2z^{\ts}_k \nabla f (z^{\ts}_k\nabla f)  \\
	&& -\sum_{i=1}^n\sum_{k=1}^m z^{\ts}_k\nabla a^{\ts}_i\nabla a^{\ts}_i\nabla f(z_k^{\ts}\nabla f)-\sum_{i=1}^n\sum_{k=1}^m z^{\ts}_k a^{\ts}_i \nabla^2a^{\ts}_i\nabla f(z_k^{\ts}\nabla f),
\eeaa 
where vectors $\mathsf F$, $\mathsf E$, and the matrix $\mathsf P$ are defined in Definition \ref{definition: F G V} and formula \eqref{matrix Q P}.
\end{lemma}
\begin{proof}[Proof of Lemma \ref{lemma z step 2}]  
	\beaa
	&&\frac{1}{2}\la a^{\ts}\nabla,  a^{\ts}\nabla |z^{\ts}\nabla f|^2 \ra -\la z^{\ts} \nabla (\la a^{\ts}\nabla,a^{\ts}\nabla f)\ra ),z^{\ts}\nabla f\ra\\
	&=&\frac{1}{2}\sum_{i=1}^n\sum_{k=1}^m (a^{\ts}_i\nabla)(a_i^{\ts}\nabla)|z^{\ts}_k\nabla f|^2-\sum_{i=1}^n\sum_{k=1}^m (z^{\ts}_k\nabla)[(a^{\ts}_i\nabla)(a^{\ts}_i\nabla f)](z_k^{\ts}\nabla f)\\
	&=&\sum_{i=1}^n\sum_{k=1}^m a^{\ts}_i\nabla\big[\Big( (a^{\ts}_i\nabla) (z^{\ts}_k\nabla f)\Big)(z^{\ts}_k\nabla f)]-\sum_{i=1}^n\sum_{k=1}^m (z^{\ts}_k\nabla)[(a^{\ts}_i\nabla)(a^{\ts}_i\nabla f)](z_k^{\ts}\nabla f)\\
	&=&\sum_{i=1}^n \sum_{k=1}^m |(a^{\ts}_i\nabla) (z^{\ts}_k\nabla f)|^2+\sum_{i=1}^n\sum_{k=1}^m \la (a^{\ts}_i\nabla) (a^{\ts}_i\nabla ) (z^{\ts}_k \nabla f) (z^{\ts}_k\nabla f)\\
	&& -\sum_{i=1}^n\sum_{k=1}^m (z^{\ts}_k\nabla)[(a^{\ts}_i\nabla)(a^{\ts}_i\nabla f)](z_k^{\ts}\nabla f)\\
	&=& \textbf{T}_1^z+\textbf{T}_2^z-\textbf{T}_3^z.
	\eeaa
We then expand $\textbf{T}_1^z$, $\textbf{T}_2^z$ and $\textbf{T}_3^z$.
	\beaa
	\textbf{T}_1^z&=&\sum_{i=1}^n\sum_{k=1}^m |(a^{\ts}_i\nabla) (z^{\ts}_k\nabla f)|^2
	=\sum_{i=1}^n \sum_{k=1}^m |a^{\ts}_iz^{\ts}_k\nabla^2f+a^{\ts}_i\nabla z^{\ts}_k\nabla f |^2 \\
	&=& \quad\sum_{i=1}^n\sum_{k=1}^m \Big[|a^{\ts}_iz^{\ts}_k\nabla^2f|^2+2(a^{\ts}_iz^{\ts}_k\nabla^2f)(a^{\ts}_i\nabla z^{\ts}_k\nabla f )+|a^{\ts}_i\nabla z^{\ts}_k\nabla f |^2 \Big] \\
	&=& [\mathsf P\mathsf X+\mathsf E]^{\ts}[\mathsf P\mathsf X+\mathsf E],\\ 
	\textbf{T}_2^z&=&\sum_{i=1}^n\sum_{k=1}^m \la (a^{\ts}_i\nabla) (a^{\ts}_i\nabla ) (z^{\ts}_k \nabla f) (z^{\ts}_k\nabla f)\\
	&=&\quad\sum_{i=1}^n\sum_{k=1}^m a^{\ts}_ia^{\ts}_iz^{\ts}_k \nabla^3 f (z^{\ts}_k\nabla f)+ \sum_{i=1}^n\sum_{k=1}^m a^{\ts}_i\nabla a^{\ts}_i\nabla  z^{\ts}_k \nabla f (z^{\ts}_k\nabla f)\\
	&&+\sum_{i=1}^n\sum_{k=1}^m a^{\ts}_i a^{\ts}_i\nabla^2z^{\ts}_k \nabla f (z^{\ts}_k\nabla f)  
+2\sum_{i=1}^n\sum_{k=1}^m  a^{\ts}_i a^{\ts}_i\nabla z^{\ts}_k \nabla^2 f (z^{\ts}_k\nabla f)
+\sum_{i=1}^n\sum_{k=1}^m  a^{\ts}_i\nabla a^{\ts}_i z^{\ts}_k \nabla^2 f (z^{\ts}_k\nabla f),
	\eeaa
and 
	\beaa
\textbf{T}_3^z&=& \sum_{i=1}^n\sum_{k=1}^m (z^{\ts}_k\nabla)[(a^{\ts}_i\nabla)(a^{\ts}_i\nabla f)](z_k^{\ts}\nabla f)\\
&=&\quad\sum_{i=1}^n\sum_{k=1}^m z^{\ts}_ka^{\ts}_ia^{\ts}_i \nabla^3 f (z^{\ts}_k\nabla f)+ \sum_{i=1}^n \sum_{k=1}^m z^{\ts}_k\nabla a^{\ts}_i\nabla a^{\ts}_i\nabla f(z_k^{\ts}\nabla f)\\
\eeaa 
\beaa 
&&+\sum_{i=1}^n\sum_{k=1}^m z^{\ts}_k a^{\ts}_i \nabla^2a^{\ts}_i\nabla f(z_k^{\ts}\nabla f)
+\sum_{i=1}^n\sum_{k=1}^m z^{\ts}_ka^{\ts}_i\nabla a^{\ts}_i \nabla^2 f(z_k^{\ts}\nabla f)\\&&+2\sum_{i=1}^n\sum_{k=1}^m z^{\ts}_k\nabla a^{\ts}_i a^{\ts}_i\nabla^2 f(z_k^{\ts}\nabla f).
	\eeaa
Combing the above three terms, we have the following equality:
\beaa
&&\textbf{T}_1^z+\textbf{T}_2^z-\textbf{T}_3^z\\
&=&[\mathsf P\mathsf X+\mathsf E]^{\ts}[\mathsf P\mathsf X+\mathsf E]\\
&&+\sum_{i=1}^n\sum_{k=1}^m a^{\ts}_i\nabla a^{\ts}_i\nabla  z^{\ts}_k \nabla f (z^{\ts}_k\nabla f)+\sum_{i=1}^n\sum_{k=1}^m a^{\ts}_i a^{\ts}_i\nabla^2z^{\ts}_k \nabla f (z^{\ts}_k\nabla f) \\
	&&-\sum_{i=1}^n\sum_{k=1}^m z^{\ts}_k\nabla a^{\ts}_i\nabla a^{\ts}_i\nabla f(z_k^{\ts}\nabla f)-\sum_{i=1}^n\sum_{k=1}^m z^{\ts}_k a^{\ts}_i \nabla^2a^{\ts}_i\nabla f(z_k^{\ts}\nabla f)\\
	&&+2\sum_{i=1}^n\sum_{k=1}^m  a^{\ts}_i a^{\ts}_i\nabla z^{\ts}_k \nabla^2 f (z^{\ts}_k\nabla f)-2\sum_{i=1}^n\sum_{k=1}^m z^{\ts}_k\nabla a^{\ts}_i a^{\ts}_i\nabla^2 f(z_k^{\ts}\nabla f).
	\eeaa
Furthermore, we investigate the last two terms in the above formula:
\beaa
&&2\sum_{i=1}^n\sum_{k=1}^m  a^{\ts}_i a^{\ts}_i\nabla z^{\ts}_k \nabla^2 f (z^{\ts}_k\nabla f)-2\sum_{i=1}^n\sum_{k=1}^m z^{\ts}_k\nabla a^{\ts}_i a^{\ts}_i\nabla^2 f(z_k^{\ts}\nabla f)\\
&=&2\sum_{\hat i,\hat k=1}^{n+m}\left[\sum_{i=1}^n\sum_{k=1}^m\sum_{i'=1}^{n+m}\left( a^{\ts}_{i\hat i}a^{\ts}_{i i'} ( \frac{\partial z^{\ts}_{k\hat k}}{\partial x_{i'}})(z^{\ts}\nabla)_kf -  z^{\ts}_{k i'}a^{\ts}_{i\hat k}\frac{\partial a^{\ts}_{i \hat i}}{\partial x_{i'}} (z^{\ts}\nabla)_kf\right)\right](\frac{\partial }{\partial x_{\hat i}}\frac{\partial f}{\partial x_{\hat k}})\\
&=&\sum_{
\hat i,\hat k=1}^n2\mathsf F_{\hat i\hat k}\frac{\partial^2 f}{\partial x_{\hat i}\partial x_{\hat k}}=2\mathsf F^{\ts}\mathsf X,
\eeaa
which completes the proof.
\qed 
\end{proof}

\begin{remark}
The proof of the identity 
	\beaa
	\div^{\pi}_z(\Gamma_{\nabla(aa^{\ts})}f,f )-\div^{\pi}_a(\Gamma_{\nabla(zz^{\ts})}f,f )&=&\mathfrak{R}_{\pi}(\nabla f,\nabla f)+2\mathsf G^{\ts}\mathsf X,
	\eeaa
is the same as \qi{the one} in \cite{FL}[Lemma 10]. We skip the proof here. 
\end{remark}

\subsection{Proof of Theorem \ref{thm: information bochner}}
Combining all the above propositions, we are now ready to prove an {\em information Bochner's formula}. 

\begin{proof}[Proof of Theorem \ref{thm: information bochner}]
We show the proof in the following three steps. 

\noindent \textbf{Step A:} According to the above three propositions, we have 
	\beaa
&&\int \Big[ \widetilde \Gamma_{2}(f,f)+\widetilde \Gamma_2^{z,\pi}(f,f)+\Gamma_{\mathcal I_{a,z}}(f,f)\Big] pdx\\
&=& \int \Big[ \widetilde \Gamma_{2}(f,f)+\widetilde \Gamma_2^{z,\pi}(f,f)+(\beta+(1-\beta))\Gamma_{\mathcal I_{a}}(f,f)+\Gamma_{\mathcal I_{z}}(f,f)\Big] pdx\\
&=&\int \Big[(\mathsf Q\mathsf X+\mathsf D)^{\ts}(\mathsf Q\mathsf X+\mathsf D)+ 2\mathsf C^{\ts}\mathsf X+ \mathfrak R_{a}(\nabla f,\nabla f)\Big]pdx\\
&&+\int \Big[(\mathsf P\mathsf X+\mathsf E)^{\ts}(\mathsf P\mathsf X+\mathsf E)+2\mathsf F^{\ts}\mathsf X+\mathfrak R_{z}(\nabla f,\nabla f)+\mathfrak{R}_{\pi}(\nabla f,\nabla f)+2\mathsf G^{\ts}\mathsf X\Big]pdx \\ 
&&+\int \Big[\beta( \mathfrak R_{\mathcal I_a}(\nabla f,\nabla f)+2\mathsf V_a^{\ts}\mathsf X) +(1-\beta)\mathfrak R_{\gamma_a}(\nabla f,\nabla f)+\mathfrak R_{\gamma_z}(\nabla f,\nabla f)\Big]pdx.
\eeaa 

\noindent\textbf{Step B:} Under Assumption \ref{prop:main condition}, we show that there exists vectors $\Lambda_1, \Lambda_2\in \mathbb{R}^{(n+m)^2\times 1}$ such that, for any $\beta \in\mathbb R^1$, 
	\bea\label{technical condition}
	(\mathsf Q^{\ts}\Lambda_1+\mathsf P^{\ts} \Lambda_2)^{\ts}\mathsf X&=&(\mathsf F+\mathsf C+\mathsf G+\beta \mathsf V^a+\mathsf Q^{\ts}\mathsf D+\mathsf P^{\ts}\mathsf E)^{\ts}\mathsf X.
	\eea
The main idea is to complete squares for all the second-order terms listed below:
\beaa 
&&(\mathsf Q\mathsf X+\mathsf D)^{\ts}(\mathsf Q\mathsf X+\mathsf D)+(\mathsf P\mathsf X+\mathsf E)^{\ts}(\mathsf P\mathsf X+\mathsf E)+ 2\mathsf C^{\ts}\mathsf X+2\mathsf F^{\ts}\mathsf X+2\mathsf G^{\ts}\mathsf X+2\mathsf V_a^{\ts}\mathsf X,
\eeaa
where the leading terms are given by 
\beaa 
\mathsf X^{\ts} \mathsf Q^{\ts}\mathsf Q\mathsf X=\sum_{i,k=1}^n |a^{\ts}_ia^{\ts}_k\nabla^2f|^2,\quad \mathsf X^{\ts} \mathsf P^{\ts}\mathsf P\mathsf X= \sum_{i=1}^n\sum_{j=1}^m |a^{\ts}_iz^{\ts}_j\nabla^2f|^2.
\eeaa 
It is enough to show that the following claim holds. 

\noindent\textbf{Claim 1:}
\bea\label{condition C}
 \mathsf C^{\ts}\mathsf X=\sum_{i,k=1}^n (a^{\ts}_ia^{\ts}_k\nabla^2f)\widetilde{\mathsf C}_1(\mathsf U)_{ik}+\sum_{i=1}^n\sum_{k=1}^m (a^{\ts}_iz^{\ts}_k\nabla^2f)\widetilde{\mathsf C}_2(\mathsf U)_{ik}.
\eea 
We show the details for $\mathsf C^{\ts}\mathsf X$ as below, and skip the details for the similar terms in $ \mathsf F^{\ts}\mathsf X, \mathsf G^{\ts}\mathsf X, \mathsf V^a\mathsf X $. We have 
\beaa 
\mathsf C^{\ts}\mathsf X=\sum_{\hat i,\hat k=1}^{n+m}\mathsf C_{\hat i\hat k}\frac{\pa^2 f}{\pa x_{\hat i}\pa x_{\hat k}} = \sum_{\hat i,\hat k=1}^{n+m} \sum_{i,k=1}^n\Big( a^{\ts}_{i\hat i}a^{\ts}_{i}\nabla a^{\ts}_{k\hat k}  -a^{\ts}_{i\hat k}a^{\ts}_{k}\nabla a^{\ts}_{i\hat i} \Big) a^{\ts}_k \mathsf{U}\frac{\pa^2 f}{\pa x_{\hat i}\pa x_{\hat k}}.
\eeaa 
We look at the first term, $a^{\ts}_i\nabla a^{\ts}_{k\hat k}=\sum_{i'=1}^{n+m}a^{\ts}_{ii'}\nabla_{i'}a^{\ts}_{k\hat k}$, and apply Definition \ref{def: curvature sum}. Hence we derive the following expression by 
\beaa 
&&\sum_{\hat i,\hat k=1}^{n+m} \sum_{i,k=1}^n\Big( a^{\ts}_{i\hat i}a^{\ts}_{i}\nabla a^{\ts}_{k\hat k}  \Big) a^{\ts}_k \mathsf{U}\frac{\pa^2 f}{\pa x_{\hat i}\pa x_{\hat k}}\\
&=& \sum_{\hat i,i',\hat k=1}^{n+m} \sum_{i,k,l=1}^n \Big( a^{\ts}_{i\hat i}a^{\ts}_{ii'} \lambda^{i'k}_l a^{\ts}_{l\hat k}  \Big) a^{\ts}_k \mathsf{U}\frac{\pa^2 f}{\pa x_{\hat i}\pa x_{\hat k}}+\sum_{\hat i,i',\hat k=1}^{n+m} \sum_{i,k=1}^n \sum_{l=1}^{m}\Big( a^{\ts}_{i\hat i}a^{\ts}_{ii'} \lambda^{i'k}_{l+n} z^{\ts}_{l\hat k}  \Big) a^{\ts}_k \mathsf{U}\frac{\pa^2 f}{\pa x_{\hat i}\pa x_{\hat k}} \\
&=& \sum_{i,l=1}^n (a^{\ts}_{i}a^{\ts}_{l} \nabla^2 f)\Big(\sum_{k=1}^n\sum_{i'=1}^{n+m} a^{\ts}_{ii'}\lambda^{i'k}_la^{\ts}_k\mathsf U \Big)+\sum_{i=1}^n\sum_{l=1}^m (a^{\ts}_iz^{\ts}_l\nabla^2 f)\Big(\sum_{k=1}^n\sum_{i'=1}^{n+m} a^{\ts}_{ii'} \lambda ^{i'k}_{l+n}a^{\ts}_k\mathsf U
\Big).
\eeaa 
We have a similar expansion for the second term of $\mathsf C^{\ts}\mathsf X$. Thus   
\beaa 
\widetilde{\mathsf C}_1(\mathsf U)_{il}&=&\sum_{k=1}^n[\sum_{i'=1}^{n+m} a^{\ts}_{ii'}\lambda^{i'k}_l-\sum_{k'=1}^{n+m}a^{\ts}_{kk'} \lambda^{k'i}_l ]a^{\ts}_k\mathsf U,\\
\widetilde{\mathsf C}_2(\mathsf U)_{il}&=&\sum_{k=1}^n[ \sum_{i'=1}^{n+m} a^{\ts}_{ii'} \lambda ^{i'k}_{l+n}-\sum_{k'=1}^{n+m}a^{\ts}_{kk'}\lambda^{k'i}_{l+n}]a^{\ts}_k\mathsf U,
\eeaa 
which proves {Claim 1} in \eqref{condition C}. From \qi{the} observation, we can show the same property for vectors $\mathsf F$, $\mathsf V^a$, and the last three terms of vector $\mathsf G$. We now need to further look at the first term of vector $\mathsf G$ as defined in Definition \ref{def: curvature sum}:
\beaa
&&\text{First Term of}~\mathsf G^{\ts}\mathsf X\\
&=& \sum_{i=1}^n\sum_{k=1}^m\left( z^{\ts}_{k \hat k} z^{\ts}_{k}\nabla a^{\ts}_{i\hat i}a^{\ts}_{i} \mathsf U\right)\frac{\pa^2 f}{\pa x_{\hat i}\pa x_{\hat k}}\\
&=&  \sum_{i,l=1}^n\sum_{k=1}^m\sum_{k'=1}^{n+m} \left( z^{\ts}_{k \hat k} z^{\ts}_{kk'}\lambda^{k'i}_l a^{\ts}_{l\hat i}a^{\ts}_{i} \mathsf U\right)\frac{\pa^2 f}{\pa x_{\hat i}\pa x_{\hat k}}+ \sum_{i=1}^n\sum_{k,l=1}^m\sum_{k'=1}^{n+m} \left( z^{\ts}_{k \hat k} z^{\ts}_{kk'}\lambda^{k'i}_{l+n} z^{\ts}_{l\hat i}a^{\ts}_{i} \mathsf U\right)\frac{\pa^2 f}{\pa x_{\hat i}\pa x_{\hat k}}\\
&=& \sum_{k=1}^m\sum_{l=1}^n(z^{\ts}_k a^{\ts}_l\nabla^2f)\Big(\sum_{i=1}^n\sum_{k'=1}^{n+m}z^{\ts}_{kk'}\lambda^{k'i}_la^{\ts}_i\mathsf U \Big) +\sum_{k,l=1}^m(z^{\ts}_kz^{\ts}_l\nabla^2 f)\Big(\sum_{i=1}^n\sum_{k'=1}^{n+m} z^{\ts}_{kk'}\lambda^{k'i}_{l+n} a^{\ts}_i\mathsf U\Big).
\eeaa 
If Assumption \ref{prop:main condition} holds, we get
$z^{\ts}_{k}\nabla a^{\ts}_{i\hat i}=\sum_{k'=1}^{n+m}\sum_{l=1}^nz^{\ts}_{kk'} \lambda^{k'i}_{l}a^{\ts}_{l\hat i}$, which means $\lambda^{k'i}_{l+n}=0$, for $l=1,\cdots,m$.  Thus we conclude the result by the following claim.

\noindent\textbf{Claim 2:}
\beaa
[\mathsf C +\mathsf F+ \mathsf G^{\ts}+\beta \mathsf V^a]\mathsf X&=&\sum_{i,k=1}^n (a^{\ts}_ia^{\ts}_k\nabla^2f)\widetilde{\mathsf \Lambda}_1(\mathsf U)_{ik}+\sum_{i=1}^n\sum_{k=1}^m (a^{\ts}_iz^{\ts}_k\nabla^2f)\widetilde{\mathsf \Lambda}_2(\mathsf U)_{ik}\\
&=&\widetilde \Lambda_1(\mathsf U)^{\ts}\mathsf Q\mathsf X+\widetilde \Lambda_2(\mathsf U)^{\ts}\mathsf P\mathsf X,
\eeaa 
where
\beaa 
\widetilde{\mathsf \Lambda}_1(\mathsf U)&=&\widetilde{\mathsf C}_1(\mathsf U)+\widetilde{\mathsf F}_1(\mathsf U)+\widetilde{\mathsf G}_1(\mathsf U)+\beta \widetilde{\mathsf V}^a_1(\mathsf U) \in \mathbb R^{n^2\times 1},\\ 
\widetilde{\mathsf \Lambda}_2(\mathsf U)&=& \widetilde{\mathsf C}_2(\mathsf U)+\widetilde{\mathsf F}_2(\mathsf U)+\widetilde{\mathsf G}_2(\mathsf U)+\beta \widetilde{\mathsf V}^a_2(\mathsf U)\in \mathbb R^{({nm})\times 1},
\eeaa 
and 
\beaa 
\widetilde{\mathsf F}_1(\mathsf U)_{il}&=& \sum_{k=1}^m\Big(\sum_{i'=1}^{n+m}  a^{\ts}_{ii'}\omega^{i'k}_l -\sum_{k'=1}^{n+m} z^{\ts}_{kk'}\lambda^{k'i}_l\Big)z^{\ts}_k\mathsf U, \\
\widetilde{\mathsf F}_2(\mathsf U)_{il}&=&\sum_{k=1}^m\Big(\sum_{i'=1}^{n+m}  a^{\ts}_{ii'}\omega^{i'k}_{l+n} -\sum_{k'=1}^{n+m} z^{\ts}_{kk'} \lambda_{l+n}^{k'i}\Big)z^{\ts}_k\mathsf U,\quad 
\widetilde{\mathsf G}_1(U)_{il}=- \sum_{k=1}^m \sum_{i'=1}^{n+m} a^{\ts}_{ii'} \omega^{i'k}_l  z^{\ts}_k\mathsf U,\\ 
\widetilde{\mathsf G}_2(U)_{lk}&=& \sum_{i=1}^n\sum_{k'=1}^{n+m}  z^{\ts}_{kk'} \lambda^{k'i}_l a^{\ts}_{i} \mathsf U+ z^{\ts}_{k}\nabla a^{\ts}_l \mathsf U-\sum_{i=1}^m\sum_{i'=1}^{n+m}a^{\ts}_{li'} \omega^{i'i}_{k+n}  z^{\ts}_i\mathsf U -a^{\ts}_{l}\nabla z^{\ts}_{k} \mathsf U,\\ 
\widetilde{\mathsf V}^a_1(\mathsf U)_{il}&=& -\frac{1}{2} \alpha_l (a^{\ts}_{ i}\mathsf U)+\frac{1}{2}   \la \mathsf U,\gamma \ra\mathbf{1}_{\{i=l\}},\quad 
\widetilde{\mathsf V}^a_2(\mathsf U)_{il}= -\frac{1}{2} \alpha_{l+n}(a^{\ts}_{ i}\mathsf U).
\eeaa 
We thus have 
\beaa 
&&(\mathsf Q\mathsf X+\mathsf D)^{\ts}(\mathsf Q\mathsf X+\mathsf D)+(\mathsf P\mathsf X+\mathsf E)^{\ts}(\mathsf P\mathsf X+\mathsf E)+ 2\mathsf C^{\ts}\mathsf X+2\mathsf F^{\ts}\mathsf X+2\mathsf G^{\ts}\mathsf X+2\mathsf V_a^{\ts}\mathsf X\\
&=& [\mathsf Q\mathsf X+\widetilde \Lambda_1(\mathsf U)+\mathsf D]^{\ts} [\mathsf Q\mathsf X+\widetilde \Lambda_1(\mathsf U)+\mathsf D]+[\mathsf P\mathsf X+\widetilde \Lambda_2(\mathsf U)+\mathsf E]^{\ts}[\mathsf P\mathsf X+\widetilde \Lambda_2(\mathsf U)+\mathsf E]\\
&&- [\widetilde \Lambda_1(\mathsf U)+\mathsf D]^{\ts}[\widetilde \Lambda_1(\mathsf U)+\mathsf D]-[\widetilde \Lambda_2(\mathsf U)+\mathsf E]^{\ts}[\widetilde \Lambda_2(\mathsf U)+\mathsf E]+\mathsf D^{\ts}\mathsf D+\mathsf E^{\ts}\mathsf E\\
&=& [\mathsf Q\mathsf X+\Lambda_1]^{\ts} [\mathsf Q\mathsf X+\Lambda_1]+[\mathsf P\mathsf X+\Lambda_2]^{\ts} [\mathsf P\mathsf X+\Lambda_2]-\Lambda_1^{\ts}\Lambda_1-\Lambda_2^{\ts}\Lambda_2+\mathsf D^{\ts}\mathsf D+\mathsf E^{\ts}\mathsf E,
\eeaa
where 
\beaa 
\Lambda_1= \widetilde \Lambda_1+\mathsf D \in \mathbb R^{n^2\times 1},\quad \Lambda_2= \widetilde \Lambda_2+\mathsf E\in \mathbb R^{({nm})\times 1}.
\eeaa

\noindent\textbf{Step C:} Given the condition \eqref{technical condition} in \textbf{Step B}, we complete the proof by
\beaa
&&\int \Big[ \widetilde \Gamma_{2}(f,f)+\widetilde \Gamma_2^{z,\pi}(f,f)+\Gamma_{\mathcal I_{a,z}}(f,f)\Big] pdx\\
&=&\int \Big[[\mathsf Q\mathsf X+\Lambda_1]^{\ts} [\mathsf Q\mathsf X+\Lambda_1]+[\mathsf P\mathsf X+\Lambda_2]^{\ts} [\mathsf P\mathsf X+\Lambda_2]+\mathfrak{R}(\nabla f,\nabla f)\Big]pdx.
\eeaa
\qed 
\end{proof}

\subsection{Proof of Step 3: entropy dissipation}
We last prove Step $3$. It is a modified Lyaponuv method in probability density space, which is a standard approach; see also \cite{AE, FL1}. We first derive the following dissipation result. 
\begin{proposition}\label{prop4}
Along with Fokker-Planck equation \eqref{FPE}, the following equality holds
\begin{equation*}
\frac{d}{dt}\mathrm{D}_{\mathrm{KL}}(p \|\pi)=-\mathcal{I}_{a}(p \|\pi).    
\end{equation*}
\end{proposition}
\begin{proof}
The proof is based on a direct calculation. Notice that 
\begin{equation*}
\begin{split}
\frac{d}{dt}\mathrm{D}_{\mathrm{KL}}(p \|\pi)= & \int \partial_t p  \log\frac{p }{\pi}+p \partial_t\log\frac{p }{\pi} dx\\   
=&\int \nabla\cdot(p \gamma)\log\frac{p }{\pi} +\nabla\cdot(p aa^{\ts}\nabla\log\frac{p }{\pi})dx+\int p  \frac{\partial_t p }{p } dx \\
=& -\int \la\nabla\log\frac{p }{\pi}, \gamma\ra p  dx -\int \la\nabla \log\frac{p }{\pi}, aa^{\ts}\nabla\log\frac{p }{\pi}\ra p  dx. 
\end{split}
\end{equation*}
And
\begin{equation*}
\begin{split}
-\int \la\nabla\log\frac{p }{\pi}, \gamma\ra p  dx=&\int -\la\nabla p , \gamma\ra +\la\nabla\log \pi, \gamma\ra dx\\
=&\int (\nabla\cdot\gamma+\la\nabla\log\pi, \gamma\ra)p  dx\\
=&\int \frac{1}{\pi}\nabla\cdot(\pi \gamma)p  dx\\
=&0. 
\end{split}
\end{equation*}
Combining the above facts, we finish the proof. 
\end{proof}
We next show the following facts to conclude all results.

\begin{proof}[Proof of Theorem \ref{thm1} and Corollary \ref{col3}]
From the information Gamma calculus and condition \eqref{C}, we know that 
\begin{equation*}
    \frac{d}{dt}\mathcal{I}_{a,z}(p \|\pi)\leq -2\kappa \mathcal{I}_{a,z}(p \|\pi). 
\end{equation*}
From Grownall's inequality, we have 
\begin{equation*}
   \mathcal{I}_{a,z}(p \|\pi)\leq e^{-2\kappa t}\mathcal{I}_{a,z}(p_0\|\pi). 
\end{equation*}
This finishes the proof of Corollary \ref{thm1}. We next prove Corollary \ref{col3}. Notice that 
\begin{equation*}
\begin{split}
-\mathrm{D}_{\mathrm{KL}}(p \|\pi)=&\int_t^{\infty}\frac{d}{ds}\mathrm{D}_{\mathrm{KL}}(p_s\|\pi)ds\\
=&-\int^\infty_t\mathcal{I}_a(p_s\|\pi)ds\\
\geq &-\int^{\infty}_t\mathcal{I}_{a,z}(p_s\|\pi) ds\\
\geq &\frac{1}{2\kappa}\int^\infty_t \frac{d}{ds}\mathcal{I}_{a,z}(p_s\|\pi)ds\\
=&-\frac{1}{2\kappa}\mathcal{I}_{a,z}(p \|\pi),
\end{split}
\end{equation*}
where we use the fact that $p_\infty=\pi$. Hence we prove the log-Sobolev inequality for $aa^{\ts}+zz^{\ts}$ with a bound $\kappa$. Using this log-Sobolev inequality for $aa^{\ts}+zz^{\ts}$, one can prove (i), (ii) directly. Notice that 
\begin{equation*}
\mathrm{D}_{\mathrm{KL}}(p \|\pi)\leq \frac{1}{2\kappa}\mathcal{I}_{a,z}(p \|\pi)\leq \frac{1}{2\kappa}e^{-2\kappa t}\mathcal{I}_{a,z}(p_0\|\pi).     
\end{equation*}
Following an inequality between KL divergence and $L_1$ distance, i.e., 
\begin{equation*}
    \int |p(t,x)-\pi(x)| dx\leq \sqrt{2\mathrm{D}_{\mathrm{KL}}(p \|\pi)},
\end{equation*}
we can show the exponential decay of the solution in terms of $L_1$ distances. This follows from the exponential decay of KL divergence. We finish the proof. 
\end{proof}

\section*{Appendix: Proof of examples.}\label{sec3}
In the appendix, we provide detailed proofs for all examples.

\begin{proof}[Proof of Proposition \ref{prop: non deg non rev a}]
Following Definition \ref{def: curvature sum}, with non-degenerate square matrix $a(x)\in\hR^{n\times n}$ defined in \eqref{non degenerate a}, we have the following tensor:
\[
\mathfrak{R}=\mathfrak R_{a}-\Lambda_1^{\ts}\Lambda_1+\mathsf D^{\ts}\mathsf D+\beta\mathfrak R_{\mathcal I_a}+(1-\beta) \mathfrak R_{\gamma_a},
\]
where $z=0$. We only need to find the vector $\Lambda_1$ ($\Lambda_2=0$ since $z=0$), which satisfies the following condition, 
\[
	(\mathsf Q^{\ts}\Lambda_1)^{\ts}\mathsf X=(\mathsf C+\mathsf Q^{\ts}\mathsf D)^{\ts}\mathsf X,
\]
where all vectors are defined in Definition \ref{def: curvature sum} and Definition \ref{definition: F G V}.
Due to the special assumption of matrix $a$, it is easy to check that $a^{\ts}_i\nabla a^{\ts}_{k\hat k}=0$, if $i\neq k$, and $a^{\ts}_{i\hat i}=0$, if $i\neq \hat i$. We thus have 
\[
\mathsf C_{\hat i\hat k}=0,\quad \text{for}\quad \hat i,\hat k=1,\cdots, n,
\]
which directly implies that $-\Lambda_1^{\ts}\Lambda_1+\mathsf D^{\ts}\mathsf D=0$. We next compute $\mathfrak R_a$ as in Definition \ref{def: curvature sum}. Following the special form of matrix $a$, we observe that,
\begin{align*}
& \sum_{i,k=1}^n a^{\ts}_i\nabla a^{\ts}_i\nabla  a^{\ts}_k \mathsf{U} (a^{\ts}_k\mathsf{U})+\sum_{i,k=1}^n a^{\ts}_i a^{\ts}_i\nabla^2a^{\ts}_k \mathsf{U} (a^{\ts}_k\mathsf{U})\nonumber  \\
	& -\sum_{i,k=1}^n a^{\ts}_k\nabla a^{\ts}_i\nabla a^{\ts}_i\mathsf{U}(a_k^{\ts}\mathsf{U})-\sum_{i,k=1}^n a^{\ts}_k a^{\ts}_i \nabla^2a^{\ts}_i\mathsf{U}(a_k^{\ts}\mathsf{U})\nonumber \\
	=&\sum_{i=k=1}^n a^{\ts}_i\nabla a^{\ts}_i\nabla  a^{\ts}_i \mathsf{U} (a^{\ts}_i\mathsf{U})+\sum_{i=k=1}^n a^{\ts}_i a^{\ts}_i\nabla^2a^{\ts}_i \mathsf{U} (a^{\ts}_i\mathsf{U})\nonumber  \\
	& -\sum_{i=k=1}^n a^{\ts}_i\nabla a^{\ts}_i\nabla a^{\ts}_i\mathsf{U}(a_i^{\ts}\mathsf{U})-\sum_{i=k=1}^n a^{\ts}_i a^{\ts}_i \nabla^2a^{\ts}_i\mathsf{U}(a_i^{\ts}\mathsf{U})=0, \\
\end{align*}
and 
\begin{align*}
    	&\la \nabla a, \sum_{k=1}^n\Big[ a^{\ts}\nabla a^{\ts}_k\mathsf{U}-a^{\ts}_k \nabla a^{\ts}\mathsf{U})\Big] a^{\ts}_k\mathsf{U}\ra \\
    	=&\sum_{i=1}^n \partial_{x_i} a^{\ts}_{ii} \Big(\sum_{k=1}^n\Big[ a^{\ts}_{ii}\nabla a^{\ts}_k\mathsf{U}-a^{\ts}_k \nabla a^{\ts}_{ii}\mathsf{U})\Big] a^{\ts}_k\mathsf{U}\Big)
    	=\sum_{i=1}^n \partial_{x_i} a^{\ts}_{ii} \Big(\Big[ a^{\ts}_{ii}\nabla a^{\ts}_i\mathsf{U}-a^{\ts}_i \nabla a^{\ts}_{ii}\mathsf{U})\Big] a^{\ts}_k\mathsf{U}\Big)=0.
\end{align*}
We thus have
\begin{align*}
\mathfrak R_{a}(\mathsf{U},\mathsf{U})=
&\sum_{i=1}^n\sum_{\hat k=1}^{n}\Big[ (aa^{\ts} \nabla\log \pi)_{\hat k} \nabla_{\hat k} a^{\ts}_i\mathsf{U}-a_i^{\ts}\nabla (aa^{\ts}\nabla\log \pi)_{\hat k} \mathsf U_{\hat k}\Big]a^{\ts}_i\mathsf{U}-\la  \left(a^{\ts} \nabla^2 a\circ  (a^{\ts}\mathsf{U})\right),a^{\ts}\mathsf{U}\ra_{\hR^n}.\nonumber
\end{align*}
Based on the definition of $\pi=\frac{1}{Z}e^{-V}$, we obtain $aa^{\ts}\nabla \log \pi=-aa^{\ts}\nabla V$. Observe that $a=a^{\ts}$. We thus get the following representation: 
\begin{align*}
\mathfrak R_{a}(\mathsf{U},\mathsf{U})=
&\sum_{i=1}^n\sum_{\hat k=1}^{n}\Big[ (aa^{\ts} \nabla\log \pi)_{\hat k} \nabla_{\hat k} a^{\ts}_i\mathsf{U}-a_i^{\ts}\nabla (aa^{\ts}\nabla\log \pi)_{\hat k} \mathsf U_{\hat k}\Big]a^{\ts}_i\mathsf{U}-\la  \left(a^{\ts} \nabla^2 a\circ  (a^{\ts}\mathsf{U})\right),a^{\ts}\mathsf{U}\ra_{\hR^n}\\
=&-\sum_{i=1}^n a_{ii}^2\partial_{x_i}V \partial_{x_i}a_{ii}\mathsf U_i a_{ii}\mathsf U_i+\sum_{i=1}^n\sum_{k=1}^na_{ii}\partial_{x_i}(a_{kk}^2\partial_{x_k}V)\mathsf U_k a_{ii}\mathsf{U}_i-\sum_{i=1}^na_{ii}\partial^2_{x_ix_i} a_{ii}a_{ii}\mathsf{U}_i a_{ii}\mathsf U_i\\
=&\sum_{i=1}^n[ a_{ii}^3\partial_{x_i}a_{ii}\partial_{x_i}V +a_{ii}^4\partial^2_{x_ix_i}V-a_{ii}^3\partial_{x_ix_i}^2a_{ii}  ](\mathsf{U}_i)^2+\sum_{i,k=1,i\neq k}^n a_{ii}^2a_{kk}^2\partial^2_{x_ix_k}V \mathsf{U}_i\mathsf{U}_k\\
=&\mathsf{U}^{\ts}\mathfrak R_a \mathsf U,
\end{align*}
where $\mathfrak R_a$ is defined in \ref{ricci curvature}. Similarly, we compute $\mathfrak R_{\mathcal I_a}$ and $\mathfrak R_{\gamma_a}$ as below.
\beaa 
\mathfrak R_{\mathcal I_a}(\mathsf{U},\mathsf{U})
&=& \la \mathsf{U},\gamma \ra\Big[ (\nabla a)\circ (a^{\ts}\mathsf{U})+\sum_{i=1}^na^{\ts}_i\nabla a^{\ts}_i\mathsf{U}+\la aa^{\ts}\nabla\log\pi,\mathsf{U}\ra\Big]\nonumber-\sum_{\hat i=1}^n\sum_{k=1}^{n+m} \gamma_{k}\nabla_{x_{k}} a^{\ts}_{\hat i}\mathsf{U}( a^{\ts}_{\hat i}\mathsf{U}),\nonumber\\ 
&=&\sum_{i=1}^n\mathsf U_i\gamma_i\sum_{k=1}^n[ 2a^{\ts}_{kk}\pa_{x_k}a^{\ts}_{kk}-(a^{\ts}_{kk})^2\pa_{x_k}V]\mathsf U_k-\sum_{k=1}^n\gamma_k \pa_{x_k}a^{\ts}_{kk}a^{\ts}_{kk} (\mathsf U_k)^2\\
&=&\mathsf U^{\ts} \mathfrak R_{\mathcal I_a}\mathsf U\\
\mathfrak R_{\gamma_a}(\mathsf{U},\mathsf{U})&=& \frac{1}{2}\sum_{\hat k=1}^{n}\gamma_{\hat k} \la \mathsf U,\nabla_{\hat k}(aa^{\ts})\mathsf U\ra -  \la \nabla \gamma \mathsf{U},aa^{\ts}\mathsf{U}\ra_{\mathbb R^{n} }\\
&=&(\mathsf U)^{\ts}\Big(\textbf{Diag}\Big\{\gamma_ia^{\ts}_{ii}(x_i)\pa_{x_i}a^{\ts}_{ii}(x_i)\Big\}_{i=1}^n-\frac{1}{2}[(\nabla \gamma)^{\ts}aa^{\ts}+aa^{\ts}\nabla \gamma ]
\Big)\mathsf U\\
&=& \mathsf U^{\ts}\mathfrak R_{\gamma_a}\mathsf U.
\eeaa \qed
\end{proof}

\noindent
\textbf{Proof of Example I: underdamped Langevin dynamics.}
For underdamped Langevin dynamics with variable coefficients defined in \eqref{variable temp underdamped}, we first observe that 
\beaa
aa^{\ts}&=&\begin{pmatrix}0&0\\
0& r(x)
\end{pmatrix},\quad \Big(\sum_{j=1}^2\frac{\partial}{\partial x_j}(aa^{\ts})_{ij} \Big)_{1\le i\le 2}=\begin{pmatrix}0\\
0
\end{pmatrix},\quad 
 aa^{\ts}\nabla \log \pi = 
\begin{pmatrix}0\\
(a^{\ts}_{12})^2 \frac{\partial  \log\pi }{\partial v}
\end{pmatrix}.
\eeaa
By routine computations, we have the following proposition.
\begin{proposition}\label{cor: constant z}
For any constant $\beta\in\mathbb R$, and any smooth function $f\in \mathcal C^{\infty}(\mathbb R^2)$, we have
	\beaa
&&\int \Big[ \widetilde \Gamma_{2}(f,f)+\widetilde \Gamma_2^{z,\pi}(f,f)+\Gamma_{\mathcal I_{a,z}}(f,f)\Big] pdx\\
&=&\int \Big[  \|\mathfrak{Hess}_{\beta}f\|^2_{\mathrm{F}}
+\mathfrak{R}(\nabla f,\nabla f) +\mathfrak{R}_{\mathcal I}(\nabla f,\nabla f)\Big]pdx,
\eeaa 
where
\beaa
\|\mathfrak{Hess}_{\beta}f\|_{\mathrm{F}}^2&=&[\mathsf X+\Lambda]^{\ts}(\mathsf Q^{\ts}\mathsf Q+\mathsf P^{\ts}\mathsf P )[\mathsf X+\Lambda],
\eeaa 
with
\beaa\label{C pi}
 \Lambda&=&(0,0,\Lambda_3,\Lambda_4), \quad  \mathsf Q=\begin{pmatrix}
0&0&0& (a^{\ts}_{12})^2
\end{pmatrix},\quad \mathsf P= \begin{pmatrix}
0&0&a^{\ts}_{12}z_1&a^{\ts}_{12}z_2
\end{pmatrix},\\
\begin{pmatrix}
\Lambda_3\\ \Lambda_4
\end{pmatrix}&=&\frac{1}{(a^{\ts}_{12})^2}(aa^{\ts}+zz^{\ts})^{-1}\mathsf K\nabla f.
\eeaa 
And the curvature tensor $\mathfrak{R}$ is presented in Theorem \ref{thm: constant z}.
\end{proposition}
\begin{remark}\label{hessian relation}
We use a different notion of the Hessian of function $f$, which gives us the other equivalent formulation of the tensor $\mathfrak{R}$. The key observation is the following relation 
\beaa
&& [\mathsf Q\mathsf X+\widehat\Lambda_1]^{\ts} [\mathsf Q\mathsf X+\widehat\Lambda_1]+[\mathsf P\mathsf X+\widehat\Lambda_2]^{\ts}[\mathsf P\mathsf X+\widehat\Lambda_2]-\widehat\Lambda_1^{\ts}\widehat\Lambda_1-\widehat\Lambda_2^{\ts}\widehat\Lambda_2\\
&=& [\mathsf X+ \Lambda_1]^{\ts}\mathsf Q^{\ts}\mathsf Q [\mathsf X+\Lambda_1]+[\mathsf X+\Lambda_2]^{\ts}\mathsf P^{\ts}\mathsf P [\mathsf X+\Lambda_2]-\Lambda_1^{\ts}\mathsf Q^{\ts}\mathsf Q\Lambda_1-\Lambda_2^{\ts}\mathsf P^{\ts}\mathsf P\Lambda_2,
\eeaa
where $\widehat \Lambda_1^{\ts} = \Lambda_1^{\ts}\mathsf Q^{\ts},\quad \widehat\Lambda_2^{\ts} =\Lambda_2^{\ts}\mathsf P^{\ts}$. Shortly, we show that the vectors $\Lambda_1\in \mathbb R^{(n+m)^2\times 1}$ and $\Lambda_2\in \mathbb R^{(n+m)^2\times 1}$ exist. And we compute them explicitly for the variable coefficient underdamped Langevin dynamics. 
\end{remark}
\begin{proof}[Proof of Proposition \ref{cor: constant z}]
We demonstrate explicit examples with constant matrix $z$, i.e. $z^{\ts}=(z_1\quad z_2)$, with $z_1$ and $z_2$ being constants. In particular, we have used the notation below: $z^{\ts}_1=(z^{\ts}_{11}\quad z^{\ts}_{12})=(z_1\quad z_2)$, and $\pa_1f=\pa_xf$, $\pa_2f=\pa_vf$.

\noindent\textbf{Step 1:} We first have the following simplified quantities:
\beaa 
\mathfrak R_a(\nabla f,\nabla f)&=&-(a^{\ts}_{12})^2 \frac{\partial^2  \log\pi }{\partial v^2} |a^{\ts}_{12}\partial_2f|^2,
\quad \mathfrak R_z(\nabla f,\nabla f)=-z^{\ts}_1\nabla((a^{\ts}_{12})^2 \frac{\partial  \log\pi }{\partial v})\partial_2f z^{\ts}_1\nabla f,\\
\label{tensor R Psi}\mathfrak{R}_{\pi}(\nabla f,\nabla f)
	&=&2\Big[z^{\ts}_{1} z^{\ts}_{1} \nabla^2 a^{\ts}_{1} \nabla f   a^{\ts}_{1}\nabla f+(z^{\ts}_{1}\nabla a^{\ts}_1\nabla f )^2+ (z^{\ts}\nabla\log\pi)_1 \left[ z^{\ts}_{ 1}\nabla a^{\ts}_1\nabla f a^{\ts}_1\nabla f \right]\Big] \nonumber,\\
\mathfrak R_{\mathcal I_a}(\nabla f,\nabla f)&=&\la \nabla f,\gamma\ra \la aa^{\ts}\nabla\log\pi ,\nabla f\ra -\gamma_1\frac{\partial }{\partial x}a^{\ts}_{12}\partial_2 fa^{\ts}_{12}\partial_2 f,\\ 
\mathfrak R_{\gamma_a}(\nabla f,\nabla f)&=&\frac{1}{2}\sum_{k=1}^2 \gamma_k  \la \nabla f, \nabla_k(aa^{\ts})\nabla f\ra -\la\begin{pmatrix}
\nabla_1\gamma_1&\nabla_1\gamma_2\\
\nabla_2\gamma_1&\nabla_2\gamma_2
\end{pmatrix}\begin{pmatrix}
\pa_1f\\
\pa_2 f
\end{pmatrix}, aa^{\ts}\nabla f \ra,\\ 
\mathfrak R_{\gamma_z}(\nabla f,\nabla f)&=&  -\la\begin{pmatrix}
\nabla_1\gamma_1&\nabla_1\gamma_2\\
\nabla_2\gamma_1&\nabla_2\gamma_2
\end{pmatrix}\begin{pmatrix}
\pa_1f\\
\pa_2f
\end{pmatrix}, zz^{\ts}\nabla f \ra.
\eeaa 
And
\beaa
\mathsf E&=&\mathsf D=0,\quad \mathsf C=(0,0,0,0)^{\ts},\quad \mathsf F=\Big(0,0,0,\mathsf F_{22} \Big)^{\ts},\quad  \mathsf F_{22}=-z_{11}^{\ts}  \pa_x a^{\ts}_{12}a^{\ts}_{12} z^{\ts}_1\nabla f,\\
 \mathsf G&=&\Big(0 ,0 ,\mathsf G_{21} ,\mathsf G_{22} \Big)^{\ts},\quad  \mathsf G_{21}=2(z^{\ts}_{11})^2\partial_x a^{\ts}_{12} a^{\ts}_{12}\partial_2f,\quad  \mathsf G_{22}= 2z^{\ts}_{11}z^{\ts}_{12}\pa_x a^{\ts}_{12} a^{\ts}_{12}\partial_2 f, \\
 \mathsf V^a&=& (0,0,\mathsf V^a_{21},\mathsf V^a_{22} ),\quad\mathsf V^a_{21}= -\frac{1}{2}\gamma_1a^{\ts}_{12}a^{\ts}_{12}\partial_2f, \quad \mathsf V^a_{22}= -\frac{1}{2}\gamma_2a^{\ts}_{12}a^{\ts}_{12}\partial_2f+\frac{1}{2}a^{\ts}_{12}a^{\ts}_{12}\la\nabla f,\gamma\ra. 
\eeaa

\noindent\textbf{Step 2:} We now find the vector $\Lambda$ in Remark \ref{hessian relation}.
By Definition \ref{def: curvature sum}, we have 
\beaa
P^TP=(a^{\ts}_{12})^2\begin{pmatrix}
	0&0&0&0\\
	0&0&0&0\\
	0&0&z_1^2&z_1z_2\\
	0&0&z_1z_2&z_2^2\\
\end{pmatrix},\quad \text{and}\quad Q^TQ=(a^{\ts}_{12})^2\begin{pmatrix}
	0&0&0&0\\
	0&0&0&0\\
	0&0&0&0\\
	0&0&0&(a^{\ts}_{12})^2\\
\end{pmatrix},
\eeaa
and $\mathsf Q^{\ts}D=(0,0,0,0)^{\ts}$, $\mathsf P^{\ts}E=( 0,0,z^{\ts}_{11}a^{\ts}_{12} \mathsf E,z^{\ts}_{12}a^{\ts}_{12} \mathsf E)^{\ts}$. For simplicity, taking $\Lambda_1=\Lambda_2=\Lambda=(0,0,\Lambda_3,\Lambda_4)$, we have
	\beaa
	(\mathsf Q^{\ts}\mathsf Q\Lambda+\mathsf P^{\ts}\mathsf P \Lambda)^{\ts}\mathsf X&=&(\mathsf F+\mathsf C+\mathsf G+\beta \mathsf V^a+\mathsf Q^{\ts}\mathsf D+\mathsf P^{\ts}\mathsf E)^{\ts}\mathsf X.
	\eeaa
By matching the coefficients of $\frac{\pa^2 f}{\pa x\pa v}$ and $\frac{\pa^2 f}{\pa v\pa v}$ for both sides of the above equation, we have 
\beaa
(a^{\ts}_{12})^2\begin{pmatrix}
z_1^2&z_1z_2\\
z_1z_2&z_2^2+(a^{\ts}_{12})^2
\end{pmatrix}\begin{pmatrix}
\Lambda_3\\
\Lambda_4
\end{pmatrix}=\begin{pmatrix}
\mathsf K_1\\
\mathsf K_2
\end{pmatrix},
\eeaa
with
\beaa
\mathsf K_1=: \mathsf F_{21}+\mathsf G_{21}+\beta\mathsf V^a_{21},\quad 
\mathsf K_2=:\mathsf F_{22}+\mathsf G_{22}+\beta\mathsf V^a_{22}.
\eeaa 
Notice that
\beaa 
\mathsf K_1&=& \Big(2(z^{\ts}_{11})^2\partial_x a^{\ts}_{12} a^{\ts}_{12}-\frac{\beta}{2} \gamma_1a^{\ts}_{12}a^{\ts}_{12}\Big)\pa_2f:= \mathsf K_{12} \pa_2f,  \\ 
\mathsf K_2&=& [-z_1^2\pa_x[a^{\ts}_{12}]a^{\ts}_{12}+\frac{\beta}{2}(a^{\ts}_{12})^2\gamma_1]\pa_1f+z_1z_2\pa_x[a^{\ts}_{12}]a^{\ts}_{12}\pa_2f:=\mathsf K_{21}\pa_1f+\mathsf K_{22}\pa_2f.
\eeaa 
Based on the above computation, we have 
\beaa
\begin{pmatrix}
\Lambda_3\\
\Lambda_4
\end{pmatrix}=\frac{1}{(a^{\ts}_{12})^2}\begin{pmatrix}
z_1^2&z_1z_2\\
z_1z_2&z_2^2+(a^{\ts}_{12})^2
\end{pmatrix}^{-1}\begin{pmatrix}
\mathsf K_{11}&\mathsf K_{12}\\
\mathsf K_{21}&\mathsf K_{22}
\end{pmatrix} \nabla f=\frac{1}{(a^{\ts}_{12})^2}(aa^{\ts}+zz^{\ts})^{-1}\mathsf K\nabla f.
\eeaa
Since $aa^{\ts}+zz^{\ts}$ is symmetric, we have the following relation 
\beaa 
\Lambda^{\ts}(\mathsf Q^{\ts}\mathsf Q+\mathsf P^{\ts}\mathsf P)\Lambda&=&(\nabla f)^{\ts} \mathfrak{M}_{\Lambda} \nabla f,
\eeaa 
with the matrix $\mathfrak M_{\Lambda}$ defined by 
\beaa \label{matrix lambda}
\mathfrak M_{\Lambda} =\frac{1}{(a^{\ts}_{12})^2} ( \mathsf K)^{\ts}(aa^{\ts}+zz^{\ts})^{-1}\mathsf K.
\eeaa

\noindent\textbf{Step 3:} We next transfer all the bilinear forms into its corresponding symmetric  matrix forms. In particular, for each bi-linear form $\mathfrak R(\nabla f,\nabla f)$, we keep the convention
\beaa 
\mathfrak R(\nabla f,\nabla f)=(\nabla f)^{\ts}\mathfrak R\nabla f.
\eeaa 
We focus on the symmetric matrix form $\mathfrak R$: 
\beaa 
\mathfrak R_a(\nabla f,\nabla f)&=&-(a^{\ts}_{12})^2 \frac{\partial^2  \log\pi }{\partial v^2} |a^{\ts}_{12}\partial_2f|^2
=(\nabla f)^{\ts} \begin{pmatrix}
0&0\\
0& -\frac{\partial^2  \log\pi }{\partial v^2} |a^{\ts}_{12}|^4
\end{pmatrix}\nabla f,
\eeaa 

\beaa 
\quad \mathfrak R_z(\nabla f,\nabla f)&=&-z^{\ts}_1\nabla((a^{\ts}_{12})^2 \frac{\partial  \log\pi }{\partial v})\partial_2f z^{\ts}_1\nabla f\\
&=&\frac{1}{2} (\nabla f)^{\ts} \Big[  \begin{pmatrix} 
0\\
-z^{\ts}_1\nabla((a^{\ts}_{12})^2 \frac{\partial  \log\pi }{\partial v})
\end{pmatrix} z^{\ts}_1+z_1\begin{pmatrix}
0& -z^{\ts}_1\nabla((a^{\ts}_{12})^2 \frac{\partial  \log\pi }{\partial v})
\end{pmatrix} \Big]\nabla f,
\eeaa 
\beaa 
\label{tensor R Psi}\mathfrak{R}_{\pi}(\nabla f,\nabla f)
	&=&2\Big[z^{\ts}_{1} z^{\ts}_{1} \nabla^2 a^{\ts}_{1} \nabla f   a^{\ts}_{1}\nabla f+(z^{\ts}_{1}\nabla a^{\ts}_1\nabla f )^2+ (z^{\ts}\nabla\log\pi)_1 \left[ z^{\ts}_{ 1}\nabla a^{\ts}_1\nabla f a^{\ts}_1\nabla f \right]\Big] \nonumber\\
	&=& (\nabla f)^{\ts} \begin{pmatrix}
0&0\\
0& C_{\pi}
\end{pmatrix}\nabla f,
	\eeaa 
	with 
\bea\label{C pi}
C_\pi=2\Big[z^{\ts}_{1} z^{\ts}_{1} \nabla^2 a^{\ts}_{12}  a^{\ts}_{12}+(z^{\ts}_{1}\nabla a^{\ts}_{12} )^2+ (z^{\ts}\nabla\log\pi)_1 [ z^{\ts}_{ 1}\nabla a^{\ts}_{12} a^{\ts}_{12}]\Big].
\eea 
We also have 
	\beaa 
\mathfrak R_{\mathcal I_a}(\nabla f,\nabla f)&=&\la \nabla f,\gamma\ra \la aa^{\ts}\nabla\log\pi ,\nabla f\ra -\gamma_1\frac{\partial }{\partial x}a^{\ts}_{12}\partial_2 fa^{\ts}_{12}\partial_2 f\\ 
&=& (\nabla f)^{\ts}\Big[\frac{\gamma (aa^{\ts}\nabla\log\pi)^{\ts}+ (aa^{\ts}\nabla\log\pi) \gamma^{\ts}}{2} - \begin{pmatrix}
0&0\\
0& \gamma_1\frac{\partial }{\partial x}a^{\ts}_{12}a^{\ts}_{12}
\end{pmatrix}\Big] \nabla f,
\eeaa 
and 
\beaa 
\mathfrak R_{\gamma_a}(\nabla f,\nabla f)&=& \frac{1}{2}\la \nabla f,\sum_{k=1}^2\gamma_k \nabla_k(aa^{\ts}) \nabla f \ra - \la \nabla \gamma \nabla f,aa^{\ts}\nabla f \ra   \\ 
&=& (\nabla f)^{\ts}\Big[\frac{1}{2} \sum_{k=1}^2 \gamma_k\nabla_k(aa^{\ts})-\frac{1}{2}[(\nabla \gamma)^{\ts}aa^{\ts}+aa^{\ts} \nabla \gamma]  \Big]\nabla f.
\eeaa 
The last equality follows from the fact that $aa^{\ts}$ is a diagonal matrix. Hence $\sum_{k=1}^2\gamma_k\nabla_k(aa^{\ts})$ is also a diagonal matrix. Similarly, we derive 
\beaa 
\mathfrak R_{\gamma_z}(\nabla f,\nabla f)&=&  -\la\begin{pmatrix}
\nabla_1\gamma_1&\nabla_1\gamma_2\\
\nabla_2\gamma_1&\nabla_2\gamma_2
\end{pmatrix}\begin{pmatrix}
f_1\\
f_2
\end{pmatrix}, zz^{\ts}\nabla f \ra,\\
&=&(\nabla f)^{\ts}\frac{1}{2} \Big[-(\nabla \gamma)^{\ts}zz^{\ts}-zz^{\ts}\nabla \gamma  \Big]\nabla f.
\eeaa 
The proof is completed.\qed
\end{proof}

\noindent
\textbf{Proof of example II: three oscillator chain model.}

\begin{proof}[Proof of Proposition \ref{prop: tensor oscilator}]
Following Theorem \ref{thm: information bochner}, consider a constant matrix $a$ and a matrix function $z$,
\beaa
z=\begin{pmatrix}
	z_1\mathsf I_3&0_{3\times 1}\\
	z_2\mathsf I_3&(\overrightarrow{z_3})_{3\times 1}
\end{pmatrix}_{6\times 4},
\eeaa
with $\overrightarrow{z_3}=(z_{31}\quad z_{32}(p_0,p_2)\quad z_{33})^{\ts}$ and constants $z_1$, $z_2$. We compute the matrix tensor $\mathfrak R$ for $\beta=0$. By abusing the notations, we denote $z^{\ts}_i$ as the $i$-th row vector of the transpose matrix $z^{\ts}$. We first find vectors $\Lambda_1$ and $\Lambda_2$. For vectors $\mathsf C, \mathsf F,\mathsf G, \mathsf V^a \in \mathbb R^{(6\times 6)\times 1}$, $\mathsf D \in \mathbb R^{(2\times 2)\times 1}$ and $\mathsf E\in \mathbb R^{(2\times 4)\times 1}$ with $\mathsf U\in\hR^{6}$, we have  
\beaa
E_{ik}&=&a^{\ts}_{i}\nabla z^{\ts}_k \mathsf U,\quad \mathsf E_{14,1}=a^{\ts}_{14}\pa_{v_0}z^{\ts}_{45} \mathsf U_5,\quad \mathsf E_{24,1}=a^{\ts}_{26}\pa_{p_2}z^{\ts}_{45}\mathsf U_5;\\
\mathsf C_{\hat i\hat k} &=& \sum_{i,k=1}^2\Big( a^{\ts}_{i\hat i}a^{\ts}_{i}\nabla a^{\ts}_{k\hat k}  -a^{\ts}_{i\hat k}a^{\ts}_{k}\nabla a^{\ts}_{i\hat i} \Big) a^{\ts}_k \mathsf{U}=0,\quad \mathsf D_{ik}= a^{\ts}_{i}\nabla a^{\ts}_k\mathsf U=0;  \nonumber\\ 
\mathsf F_{\hat i\hat k}&=& \sum_{i=1}^2\sum_{k=1}^4\Big(  a^{\ts}_{i\hat i}a^{\ts}_{i}\nabla z^{\ts}_{k\hat k}\Big)z^{\ts}_k\mathsf U,\quad \mathsf F_{45,1}=|a^{\ts}_{14}|^2\pa_{p_0} z^{\ts}_{45}z^{\ts}_4\mathsf U,\quad \mathsf F_{65,1}=|a^{\ts}_{26}|^2\pa_{p_2}z^{\ts}_{45}z^{\ts}_4\mathsf U; \\
\mathsf G_{\hat i\hat k}&=& \sum_{i=1}^2\sum_{k=1}^4 \left[-\left(a^{\ts}_{i\hat i}a^{\ts}_{i}\nabla z^{\ts}_{k\hat k} z^{\ts}_k\mathsf U +z^{\ts}_{k\hat k}a^{\ts}_{i\hat i}a^{\ts}_{i}\nabla z^{\ts}_{k} \mathsf U\right)\right],\nonumber\\
\mathsf G_{45,1}&=&-|a^{\ts}_{14}|^2\pa_{p_0}z^{\ts}_{45}z^{\ts}_4\mathsf U-|a^{\ts}_{14}|^2\pa_{p_0}z^{\ts}_{45}z^{\ts}_{45}\mathsf U_5,\\
\mathsf G_{65,1}&=&-|a^{\ts}_{26}|^2\pa_{p_2}z^{\ts}_{45}z^{\ts}_4\mathsf U-|a^{\ts}_{26}|^2z^{\ts}_{45}\pa_{p_2}z^{\ts}_{45}\mathsf U_5; \\
(\mathsf P^{\ts} \mathsf E)_{45,1}&=&a^{\ts}_{14}z^{\ts}_{45}\mathsf E_{14,1}=|a^{\ts}_{14}|^2z^{\ts}_{45}\pa_{p_0}z^{\ts}_{45}\mathsf U_5,\quad (\mathsf P^{\ts} \mathsf E)_{65,1}=a^{\ts}_{26}z^{\ts}_{45}\mathsf E_{24,1}=|a^{\ts}_{26}|^2z^{\ts}_{45}\pa_{p_2}z^{\ts}_{45}\mathsf U_5.
\eeaa
By direct computations, we observe that
$\mathsf F+\mathsf G+\mathsf C+\mathsf Q^{\ts}\mathsf D+\mathsf P^{\ts}\mathsf E=\overrightarrow{0}_{36\times 1}$. We thus get $\Lambda_1=\Lambda_2={\bf 0}\in\hR^6$. We next compute the following matrix tensors. For simplicity, we directly omit the zero terms based on Definition \ref{def: curvature sum} and matrix functions $a$, $z$. We have
\beaa
	\mathfrak R_{a}(\mathsf{U},\mathsf{U})&=& 
-\sum_{i=1}^2\sum_{\hat k=1}^{6}\Big[a_i^{\ts}\nabla (aa^{\ts}\nabla\log \pi)_{\hat k} \mathsf U_{\hat k}\Big]a^{\ts}_i\mathsf{U}\\
&=&-a^{\ts}_{14}\pa_{p_0}(|a^{\ts}_{14}|^2\pa_{p_0}\log \pi)\mathsf U_4 a^{\ts}_{14}\mathsf U_4-a^{\ts}_{26}\pa_{p_2}(|a^{\ts}_{26}|^2\pa_{p_2}\log \pi)\mathsf U_6 a^{\ts}_{26}\mathsf U_6\\
&=& \mathsf U^{\ts}\begin{pmatrix}
	0&0\\
	0&(\xi T)^2\mathsf I_3^O
\end{pmatrix} \mathsf U,
\eeaa 
and
\beaa 
  \label{tenser z}\mathfrak R_{z}(\mathsf{U},\mathsf{U})&=&\sum_{i,k=1}^2 a^{\ts}_i a^{\ts}_i\nabla^2z^{\ts}_k \mathsf{U} (z^{\ts}_k\mathsf{U}) \nonumber \\
&&+\sum_{k=1}^4\sum_{\hat k=1}^{6}\Big[ (aa^{\ts} \nabla\log \pi)_{\hat k} \nabla_{\hat k} z^{\ts}_k\mathsf{U}-z_k^{\ts}\nabla (aa^{\ts}\nabla\log \pi)_{\hat k} \mathsf U_{\hat k}\Big]z^{\ts}_k\mathsf{U}\nonumber\\
&=&|a^{\ts}_{14}|^2\pa_{p_0p_0}z^{\ts}_{4}\mathsf U z^{\ts}_4 \mathsf U+|a^{\ts}_{26}|^2 \pa_{p_2p_2}z^{\ts}_4\mathsf U z^{\ts}_4\mathsf U\\
&&+(|a^{\ts}_{14}|^2 \pa_{p_0}\log\pi)\pa_{p_0}z^{\ts}_4\mathsf U z_4^{\ts}\mathsf U+(|a^{\ts}_{26})^2 \pa_{p_2}\log\pi)\pa_{p_2}z^{\ts}_4\mathsf U z_4^{\ts}\mathsf U\\
&&-z_1^{\ts}\nabla((a^{\ts}_{14})^2\pa_{p_0}\log \pi \mathsf U_4)z^{\ts}_1\mathsf U-z_3^{\ts}\nabla((a^{\ts}_{26})^2\pa_{p_2}\log \pi \mathsf U_6)z^{\ts}_3\mathsf U\\
&&-z^{\ts}_4\nabla[(a^{\ts}_{14})^2\pa_{p_0}\log \pi \mathsf U_4 +(a^{\ts}_{26})^2\pa_{p_2}\log \pi \mathsf U_6]z^{\ts}_4\mathsf U\\
&=&|a^{\ts}_{14}|^2\pa_{p_0p_0}z^{\ts}_{45}\mathsf U_5 z^{\ts}_4 \mathsf U+|a^{\ts}_{26}|^2 \pa_{p_2p_2}z^{\ts}_{45}\mathsf U_5 z^{\ts}_4\mathsf U\\
&&+(|a^{\ts}_{14}|^2 \pa_{p_0}\log\pi)\pa_{p_0}z^{\ts}_{45}\mathsf U_5 z_4^{\ts}\mathsf U+(|a^{\ts}_{26})^2 \pa_{p_2}\log\pi)\pa_{p_2}z^{\ts}_{45}\mathsf U_5 z_4^{\ts}\mathsf U\\
&&-z_1^{\ts}\nabla((a^{\ts}_{14})^2\pa_{p_0}\log \pi \mathsf U_4)z^{\ts}_1\mathsf U-z_3^{\ts}\nabla((a^{\ts}_{26})^2\pa_{p_2}\log \pi \mathsf U_6)z^{\ts}_3\mathsf U\\
&&-z^{\ts}_4\nabla[(a^{\ts}_{14})^2\pa_{p_0}\log \pi \mathsf U_4 +(a^{\ts}_{26})^2\pa_{p_2}\log \pi \mathsf U_6]z^{\ts}_4\mathsf U
\eeaa 
\beaa 
&=&\frac{1}{2}\mathsf U^{\ts} \Big\{ \begin{pmatrix}
	0&0&0&0\\
	0&0&0&0\\
	0&0&0&0\\
	-z_1^{\ts}\nabla((a^{\ts}_{14})^2\pa_{p_0}\log \pi )&0&0&-z^{\ts}_4\nabla((a^{\ts}_{14})^2\pa_{p_0}\log \pi) \\
	0&0&0&\mathsf S_1\\
	0&0&-z_3^{\ts}\nabla((a^{\ts}_{26})^2\pa_{p_2}\log \pi )&-z^{\ts}_4\nabla((a^{\ts}_{26})^2\pa_{p_2}\log \pi )
\end{pmatrix}z^{\ts}\\
&&\qquad+z\begin{pmatrix}
	0&0&0&-z_1^{\ts}\nabla((a^{\ts}_{14})^2\pa_{p_0}\log \pi )&0&0\\
	0&0&0&0&0&0\\
	0&0&0&0&0&-z_3^{\ts}\nabla((a^{\ts}_{26})^2\pa_{p_2}\log \pi )\\
	0&0&0&-z^{\ts}_4\nabla((a^{\ts}_{14})^2\pa_{p_0}\log \pi)&\mathsf S_1&-z^{\ts}_4\nabla((a^{\ts}_{26})^2\pa_{p_2}\log \pi )
\end{pmatrix}\Big\} \mathsf U,
	\eeaa 
where we denote
\beaa 
\mathsf S_1 &=&|a^{\ts}_{14}|^2\pa_{p_0p_0}z^{\ts}_{45} +|a^{\ts}_{26}|^2 \pa_{p_2p_2}z^{\ts}_{45}+(|a^{\ts}_{14}|^2 \pa_{p_0}\log\pi)\pa_{p_0}z^{\ts}_{45}+(a^{\ts}_{26})^2 \pa_{p_2}\log\pi \pa_{p_2}z^{\ts}_{45}\\
&=&\xi T\pa_{p_0p_0}z^{\ts}_{45} +\xi T \pa_{p_2p_2}z^{\ts}_{45}+(\xi T \pa_{p_0}\log\pi)\pa_{p_0}z^{\ts}_{45}+(\xi T \pa_{p_2}\log\pi)\pa_{p_2}z^{\ts}_{45}.
\eeaa 
According to our definition, for $z^{\ts}_{44}=z_{31}, z^{\ts}_{45}=z_{32}(p_0,p_2),z^{\ts}_{46}=z_{33}$, we have
\beaa
	\mathfrak R_{z}(\mathsf{U},\mathsf{U})&=&\frac{1}{2}\mathsf U^{\ts} \Big[ \begin{pmatrix}
	0&0&0&0\\
	0&0&0&0\\
	0&0&0&0\\
	z_2\xi &0&0&z_{31}\xi \\
	0&0&0&\mathsf S_1\\
	0&0&z_2\xi&z_{33}\xi
\end{pmatrix}z^{\ts}+z\begin{pmatrix}
	0&0&0&z_2\xi&0&0\\
	0&0&0&0&0&0\\
	0&0&0&0&0&z_2\xi\\
	0&0&0&z_{31}\xi&\mathsf S_1&z_{33}\xi
\end{pmatrix}\Big] \mathsf U\\
	&=&\begin{pmatrix}
	0 & \frac{1}{2}z_1z_2\xi\mathsf I_3^O\\
	\frac{1}{2}z_1z_2\xi\mathsf I_3^O & \frac{1}{2}(\overrightarrow{z_3^{\mathsf S}}\overrightarrow{z_3}^{\ts}+\overrightarrow{z_3}(\overrightarrow{z_3^{\mathsf S}})^{\ts})
\end{pmatrix}_{6\times 6},
	\eeaa
	where we denote	$\overrightarrow{z_3^\mathsf S}=(z_{31}\xi\quad \mathsf S_1 \quad z_{33}\xi)^{\ts}$.
Furthermore, we have  $	\mathfrak R_{\gamma_a}(\mathsf{U},\mathsf{U})= -  \la \nabla \gamma \mathsf{U},aa^{\ts}\mathsf{U}\ra_{\mathbb R^{6}}=-\mathsf U^{\ts} \frac{1}{2}[(\nabla \gamma)^{\ts}aa^{\ts}+aa^{\ts} \nabla \gamma]\mathsf U$, since $\nabla(aa^{\ts})=0$. From our definition for $(\nabla \gamma)_{ij}=\nabla_i\gamma_j$, we have 
\beaa
\nabla \gamma=\begin{pmatrix}
	0& \nabla_{qq}H\\
	-\nabla_{pp}H& 0
\end{pmatrix}=\begin{pmatrix}
	0& L\\
	-\mathsf I_3& 0
\end{pmatrix}. 
\eeaa 
Here we denote \qi{$\nabla^2_{qq}H$ ($\nabla^2_{pp}H$ resp.)} the Hessian in the q ($p$ resp.) direction, i.e. \qi{$L=(\nabla^2_{q_iq_j}H)_{0\le i,j\le 2}$}.
Plugging them in $\nabla\gamma$, we have
\beaa
\mathfrak R_{\gamma_a}&=& \frac{1}{2} \begin{pmatrix}
		0&\xi T\mathsf I_3^O\\
		\xi T\mathsf I_3^O&0
	\end{pmatrix}_{6\times 6}.
\eeaa 
{ Similarly, $
	\mathfrak R_{\gamma_z}(\mathsf{U},\mathsf{U})=\frac{1}{2}\la \mathsf{U},\sum_{k=1}^6\gamma_k \nabla_k(zz^{\ts})\mathsf{U}\ra -\frac{1}{2}\mathsf{U}^{\ts}[(\nabla \gamma)^{\ts}zz^{\ts}+zz^{\ts} \nabla \gamma]\mathsf{U}$. And we get
	\beaa
	\mathfrak R_{\gamma_z}&=&\frac{1}{2}\begin{pmatrix}0& 0 \\
	0& [\partial_{q_0}H\partial_{p_0}+\partial_{q_2}H\partial_{p_2}](\overrightarrow{z_3}\overrightarrow{z_3}^{\ts} )
	\end{pmatrix}_{6\times 6}\\
	&&+\begin{pmatrix}
		z_1z_2\mathsf I_3&\frac{1}{2}(z_2^2\mathsf I_3+ \overrightarrow{z_3}\overrightarrow{z_3}^{\ts} -z_1^2L)\\
		\frac{1}{2}(z_2^2\mathsf I_3+ \overrightarrow{z_3}\overrightarrow{z_3}^{\ts} -z_1^2L)& -z_1z_2 L
	\end{pmatrix}_{6\times 6}.
	\eeaa 
}
	
Lastly, we compute the matrix tensor $\mathfrak R_{\pi}$.
	\beaa
\label{tensor R Psi}\mathfrak{R}_{\pi}(\mathsf{U},\mathsf{U})
	&=&-2\sum_{j=1}^4\sum_{l=1}^2\left[ a^{\ts}_{l}a^{\ts}_{l} \nabla^2z^{\ts}_{j}\mathsf{U} z^{\ts}_{j} \mathsf{U} \right] \nonumber\\
&&-2\sum_{j=1}^4\sum_{l=1}^2 \left[ (a^{\ts}_{l}\nabla z^{\ts}_j \mathsf{U})^2  +(a^{\ts}\nabla\log\pi)_l \Big[ a^{\ts}_{l}\nabla z^{\ts}_{j}\mathsf{U} z^{\ts}_{j}\mathsf{U}\Big]  \right], \nonumber\\
&=&-2 |a^{\ts}_{14}|^2\pa_{p_0p_0}z^{\ts}_{45}\mathsf U_5 z^{\ts}_4\mathsf U-2 |a^{\ts}_{26}|^2\pa_{p_2p_2}z^{\ts}_{45}\mathsf U_5 z^{\ts}_4\mathsf U\\
&&-2(a^{\ts}_{14}\pa_{p_0}z^{\ts}_{45}\mathsf U_5)^2-2(a^{\ts}_{26}\pa_{p_2}z^{\ts}_{45}\mathsf U_5)^2\\
&&-2a^{\ts}_{14}\pa_{p_0}\log \pi a^{\ts}_{14}\pa_{p_0}z^{\ts}_{45}\mathsf U_5z^{\ts}_4\mathsf U-2a^{\ts}_{26}\pa_{p_2}\log \pi a^{\ts}_{26}\pa_{p_2}z^{\ts}_{45}\mathsf U_5z^{\ts}_4\mathsf U\\
&=&-2\xi T [\pa_{p_0p_0}z^{\ts}_{45} +\pa_{p_2p_2}z^{\ts}_{45}+\pa_{p_0}\log\pi \pa_{p_0}z^{\ts}_{45}+\pa_{p_2}\log\pi\pa_{p_2}z^{\ts}_{45}  ]\mathsf U_5z^{\ts}_4 \mathsf U\\
&&-2(a^{\ts}_{14}\pa_{p_0}z^{\ts}_{45}\mathsf U_5)^2-2(a^{\ts}_{26}\pa_{p_2}z^{\ts}_{45}\mathsf U_5)^2\\
&=& -2 \mathsf S_1 \mathsf U_5 z^{\ts}_4 \mathsf U -2\xi T(\pa_{p_0}z^{\ts}_{45}\mathsf U_5)^2-2\xi T  (\pa_{p_2}z^{\ts}_{45}\mathsf U_5)^2\\
&=&-\mathsf U^{\ts}\Big[\begin{pmatrix}
	0\\
	0\\
	0\\
	0\\
	 \mathsf S_1\\
	0
\end{pmatrix} z^{\ts}_4 +z_4 \begin{pmatrix}
	0&0&0&0& \mathsf S_1 &0
\end{pmatrix} \Big] \mathsf U\\
&&-\mathsf U^{\ts}\begin{pmatrix}
	0&0&0&0&0&0\\
	0&0&0&0&0&0\\
	0&0&0&0&0&0\\
	0&0&0&0&0&0\\
	0&0&0&0&2\xi T [|\pa_{p_0}z^{\ts}_{45}|^2+|\pa_{p_2}z^{\ts}_{45}|^2 ] &0\\
	0&0&0&0&0&0
\end{pmatrix} \mathsf U,
\eeaa 
which directly gives
\beaa
\mathfrak R_{\pi}= \begin{pmatrix}
	0&0\\
	0&\mathsf I_{\pi}
\end{pmatrix}_{6\times 6},\q\text{with}\q  \mathsf I_{\pi}=\begin{pmatrix}
	0&0&0\\
	0&-2 z_{32}\mathsf S_1- 2\xi T [|\pa_{v_0}z^{\ts}_{45}|^2+|\pa_{v_2}z^{\ts}_{45}|^2 ]&0\\
	0&0&0
\end{pmatrix}.
\eeaa 
Summing over all above matrices, we finish the proof.
\qed
\end{proof}
\end{document}